\documentclass[11pt]{article}
\usepackage[a4paper,margin=1.02in]{geometry}
\usepackage{amsmath,amssymb,amsthm,mathtools,mathrsfs}
\usepackage{enumitem,booktabs,tabularx}
\usepackage[colorlinks=true,linkcolor=blue,citecolor=blue,urlcolor=blue]{hyperref}
\newtheorem{theorem}{Theorem}[section]
\newtheorem{proposition}[theorem]{Proposition}
\newtheorem{lemma}[theorem]{Lemma}
\newtheorem{corollary}[theorem]{Corollary}
\theoremstyle{definition}
\newtheorem{definition}[theorem]{Definition}
\theoremstyle{remark}
\newtheorem{remark}[theorem]{Remark}
\numberwithin{equation}{section}
\newcommand{\R}{\mathbb R}
\newcommand{\dd}{\,\mathrm d}
\newcommand{\Dt}{D_t}
\newcommand{\divg}{\operatorname{div}}

\newcommand{\bu}{\mathbf u}
\newcommand{\bv}{\mathbf v}
\title{Global regularity for 3D large spherically symmetric degenerate compressible Navier-Stokes equations with superlinear BD viscosity and far-field vacuum}
\author{%
	Xiangdi Huang\textsuperscript{a,}\thanks{Corresponding author.
		Email: \texttt{xdhuang@amss.ac.cn} (X.~Huang).},
	Bingshu Li\textsuperscript{b,}\thanks{Email:
		\texttt{bingshuli.math@gmail.com} (B.~Li).}\\[0.8em]
	\parbox{0.95\textwidth}{\centering\normalsize
		\textsuperscript{a} State Key Laboratory of Mathematical Sciences,
		Academy of Mathematics and Systems Science,\\
		Chinese Academy of Sciences, Beijing 100190, China;\\[0.4em]
		\textsuperscript{b} School of Mathematics, Jilin University,\\
		Changchun 130012, China.
	}%
}
\date{}
\begin{document}
	\maketitle
	\begin{abstract}
		In recent years, important progress has been made regarding whether smooth solutions of the isentropic compressible Navier-Stokes equations with general initial data develop singularities in finite tim-such as vacuum or implosion. The work of Merle et al. revealed that smooth solutions of the spherically symmetric isentropic compressible Navier-Stokes equations with constant viscosity coefficients can develop implosion at the center of the sphere under the evolution of certain initial data. However, for the compressible Navier-Stokes equations with variable viscosity coefficients, especially when the coefficients satisfy the BD structural condition $\mu(\rho)=\rho^{\alpha}$ ,$\lambda=(\alpha-1)\rho^{\alpha}$  and $\alpha\le 1$, the situation becomes very different. Huang-Meng-Zhang proved that stong solutions of the equations in a 3D bounded ball with arbitrarily large initial data do not blow up, and meanwhile Chen-Zhang-Zhu independently proved the global existence of strong solutions for the Cauchy problem in the far-field vacuum case. The methods in both papers rely fully on the viscosity coefficients being linear or sublinear in density. When $\alpha>1$, the situation becomes much more complicated. Superlinear density-dependent viscosity presents a fundamental obstacle of global regularity to the compressible Navier-Stokes equations in the presence of far-field vacuum: the viscous diffusion per unit mass vanishes at infinity, so excluding vacuum on bounded sets does not suffice to control the flow globally. In this paper, for $\alpha\in (1,11.7)$ and $\gamma\in [2\alpha-1,3\alpha-1)$ , we establish global regularity of 3D spherically symmetric degenerate compressible Navier-Stokes equations for arbitrarily large, strong, finite-mass initial data with positive density at every finite point and vacuum at infinity, subject to explicit Sobolev and logarithmic derivative assumptions. One of the main difficulty lies in deriving a local lower bound for the density, which requires handling singular estimates at the spherical center. To overcome this, we introduce some new rescaled variables regarding the velocity, effective velocity and mass. Moreover, we  further exploit several novel quantities near the center and employ a refined maximum principle to establish the Lipchitz continuity of the velocity at the center-thereby resolving the singularity of the velocity gradient. Another key mechanism is a quantitative control of material volume distortion that remains effective in the vacuum tail. Conservative estimates for the physical and effective velocity gradients yield uniformly vanishing exterior variation tails. These estimates identify a material trajectory with small cumulative Jacobian variation and lead to two-sided bounds for the exterior volume Jacobian, making the evolving density comparable to its initial value along particle paths. On spatial scales determined by the initial density, the resulting degenerate equations admit uniform parabolic estimates that recover control of the physical deformation without a global density lower bound. Combined with compact positivity estimates and a one-sided Glaeser inequality, this closes a homogeneous continuation argument without imposing $L^2$-integrability on the density powers themselves. The solutions conserve mass, preserve far-field vacuum, and remain strong and strictly positive on every compact spacetime cylinder.
		
	\end{abstract}
	
	

	\setcounter{tocdepth}{2}
	\tableofcontents
	\bigskip
	
	\section{Introduction and main theorem}\label{sec:intro}
	We study the global regularity of spherically symmetric solutions to the
	three-dimensional isentropic compressible Navier--Stokes equations with
	superlinear density-dependent viscosities and vacuum at infinity. The
	initial density is positive at every finite point and has finite total
	mass. Our aim is to propagate this positivity and control the regularity
	of the flow for arbitrarily large data, despite the degeneration of
	viscous diffusion in the far field. The equations are, for
	$(t,x)\in[0,\infty)\times\R^3$,
	\begin{equation}\label{eq:system}
		\begin{cases}
			\rho_t+\divg(\rho\bu)=0,\\
			(\rho\bu)_t+\divg(\rho\bu\otimes\bu)+\nabla\rho^\gamma\\
			\qquad=\divg\!\left[\dfrac{\rho^\alpha}{2}
			\bigl(\nabla\bu+(\nabla\bu)^{\mathsf T}\bigr)
			+(\alpha-1)\rho^\alpha(\divg\bu)I\right].
		\end{cases}
	\end{equation}
	Here $\rho$ and $\bu$ denote the density and velocity, respectively, the
	pressure is $P(\rho)=\rho^\gamma$, and $I$ is the identity matrix.
	We use the symmetric-gradient stress convention displayed in
	\eqref{eq:system}, with
	\[
	\mu(\rho)=\rho^\alpha,\qquad
	\lambda(\rho)=(\alpha-1)\rho^\alpha.
	\]
	These coefficients satisfy the Bresch--Desjardins compatibility relation.
	The Cauchy data and far-field condition are
	\[
	(\rho,\bu)|_{t=0}=(\rho_0,\bu_0),\qquad
	(\rho,\bu)(t,x)\longrightarrow(0,0)\quad\text{as }|x|\to\infty.
	\]
	We work with spherical symmetry:
	\[
	\rho(t,x)=\rho(t,r),\qquad \bu(t,x)=u(t,r)\frac{x}{r},\qquad r=|x|.
	\]
	Smoothness at the origin always means smoothness of these Cartesian fields.
	
	The theory of compressible viscous flow provides several distinct
	approaches to global solvability. For constant viscosities,
	Matsumura--Nishida \cite{MatsumuraNishida1980} established global smooth
	solutions for small perturbations of a positive equilibrium. The
	finite-energy weak theories of Lions \cite{Lions1998} and
	Feireisl--Novotn\'y--Petzeltov\'a
	\cite{FeireislNovotnyPetzeltova2001} allow large data and vacuum.
	These theories provide a broad existence framework, while the
	propagation of smoothness and strict positivity for large data
	requires additional methods.
	
	For density-dependent viscosities, a central structural ingredient
	is the additional entropy developed by Bresch--Desjardins
	\cite{BreschDesjardins2003} and Bresch--Desjardins--Lin
	\cite{BreschDesjardinsLin2003}. The compatibility between the shear
	and bulk viscosities yields a modified velocity and additional
	control of density gradients. Mellet--Vasseur
	\cite{MelletVasseur2007} established stability estimates for weak
	solutions, and Guo--Jiu--Xin \cite{GuoJiuXin2008} constructed global
	weak solutions under spherical symmetry. Global weak existence for
	multidimensional degenerate systems was obtained by Li--Xin
	\cite{LiXin2015} and Vasseur--Yu \cite{VasseurYu2016}.
	Bresch--Vasseur--Yu \cite{BreschVasseurYu2022} subsequently treated
	the physical symmetric stress with a broad class of nonlinear
	viscosity laws, including power laws on both sides of the exponent
	one. These weak solutions may contain vacuum; their existence does
	not by itself give the pointwise density bounds and derivative
	control needed to continue a smooth solution.
	
	Effective quantities also play an important role in the
	one-dimensional strong theory. Constantin--Drivas--Nguyen--Pasqualotto
	\cite{ConstantinEtAl2020} studied the evolution of an active
	potential combining viscous stress and pressure, and used its
	equation to obtain regularity and global existence results.
	Burtea--Haspot \cite{BurteaHaspot2020} introduced an effective
	pressure that yields an Oleinik-type estimate for the effective
	velocity and global strong solutions for general viscosity laws
	under their hypotheses. These methods motivate the use of
	one-sided estimates for effective quantities. In spherical
	geometry, their implementation must also account for the center
	and the unbounded region where the density approaches zero.
	
	For the multidimensional Cauchy problem with vacuum, Li--Pan--Zhu
	\cite{LiPanZhu2019} established local well-posedness of regular
	solutions for a class of power-law viscosities and arbitrarily
	large initial data. Xin--Zhu \cite{XinZhu2021} obtained global
	regular solutions with far-field vacuum for small initial density
	and an expanding initial velocity, allowing large velocity
	amplitudes. Huang--Li--Zhang \cite{HuangLiZhang2026} established
	global strong solutions when the initial density is comparable
	to a sufficiently large positive constant.
	The expansion and positive-background assumptions in these
	global results are distinct from the finite-mass radial setting
	considered here. The local existence and continuation statement
	used in the present paper is specified separately in
	Proposition~\ref{prop:local}.
	
	Recent work has substantially developed the radial strong theory.
	On a bounded domain with positive initial density, Zhang
	\cite{Zhang2025} proved global strong solvability for
	$\frac{7}{8}\le\alpha<1$ and $1<\gamma<9\alpha-6$ in three
	dimensions, together with a two-dimensional result.
	Huang--Meng--Zhang \cite{HMZ} enlarged the sublinear ranges,
	treated the two-dimensional linear-viscosity endpoint, and
	constructed weak solutions allowing vacuum in additional
	parameter ranges. Their radial moment estimates are a starting
	point for the density upper bound in this paper.
	Lei \cite{Lei2026} further obtained three-dimensional global
	classical solutions for $0.67661<\alpha<1$ and
	$1<\gamma<6\alpha-3$, both on bounded domains and in the whole
	space with a non-vacuum far-field density. In transferring radial
	estimates to the finite-mass Cauchy problem, the pressure terms
	in the exterior and the loss of a global density lower bound
	require separate treatment.
	
	Chen--Zhang--Zhu studied global regularity for the Cauchy problem
	with vacuum at infinity. For viscosities proportional to the density,
	they established global regular solutions with large radial
	data in two and three dimensions \cite{CZZ}; the
	three-dimensional pressure range is $1<\gamma<3$.
	Their logarithmic density variables and treatment of the origin
	are particularly relevant to the initial regularity assumptions
	used here. In a subsequent work \cite{ChenZhangZhu2026}, they
	treated nonlinear sublinear viscosities with positive initial
	density vanishing in the far field. In the notation of this
	paper, their exponent ranges are
	\[
	\begin{aligned}
		N=2:\quad &\frac{1}{2}<\alpha<1,\qquad \gamma>1,\\
		N=3:\quad &7-2\sqrt{10}<\alpha<1,\qquad 1<\gamma<6\alpha-3,
	\end{aligned}
	\]
	subject to the coefficient and initial regularity assumptions
	in that work. The superlinear case considered below changes the
	behavior of the diffusion per unit mass at vacuum and therefore
	requires a different mechanism for controlling the far field.
	
	Singularity formation provides a complementary perspective.
	Merle--Rapha\"el--Rodnianski--Szeftel \cite{MerleEtAl2022}
	constructed finite-time implosion for compressible flow with
	constant viscosities. Chen--Liu--Zhu \cite{ChenLiuZhu2026}
	constructed implosion for nonlinear density-dependent
	viscosities in specified sublinear parameter regimes.
	These results concern concentration of the density at the
	origin and show that viscosity structure and parameter
	restrictions matter even in the study of smooth radial
	dynamics. The present problem is to preserve regularity for
	the superlinear BD law \eqref{eq:system}, with finite mass,
	positive density at every finite point, and vacuum at infinity.
	
	Two difficulties must be addressed in this setting. First, for
	$\alpha>1$, the viscous diffusion per unit mass is proportional to
	\[
	\frac{\mu(\rho)}{\rho}=\rho^{\alpha-1},
	\]
	which tends to zero with the density. Since the far-field state is
	vacuum, the momentum equation is not uniformly parabolic on $\R^3$,
	even at the initial time. A positive lower bound on each fixed ball
	therefore leaves open the question of derivative control in the
	unbounded exterior. Second, spherical symmetry introduces the
	terms $\frac{u}{r}$ and $u_r-\frac{u}{r}$, so estimates near the origin must respect
	the smooth Cartesian extension.
	
	The effective velocity associated with the viscosity law is
	\begin{equation}\label{eq:effective}
		\bv=\bu+\frac{\alpha}{\alpha-1}\nabla {\rho^{\alpha-1}},\qquad
		v=u+\alpha\rho^{\alpha-2}\rho_r.
	\end{equation}
	Here $v$ is the radial component of $\bv$, with
	$\bv(t,x)=v(t,r)\frac{x}{r}$ for $r>0$; the second identity
	is the radial form of the first.
	It links the two velocities to the density through
	\[
	\partial_r\rho^{\alpha-1}=\frac{\alpha-1}{\alpha}(v-u).
	\]
	For $\alpha>1$, the power $\rho^{\alpha-1}$ remains finite as
	$\rho\downarrow0$. Consequently, control of its spatial variation
	alone does not prevent vacuum formation at a finite point.
	The lower density estimate must use the evolution of the flow as well
	as the relation between $u$, $v$, and $\rho$.
	
	Our main result establishes global smooth radial solutions in the
	range $1<\alpha<11.7$ and $2\alpha-1\le\gamma<3\alpha-1$ under the
	initial regularity assumptions stated below. The data may be
	arbitrarily large, and no prescribed density decay profile or monotonicity is required. On every finite time interval the density has a global upper bound and a positive
	lower bound on each fixed ball, while the solution preserves its vacuum
	state at infinity. The proof combines estimates near the center with
	control of volume distortion along exterior particle trajectories.
	
	We formulate the result in a homogeneous class for the density
	powers. Finite mass does not imply square integrability of
	$\rho^{\frac{\alpha-1}{2}}$ or $\rho^{\frac{\gamma-1}{2}}$ throughout
	the exponent range considered here. The class below instead specifies
	their boundedness and derivative regularity, together with the
	regularity of their differences from the initial powers.
	All unqualified Sobolev norms below are over $\R^3$.
	
	\begin{definition}[Homogeneous power solution class]\label{def:class}
		On $[0,T]$, the class $\mathcal X_T$ consists of radial solutions with
		positive density at every finite point, for which
		\begin{equation}\label{eq:class}
			\begin{gathered}
				\bu\in L^\infty(0,T;H^3)\cap C([0,T];H^s),\quad \frac{5}{2}<s<3,\\
				\rho^{\frac{\alpha-1}{2}},\rho^{\frac{\gamma-1}{2}}\in L^\infty((0,T)\times\R^3),\\
				\nabla^j\rho^{\frac{\alpha-1}{2}},\nabla^j\rho^{\frac{\gamma-1}{2}}\in L^\infty(0,T;L^2),\quad j=2,3,\\
				\nabla\rho^{\gamma-1}\in C([0,T];L^2),\quad
				\rho^{\frac{\alpha-1}{2}}\nabla^j\bu\in L^2((0,T)\times\R^3),\quad 1\le j\le4,\\
				\partial_t\rho^{\frac{\alpha-1}{2}},\partial_t\rho^{\frac{\gamma-1}{2}}\in L^\infty(0,T;H^2),\quad
				\bu_t\in L^\infty(0,T;H^1),\\
				\rho^{\frac{\alpha-1}{2}}-\rho_0^{\frac{\alpha-1}{2}},\rho^{\frac{\gamma-1}{2}}-\rho_0^{\frac{\gamma-1}{2}}\in C([0,T];H^s),\quad 0\le s<3.
			\end{gathered}
		\end{equation}
		The powers $\rho^{\frac{\alpha-1}{2}},\rho^{\frac{\gamma-1}{2}}$ need not belong to $L^2$.
	\end{definition}
	
	\begin{theorem}[Global strong radial Cauchy flow]\label{thm:main}
		Assume
		\begin{equation}\label{eq:range}
			1<\alpha<11.7,\qquad 2\alpha-1\le\gamma<3\alpha-1.
		\end{equation}
		Let $(\rho_0,\bu_0)$ be radial, strong on every finite ball with smooth
		Cartesian center extension, and satisfy
		\begin{equation}\label{eq:data}
			0<\rho_0\in L^1\cap L^\infty,\quad \rho_0(x)\to0\ (|x|\to\infty),
			\qquad \bu_0\in H^3,
		\end{equation}
		\begin{equation}\label{eq:logdata}
			\nabla\log\rho_0\in L^\infty,\qquad
			\nabla^2\log\rho_0,\nabla^3\log\rho_0\in L^2,
		\end{equation}
		\begin{equation}\label{eq:forcedata}
			\nabla\rho_0^{\gamma-1}\in H^2.
		\end{equation}
		Then \eqref{eq:system} has a unique global radial solution belonging to
		$\mathcal X_T$ for every finite $T$. The solution is strong on every finite
		spacetime cylinder, including at the center, and
		\begin{equation}\label{eq:massconclusion}
			\rho\in C([0,T];L^1),\qquad
			\int_{\R^3}\rho(t,x)\dd x=\int_{\R^3}\rho_0(x)\dd x.
		\end{equation}
		For every finite $T,R$ there are constants $C_T<\infty$ and $c_{T,R}>0$
		such that
		\begin{equation}\label{eq:densityconclusion}
			\sup_{t\le T,x\in\R^3}\rho(t,x)\le C_T,\qquad
			\inf_{t\le T,\,|x|\le R}\rho(t,x)\ge c_{T,R}.
		\end{equation}
		Moreover $\rho,\bu,\bv\to0$ as $|x|\to\infty$, uniformly for $t\in[0,T]$.
		Uniqueness is among radial solutions in the stated strong class with the
		same original data. There is no finite-time vacuum at a finite point.
	\end{theorem}
	
	\begin{remark}\label{rem:scope}
		The initial data can be arbitrarily large and the constants above may depend on the actual data and on $T,R$; no
		positive global density lower bound is asserted. The logarithmic norms
		in \eqref{eq:logdata} are initial assumptions. Their propagation as global
		norms is not part of the theorem:
		\[
		\nabla\log\rho=\frac{2}{\alpha-1}
		\frac{\nabla\rho^{\frac{\alpha-1}{2}}}{\rho^{\frac{\alpha-1}{2}}},
		\]
		and $\rho^{\frac{\alpha-1}{2}}\to0$ at infinity. Smoothness on finite cylinders and the norms
		\eqref{eq:class} are the conclusions used below.
	\end{remark}
	
	We now describe the main steps of the argument. The ordinary and
	effective energy estimates, together with higher radial moments of
	the velocities, give a global upper bound for the density. Coupled
	comparison estimates for $u$ and $v$, followed by a material estimate
	for the reciprocal density, then give a positive lower bound on every
	fixed ball. These estimates also yield local regularity and control
	of the time integral of the squared velocity gradient on bounded
	regions. The remaining task is to obtain a bound uniform across the
	vacuum tail.
	
	For this purpose, let $r(t,a)$ be the particle trajectory starting
	at radius $a>0$:
	\[
	\partial_t r(t,a)=u(t,r(t,a)),\qquad r(0,a)=a.
	\]
	The ratio of the current volume element to its initial value is
	\[
	J(t,a)=\frac{r(t,a)^2\partial_a r(t,a)}{a^2}.
	\]
	The continuity equation gives
	\[
	\rho(t,r(t,a))J(t,a)=\rho_0(a),\qquad
	\partial_t J=J\left(u_r+\frac{2u}{r}\right)(t,r(t,a)).
	\]
	Thus bounds for $J$ from above and away from zero compare the evolving
	density with its initial value along each trajectory. The challenge
	is to obtain these bounds before the exterior velocity gradient is
	controlled.
	
	We derive conservative equations for the radial derivatives of the
	physical and effective velocities. Their variation estimates give
	\[
	\lim_{R\to\infty}\sup_{0\le t<T}
	\int_R^\infty\bigl(|u_r(t,r)|+|v_r(t,r)|\bigr)\dd r=0
	\]
	on a maximal interval $[0,T)$, assuming $T<\infty$. This control selects an
	exterior material trajectory on which the cumulative variation of
	$J$ is small. Combining the resulting boundary control with the
	Jacobian equation and a normalized effective strain yields
	\[
	C_T^{-1}\le J(t,a)\le C_T
	\]
	for all sufficiently large initial radii, uniformly for $t<T$.
	In particular, the comparison with $\rho_0$ is obtained without
	assuming a decay rate or a global positive lower bound for the density.
	
	The comparison determines the spatial scale for the exterior
	regularity estimates. Around a large initial radius $a_c$, we use
	\[
	\ell=\rho_0(a_c)^{\frac{\alpha-1}{2}},\qquad
	a=a_c+\ell y,
	\]
	with time unchanged. On these scales the equation for $J^{-\alpha}$
	has a diffusion coefficient bounded above and away from zero,
	uniformly as $a_c\to\infty$ and $\ell\to0$. Interior parabolic
	estimates then control its time derivative and recover the physical
	strain through the continuity equation. Together with the estimates
	on fixed balls, this gives
	\[
	\int_0^T\|\nabla\bu(t)\|_\infty^2\dd t<\infty.
	\]
	A one-sided Glaeser inequality applied to $\rho^{\alpha-1}$ supplies
	the remaining bounds for the gradients of the density powers.

	The paper is organized as follows. Section~\ref{sec:prelim} records
	the radial equations, the local existence and continuation result,
	and the parabolic estimates used below. Section~\ref{sec:upper}
	proves the density upper bound. Sections~\ref{sec:barriers} and
	\ref{sec:core} establish local positivity, local regularity, and
	global velocity bounds. Section~\ref{sec:variation} gives the
	exterior variation estimates, and Section~\ref{sec:jacobian} proves
	the two-sided Jacobian bounds. Section~\ref{sec:parabolic} obtains
	the exterior strain estimate on the scales determined by the initial
	density. Finally, Section~\ref{sec:completion} controls the
	density-power gradients and completes the continuation argument.
	
	\section{Preliminaries}\label{sec:prelim}
	We first record the radial equations and introduce the quantities used
	in the a priori estimates. We then state the local existence and
	continuation result and the two parabolic estimates used in
	Section~\ref{sec:parabolic}.
	
	For $t>0$ and $r>0$, the radial form of \eqref{eq:system} is
	\begin{equation}\label{eq:radial-cauchy}
		\begin{cases}
			\rho_t+(\rho u)_r+\dfrac{2}{r}\rho u=0,\\[2mm]
			(\rho u)_t+(\rho u^2)_r+\dfrac{2}{r}\rho u^2+(\rho^\gamma)_r\\
			\qquad=\left[\alpha\rho^\alpha
			\left(u_r+\dfrac{2u}{r}\right)\right]_r
			-\dfrac{2u}{r}(\rho^\alpha)_r.
		\end{cases}
	\end{equation}
	The initial, center, and far-field conditions become
	\[
	\begin{gathered}
		(\rho,u)(0,r)=(\rho_0,u_0)(r),\qquad r\ge0,\\
		\rho_r(t,0)=0,\qquad u(t,0)=0,\qquad
		(\rho,u)(t,r)\longrightarrow(0,0)\quad(r\to\infty).
	\end{gathered}
	\]
	The center conditions are understood together with the smooth Cartesian
	extension: $\rho$ is even and $u$ is odd in the radial variable near zero.
	All quotients at $r=0$ are interpreted by these smooth limits.
	
	The material derivative and the divergence of the physical velocity are
	denoted by
	\[
	\Dt=\partial_t+u\partial_r,\qquad
	d=\divg\bu=u_r+\frac{2u}{r}.
	\]
	We also use the transverse strain rates of the physical and effective
	velocities,
	\[
	U=\frac{u}{r},\qquad V=\frac{v}{r}.
	\]
	For $r>0$, these are the tangential eigenvalues of $\nabla\bu$ and
	$\nabla\bv$, respectively.
	Since $u$ and $v$ are smooth and odd near the center, these quotients
	extend smoothly to $r=0$, with $U(t,0)=u_r(t,0)$ and
	$V(t,0)=v_r(t,0)$. They represent the velocity slopes at the center.
	
	The pressure coefficient coupling $u$ and $v$ is
	\[
	F_1=\frac{\gamma}{\alpha}\rho^{\gamma-\alpha}.
	\]
	In this notation, continuity and the physical and effective velocity
	equations take the following material forms on a positive smooth radial
	solution:
	\begin{equation}\label{eq:radial}
		\begin{gathered}
			\Dt \rho=-\rho d,\qquad
			\rho \Dt u+(\rho^\gamma)_r-(\alpha\rho^\alpha d)_r
			+\frac{2}{r}(\rho^\alpha)_r u=0,\\
			\Dt v=F_1(u-v),\qquad
			\partial_r\rho^{\alpha-1}=\frac{\alpha-1}{\alpha}(v-u),\\
			\begin{aligned}
				\Dt u={}&\alpha\rho^{\alpha-1}
				\left(u_{rr}+\frac{2u_r}{r}-\frac{2u}{r^2}\right)\\
				&+(v-u)\left(\alpha u_r+\frac{2(\alpha-1)u}{r}-F_1\right).
			\end{aligned}
		\end{gathered}
	\end{equation}
	For the differentiated effective velocity equation, we write
	\begin{equation}\label{eq:notation}
		Q=(\partial_r F_1)(v-u)
		=\frac{\gamma(\gamma-\alpha)}{\alpha^2}
		\rho^{\gamma-2\alpha+1}(v-u)^2\ge0.
	\end{equation}
	Its sign follows from $\gamma-\alpha\ge\alpha-1>0$ in
	\eqref{eq:range}.
	
	We suppress the factor $4\pi$ in radial integrals. The enclosed mass
	and the initial total radial mass are
	\[
	m(t,r)=\int_0^r\rho(t,s)s^2\dd s,\qquad
	M=\int_0^\infty\rho_0(r)r^2\dd r.
	\]
	The mass coordinate satisfies
	\[
	m_r=\rho r^2,\qquad m_t=-\rho u r^2,\qquad \Dt m=0.
	\]
	Since $\rho>0$, it has an inverse $r=r(t,m)$ for $r>0$, and
	\[
	\left.\partial_t r\right|_m=u.
	\]
	Thus differentiation at fixed mass label agrees with material
	differentiation. We use these identities below without further proof.
	
	We use the following local existence and continuation result in the
	homogeneous solution class without proof. To specify the admissible
	initial data, define
	\begin{equation}\label{eq:hom:4.2}
		\begin{aligned}
			\mathcal E={}&1+\|\bu\|_{H^3}^2+\|\nabla\rho^{\gamma-1}\|_2^2\\
			&+\sum_{j=2}^3\left(
			\|\nabla^j\rho^{\frac{\alpha-1}{2}}\|_2^2
			+\frac{4\gamma}{(\gamma-1)^2}
			\|\nabla^j\rho^{\frac{\gamma-1}{2}}\|_2^2\right),
		\end{aligned}
	\end{equation}
	This energy contains no $L^2$ norm of either density power or its first
	derivative.
	
	\begin{proposition}[Local existence and continuation in the homogeneous solution class]\label{prop:local}
		The data of Theorem~\ref{thm:main} generate a positive radial solution
		in $\mathcal X_{t_*}$ for some $t_*>0$, smooth on finite cylinders.
		It satisfies $\rho\in C([0,t_*];L^1)$, conserves mass, and has
		$\rho,\bu,\bv\to0$ at infinity uniformly on this interval.
		Uniqueness holds among radial solutions in the stated class issued from
		the original data.
		
		More generally, a radial solution in the same homogeneous class can be
		started from positive density at every finite point with
		$\rho\in L^1$, bounded density powers
		$\rho^{\frac{\alpha-1}{2}}$ and $\rho^{\frac{\gamma-1}{2}}$ tending to
		zero at infinity, and $\mathcal E<\infty$. The initial derivatives are
		understood weakly. The existence time depends only on $\alpha,\gamma$,
		the amplitudes of these powers, and $\mathcal E$, without requiring
		a positive global density lower bound or finite $L^2$ norms of the
		two powers or their first derivatives.
		
		Suppose a solution issued from the original data is defined on $[0,T)$,
		where $T<\infty$, and belongs to $\mathcal X_\tau$ for every $\tau<T$.
		If
		\begin{equation}\label{eq:continuationcriterion}
			\begin{gathered}
				\sup_{0\le t<T}\|\rho(t)\|_\infty<\infty,\\
				\int_0^T\left(
				1+\|\nabla\bu(t)\|_\infty^2
				+\|\nabla\rho^{\frac{\alpha-1}{2}}(t)\|_\infty^2
				+\|\nabla\rho^{\frac{\gamma-1}{2}}(t)\|_\infty
				\right)\dd t<\infty,
			\end{gathered}
		\end{equation}
		then it extends to a solution in $\mathcal X_{T+\delta}$ for some
		$\delta>0$. The extension preserves smoothness on finite spacetime
		cylinders, mass conservation, continuity of $\rho$ in $L^1$, and uniform
		far-field decay on finite time intervals. No propagation of the initial
		logarithmic derivative norms is required for this continuation.
	\end{proposition}
	
	We next state the required forms of Lieberman's interior estimates without
	proof. Both concern one spatial variable and have no zeroth-order
	coefficient. For $s_0<s_1<s_2$, $y_0\in\mathbb R$, and $0<r<R$, write
	\[
	\begin{aligned}
		\Omega&=(s_0,s_2)\times(y_0-R,y_0+R),\\
		\Omega'&=(s_1,s_2)\times(y_0-r,y_0+r).
	\end{aligned}
	\]
	The parabolic boundary of $\Omega$ consists of the bottom $t=s_0$
	and the lateral sides $y=y_0\pm R$. Its distance from $\Omega'$ is
	\[
	\min\left\{(s_1-s_0)^{\frac{1}{2}},\,R-r\right\}>0.
	\]
	The top $t=s_2$ is excluded from this boundary. In particular, the
	estimates below impose no condition at $t=s_2$.
	For $0<\beta<1$, we use the parabolic H\"older norm
	\[
	\|g\|_{C^{\beta,\frac{\beta}{2}}(\overline\Omega)}
	=\|g\|_{L^\infty(\Omega)}
	+\sup_{\substack{(t,y),(s,z)\in\overline\Omega\\(t,y)\ne(s,z)}}
	\frac{|g(t,y)-g(s,z)|}{|y-z|^\beta+|t-s|^{\frac{\beta}{2}}}.
	\]
	The norm on $\overline{\Omega'}$ is defined in the same way.
	
	The first statement is the one-dimensional specialization of
	Lieberman~\cite[Theorem~7.22, estimate~(7.39), pp.~175--177]{Lieberman}
	to bounded drift and a H\"older continuous leading coefficient. We use
	the forward equation; in the book's notation its coefficients are
	$a^{11}=\mathfrak a$, $b^1=\mathfrak b$, $c=0$, and its right-hand
	side is $-\mathfrak f$.
	\begin{theorem}[Interior $L^p$ estimate in one space dimension]
		\label{thm:interior-lp}
		Let $2\le p<\infty$, $0<\beta<1$, and $0<\lambda\le\Lambda$.
		Suppose that
		\[
		\lambda\le\mathfrak a\le\Lambda,\qquad
		\|\mathfrak a\|_{C^{\beta,\frac{\beta}{2}}(\overline\Omega)}
		+\|\mathfrak b\|_{L^\infty(\Omega)}\le K,
		\]
		where $\mathfrak b$ may be merely measurable. If
		$w\in W^{2,1}_p(\Omega)$ and $\mathfrak f\in L^p(\Omega)$ satisfy
		\[
		w_t-\mathfrak a w_{yy}-\mathfrak b w_y=\mathfrak f
		\quad\text{almost everywhere in }\Omega,
		\]
		then
		\[
		\|w_t\|_{L^p(\Omega')}+\|w_{yy}\|_{L^p(\Omega')}
		\le C\left(\|\mathfrak f\|_{L^p(\Omega)}
		+\|w\|_{L^p(\Omega)}\right).
		\]
		The constant $C$ depends only on $p,\beta,\lambda,\Lambda,K,R$ and
		the positive gaps $R-r$ and $s_1-s_0$. It is independent of $y_0$
		and $s_2$. No boundary values of $w$ are prescribed.
	\end{theorem}
	
	When the drift and source are also H\"older continuous, we use the
	following one-dimensional form of
	Lieberman~\cite[Theorem~4.9, estimate~(4.27), pp.~59--60]{Lieberman}.
	It is stated as an a priori estimate for a classical solution.
	\begin{theorem}[Interior Schauder estimate in one space dimension]
		\label{thm:interior-schauder}
		Let $0<\beta<1$ and $0<\lambda\le\Lambda$. Suppose that
		\[
		\lambda\le\mathfrak a\le\Lambda,\qquad
		\|\mathfrak a\|_{C^{\beta,\frac{\beta}{2}}(\overline\Omega)}
		+\|\mathfrak b\|_{C^{\beta,\frac{\beta}{2}}(\overline\Omega)}\le K.
		\]
		If $w\in C^{2+\beta,1+\frac{\beta}{2}}(\overline\Omega)$ and
		$\mathfrak f\in C^{\beta,\frac{\beta}{2}}(\overline\Omega)$ satisfy
		\[
		w_t-\mathfrak a w_{yy}-\mathfrak b w_y=\mathfrak f
		\quad\text{in }\Omega,
		\]
		then
		\[
		\begin{aligned}
			&\|w_t\|_{C^{\beta,\frac{\beta}{2}}(\overline{\Omega'})}
			+\|w_{yy}\|_{C^{\beta,\frac{\beta}{2}}(\overline{\Omega'})}\\
			&\qquad\le C\left(\|w\|_{L^\infty(\Omega)}
			+\|\mathfrak f\|_{C^{\beta,\frac{\beta}{2}}(\overline\Omega)}\right).
		\end{aligned}
		\]
		The constant $C$ depends only on $\beta,\lambda,\Lambda,K,R$ and
		the positive gaps $R-r$ and $s_1-s_0$. It is independent of $y_0$
		and $s_2$. No boundary values of $w$ are prescribed.
	\end{theorem}
	
	In both statements, the derivative norms assumed finite on a fixed
	domain do not enter the constant in the estimate. This allows their
	use on each interval $[0,\tau]$, $\tau<T$, on which the solution is
	smooth, with constants independent of $\tau$ as $\tau\uparrow T$.
	
	\section{Energy estimates and the density upper bound}\label{sec:upper}
	We first derive mass conservation and the ordinary and effective energy
	identities. These give radial density estimates, which allow us to control
	higher moments of both velocities. The effective-velocity relation then
	converts these moment bounds into the density upper bound.
	Throughout this section, $C_T$ denotes a finite positive constant depending
	only on $T$, $\alpha,\gamma$, the fixed moment exponents, and the initial
	data; its value may increase from one estimate to the next. Constants
	denoted by $C_\tau$ are used only to justify calculations on a fixed
	time interval $[0,\tau]$, $\tau<T$, and need not be uniform as $\tau\uparrow T$.
	
	The radial moment argument follows Huang--Meng--Zhang \cite{HMZ},
	with the exterior pressure terms estimated on the Cauchy domain.
	The exponents in the moment estimates are fixed once and for all. At
	$p=\frac{3}{2}$ the coercivity margin equals
	\[
	(2p-1)\alpha(4\alpha-2)-4p^2(\alpha-1)^2
	=-\alpha^2+14\alpha-9.
	\]
	This concave polynomial is positive at both endpoints $1,11.7$ and
	therefore on the whole interval. We can choose $p>q>\frac{3}{2}$ such that
	\begin{equation}\label{eq:momentchoice}
		4p^2(\alpha-1)^2<(2p-1)\alpha(4\alpha-2).
	\end{equation}
	The letter $q$ in this section denotes the scalar moment exponent.
	
	\begin{lemma}[Mass and the two energies]\label{lem:energies}
		For every $0\le t<T$, mass is conserved:
		\[
		\int_{\R^3}\rho(t,x)\dd x=\int_{\R^3}\rho_0(x)\dd x=4\pi M.
		\]
		The ordinary energy identity is
		\begin{equation}\label{eq:ordinary-energy}
			\begin{aligned}
				&\int_{\R^3}\left(\frac{\rho|\bu|^2}{2}
				+\frac{\rho^\gamma}{\gamma-1}\right)(t,x)\dd x
				+\int_0^t\int_{\R^3}\rho^\alpha
				\bigl(|\nabla\bu|^2+(\alpha-1)(\divg\bu)^2\bigr)\dd x\dd s\\
				&\qquad=\int_{\R^3}\left(\frac{\rho_0|\bu_0|^2}{2}
				+\frac{\rho_0^\gamma}{\gamma-1}\right)\dd x,
			\end{aligned}
		\end{equation}
		and the effective (BD) energy identity is
		\begin{equation}\label{eq:bd-energy}
			\begin{aligned}
				&\int_{\R^3}\left(\frac{\rho|\bv|^2}{2}
				+\frac{\rho^\gamma}{\gamma-1}\right)(t,x)\dd x
				+\alpha\gamma\int_0^t\int_{\R^3}
				\rho^{\alpha+\gamma-3}|\nabla\rho|^2\dd x\dd s\\
				&\qquad=\int_{\R^3}\left(\frac{\rho_0|\bv_0|^2}{2}
				+\frac{\rho_0^\gamma}{\gamma-1}\right)\dd x.
			\end{aligned}
		\end{equation}
		The spacetime integrands are evaluated at $(s,x)$. Both dissipations are
		nonnegative, so each energy is bounded by its initial value.
		In particular,
		\begin{equation}\label{eq:baseenergy}
			\begin{gathered}
				\sup_{t<T}\int_0^\infty\rho^\gamma r^2\dd r\le C_T,\qquad
				\sup_{t<T}\int_0^\infty|\partial_r\rho^{\alpha-\frac{1}{2}}|^2r^2\dd r\le C_T,\\
				\int_0^T\int_0^\infty\rho^\alpha(u_r^2r^2+u^2)\dd r\dd t\le C_T.
			\end{gathered}
		\end{equation}
		These particular constants depend only on $\alpha,\gamma$ and the
		initial energies, and can be chosen independently of $T$.
		The global physical and effective moment identities used below are valid
		on every fixed time interval $[0,\tau]$, $\tau<T$.
	\end{lemma}
	\begin{proof}
		We first work on $0<r<\infty$ with measure $r^2\dd r$, omitting $4\pi$,
		and fix a time interval $[0,\tau]$, with $\tau<T$. We return to $\R^3$ at the end.
		
		\par\medskip\noindent\textit{Mass in radial coordinates.}\quad
		Here $a\ge0$ denotes a particle's \emph{initial radius}. Define its radius
		at time $t$ by the characteristic equation
		\[
		\partial_t r(t,a)=u(t,r(t,a)),\qquad r(0,a)=a.
		\]
		This change of variables is justified by the regularity of the solution
		on $[0,\tau]$. In fact, $u$ is continuous in time and continuously
		differentiable in radius, and
		\[
		\sup_{0\le s\le\tau}\|u(s)\|_\infty
		+\sup_{0\le s\le\tau}\|u_r(s)\|_\infty<\infty.
		\]
		The smooth odd extension across $r=0$ is therefore bounded and uniformly
		Lipschitz in radius. The characteristic equation has a unique solution
		throughout $[0,\tau]$ that depends continuously differentiably on $a$.
		To compute this derivative, first write the integral equation and
		differentiate with respect to the initial radius:
		\[
		\begin{aligned}
			r(t,a)&=a+\int_0^t u(s,r(s,a))\dd s,\\
			r_a(t,a)&=1+\int_0^t u_r(s,r(s,a))r_a(s,a)\dd s.
		\end{aligned}
		\]
		Thus $r_a=\frac{\partial r}{\partial a}$ satisfies the linear equation
		\[
		\partial_t r_a(t,a)=u_r(t,r(t,a))r_a(t,a),\qquad r_a(0,a)=1.
		\]
		Multiplication by the integrating factor gives
		\[
		\partial_t\left[
		\exp\left(-\int_0^t u_r(s,r(s,a))\dd s\right)r_a(t,a)
		\right]=0.
		\]
		The expression in brackets equals one initially, so
		\begin{equation}\label{eq:initial-radius-jacobian}
			\begin{gathered}
				r_a(t,a)=\exp\left(\int_0^t u_r(s,r(s,a))\dd s\right),\\
				\exp\left(-\int_0^t\|u_r(s)\|_\infty\dd s\right)
				\le r_a(t,a)\le
				\exp\left(\int_0^t\|u_r(s)\|_\infty\dd s\right).
			\end{gathered}
		\end{equation}
		In particular, $r_a>0$ follows from this calculation; no positivity of
		$r_a$ was assumed in solving the linear equation. Since $u(t,0)=0$,
		uniqueness gives $r(t,0)=0$. Hence $a\mapsto r(t,a)$ is strictly
		increasing and maps positive initial radii to positive radii. Also,
		\[
		|r(t,a)-a|\le\int_0^t\|u(s)\|_\infty\dd s<\infty,
		\]
		so $r(t,a)\to\infty$ as $a\to\infty$.
		The map is consequently a $C^1$ bijection of $[0,\infty)$ onto itself,
		with $\dd r=r_a\dd a$. Only regularity on $[0,\tau]$ was used here.
		
		For $a>0$, the radial volume factor evolves according to
		\[
		\begin{aligned}
			\partial_t(r^2r_a)
			&=2r\,u(t,r)r_a+r^2u_r(t,r)r_a\\
			&=\left(u_r(t,r)+\frac{2u(t,r)}{r}\right)r^2r_a
			=d(t,r)r^2r_a,
		\end{aligned}
		\]
		where $r=r(t,a)$. Along the same trajectory, continuity gives
		\[
		\frac{\dd}{\dd t}\rho(t,r(t,a))
		=\rho_t(t,r)+u(t,r)\rho_r(t,r)=-\rho(t,r)d(t,r).
		\]
		The product rule therefore yields the cancellation
		\[
		\begin{aligned}
			\partial_t\bigl[\rho(t,r(t,a))r(t,a)^2r_a(t,a)\bigr]
			&=(-\rho d)r^2r_a+\rho\,d r^2r_a=0.
		\end{aligned}
		\]
		Thus $\rho(t,r(t,a))r(t,a)^2r_a(t,a)=\rho_0(a)a^2$, and changing
		variables yields
		\[
		\begin{aligned}
			\int_0^\infty\rho(t,r)r^2\dd r
			&=\int_0^\infty\rho(t,r(t,a))r(t,a)^2r_a(t,a)\dd a\\
			&=\int_0^\infty\rho_0(a)a^2\dd a=M.
		\end{aligned}
		\]
		For comparison, the enclosed-mass label of this particle is
		$m_0(a)=\int_0^a\rho_0(s)s^2\dd s$. The same calculation on $[0,a]$
		gives $m(t,r(t,a))=m_0(a)$. After relabelling by $m=m_0(a)$, one has
		\[
		r_m=\frac{r_a}{m_0'(a)}
		=\frac{r_a}{\rho_0(a)a^2}
		=\frac{1}{\rho(t,r)r^2},\qquad a>0.
		\]
		This distinguishes differentiation in initial radius from
		differentiation in enclosed mass.
		
		\par\medskip\noindent\textit{Ordinary energy in radial coordinates.}\quad
		The radial and tangential normal stresses are, respectively,
		\[
		\sigma=\rho^\alpha\bigl(u_r+(\alpha-1)d\bigr),\qquad
		\theta=\rho^\alpha\bigl(\frac{u}{r}+(\alpha-1)d\bigr).
		\]
		Because $\sigma=\alpha\rho^\alpha d-\frac{2\rho^\alpha u}{r}$ and
		$\sigma-\theta=\rho^\alpha(u_r-\frac{u}{r})$, their radial divergence is
		\[
		\begin{aligned}
			\sigma_r+\frac{2}{r}(\sigma-\theta)
			&=(\alpha\rho^\alpha d)_r
			-2(\rho^\alpha)_r\frac{u}{r}
			-2\rho^\alpha\left(\frac{u_r}{r}-\frac{u}{r^2}\right)
			+\frac{2\rho^\alpha}{r}\left(u_r-\frac{u}{r}\right)\\
			&=(\alpha\rho^\alpha d)_r-\frac{2u}{r}(\rho^\alpha)_r.
		\end{aligned}
		\]
		Hence \eqref{eq:radial-cauchy} reads
		$\rho \Dt u+(\rho^\gamma)_r=\sigma_r+\frac{2(\sigma-\theta)}{r}$.
		Continuity implies the transport identity
		\[
		\partial_t(\rho w)+r^{-2}\partial_r(r^2\rho u w)=\rho \Dt w.
		\]
		
		Applying it to $w=\frac{u^2}{2}$, and using $\Dt \rho=-\rho d$ for the
		internal energy, gives
		\[
		\begin{aligned}
			\partial_t\left(\frac{\rho u^2}{2}\right)
			+\frac{1}{r^2}\partial_r\left(r^2u\frac{\rho u^2}{2}\right)
			&=-u(\rho^\gamma)_r+u\sigma_r+\frac{2u}{r}(\sigma-\theta),\\
			\partial_t\left(\frac{\rho^\gamma}{\gamma-1}\right)
			+\frac{1}{r^2}\partial_r\left(r^2u\frac{\rho^\gamma}{\gamma-1}\right)
			&=-\rho^\gamma d.
		\end{aligned}
		\]
		The pressure and viscous terms combine as follows:
		\[
		\begin{aligned}
			-r^2u(\rho^\gamma)_r-r^2\rho^\gamma d
			&=-\partial_r(r^2u\rho^\gamma),\\
			r^2u\sigma_r+2ru(\sigma-\theta)
			&=\partial_r(r^2u\sigma)-r^2\sigma u_r-2ru\theta,\\
			r^2\sigma u_r+2ru\theta
			&=\rho^\alpha\bigl[r^2u_r^2+2u^2
			+(\alpha-1)d(r^2u_r+2ru)\bigr]\\
			&=\rho^\alpha\bigl[r^2u_r^2+2u^2+(\alpha-1)r^2d^2\bigr].
		\end{aligned}
		\]
		We obtain the radial local balance
		\begin{equation}\label{eq:radial-energy-u}
			\begin{aligned}
				&\partial_t\left[r^2\left(\frac{\rho u^2}{2}
				+\frac{\rho^\gamma}{\gamma-1}\right)\right]\\
				&\quad+\partial_r\left[r^2u\left(\frac{\rho u^2}{2}
				+\frac{\gamma\rho^\gamma}{\gamma-1}-\sigma\right)\right]\\
				&\quad+\rho^\alpha\bigl[r^2u_r^2+2u^2+(\alpha-1)r^2d^2\bigr]=0.
			\end{aligned}
		\end{equation}
		Its dissipation is nonnegative since $\alpha>1$.
		
		\par\medskip\noindent\textit{Effective energy in radial coordinates.}\quad
		Equations~\eqref{eq:radial} and \eqref{eq:effective} give
		\[
		\rho \Dt v=-\rho F_1(v-u)=-(\rho^\gamma)_r,
		\qquad v-u=\alpha\rho^{\alpha-2}\rho_r.
		\]
		The transport identity with $w=\frac{v^2}{2}$ therefore yields
		\[
		\partial_t\left(r^2\frac{\rho v^2}{2}\right)+\partial_r\left(r^2u\frac{\rho v^2}{2}\right)
		=-r^2v(\rho^\gamma)_r.
		\]
		
		Adding the same internal-energy balance as above, we calculate
		\[
		\begin{aligned}
			-r^2v(\rho^\gamma)_r-r^2\rho^\gamma d
			&=-\partial_r(r^2u\rho^\gamma)-r^2(v-u)(\rho^\gamma)_r,\\
			(v-u)(\rho^\gamma)_r
			&=\alpha\rho^{\alpha-2}\rho_r\,
			\gamma\rho^{\gamma-1}\rho_r
			=\alpha\gamma\rho^{\alpha+\gamma-3}\rho_r^2.
		\end{aligned}
		\]
		Consequently the radial BD balance is
		\begin{equation}\label{eq:radial-energy-v}
			\begin{aligned}
				&\partial_t\left[r^2\left(\frac{\rho v^2}{2}
				+\frac{\rho^\gamma}{\gamma-1}\right)\right]\\
				&\quad+\partial_r\left[r^2u\left(\frac{\rho v^2}{2}
				+\frac{\gamma\rho^\gamma}{\gamma-1}\right)\right]
				+\alpha\gamma\rho^{\alpha+\gamma-3}\rho_r^2r^2=0.
			\end{aligned}
		\end{equation}
		
		\par\medskip\noindent\textit{Radial integration and cutoff limits.}\quad
		Choose $\eta_R(r)=\eta\left(\frac{r}{R}\right)$, where $0\le\eta\le1$, $\eta=1$ on
		$[0,1]$, $\eta=0$ on $[2,\infty)$, and $|\eta_R'|\le \frac{C}{R}$.
		Multiply \eqref{eq:radial-energy-u} and \eqref{eq:radial-energy-v}
		by $\eta_R$ and integrate over $(0,\infty)\times(0,t)$, $t\le\tau$.
		Smooth center parity gives $u,v=O(r)$, $\rho_r=O(r)$, and bounded
		$\sigma,\theta$ near zero, so both fluxes multiplied by $r^2$ vanish
		at the origin. On the fixed interval, $\rho^\alpha,\rho^\gamma\le C_\tau\rho$
		and $u,v,u_r,\frac{u}{r}$ are bounded. Thus
		\[
		\left|u\left(\frac{\rho u^2}{2}
		+\frac{\gamma\rho^\gamma}{\gamma-1}-\sigma\right)\right|
		+\left|u\left(\frac{\rho v^2}{2}
		+\frac{\gamma\rho^\gamma}{\gamma-1}\right)\right|\le C_\tau\rho,
		\]
		and the time-integrated cutoff errors are bounded by
		\[
		\frac{C_\tau}{R}\int_0^t\int_R^{2R}\rho(s,r)r^2\dd r\dd s
		\le\frac{C_\tau\tau M}{R}\longrightarrow0.
		\]
		Fatou's lemma first bounds the nonnegative dissipations by the initial
		energies. Once their integrability is known, dominated convergence in
		the cutoff balances gives the identities
		\begin{equation}\label{eq:radial-integrated-energies}
			\begin{aligned}
				&\int_0^\infty\left(\frac{\rho u^2}{2}
				+\frac{\rho^\gamma}{\gamma-1}\right)(t,r)r^2\dd r
				+\int_0^t\int_0^\infty\rho^\alpha
				\left[u_r^2+\frac{2u^2}{r^2}+(\alpha-1)d^2\right]r^2\dd r\dd s\\
				&\qquad=\int_0^\infty\left(\frac{\rho_0 u_0^2}{2}
				+\frac{\rho_0^\gamma}{\gamma-1}\right)r^2\dd r,\\[1.5ex]
				&\int_0^\infty\left(\frac{\rho v^2}{2}
				+\frac{\rho^\gamma}{\gamma-1}\right)(t,r)r^2\dd r
				+\alpha\gamma\int_0^t\int_0^\infty
				\rho^{\alpha+\gamma-3}\rho_r^2r^2\dd r\dd s\\
				&\qquad=\int_0^\infty\left(\frac{\rho_0 v_0^2}{2}
				+\frac{\rho_0^\gamma}{\gamma-1}\right)r^2\dd r.
			\end{aligned}
		\end{equation}
		The constant $C_\tau$ has disappeared. The initial energies are
		finite: $u_0$ is bounded, and
		$v_0-u_0=\alpha\rho_0^{\alpha-1}(\log\rho_0)_r$ is bounded by
		\eqref{eq:data}--\eqref{eq:logdata}; finite mass then controls both
		kinetic terms and $\rho_0^\gamma$.
		
		For the density-gradient estimate, the effective-velocity relation gives
		\[
		\partial_r\rho^{\alpha-\frac{1}{2}}
		=\frac{\alpha-\frac{1}{2}}{\alpha}\sqrt\rho\,(v-u),
		\]
		and hence
		\[
		\begin{aligned}
			\int_0^\infty|\partial_r\rho^{\alpha-\frac{1}{2}}|^2r^2\dd r
			&\le\frac{2(\alpha-\frac{1}{2})^2}{\alpha^2}
			\int_0^\infty\rho(u^2+v^2)r^2\dd r\\
			&\le\frac{4(\alpha-\frac{1}{2})^2}{\alpha^2}
			\int_0^\infty\left(
			\frac{\rho_0(u_0^2+v_0^2)}{2}
			+\frac{2\rho_0^\gamma}{\gamma-1}\right)r^2\dd r
			\le C_T.
		\end{aligned}
		\]
		Together with \eqref{eq:radial-integrated-energies}, this proves
		\eqref{eq:baseenergy} by letting $\tau\uparrow T$.
		
		\par\medskip\noindent\textit{Velocity moment identities in radial coordinates.}\quad
		Use $\eta_R|u|^{2p-2}u r^2$ in radial momentum and continuity for
		$\frac{\rho|u|^{2p}}{2p}$. The three scalar fluxes whose cutoff derivatives
		occur are
		\[
		\frac{\rho|u|^{2p}}{2p}u,\qquad
		\rho^\gamma|u|^{2p-2}u,\qquad
		\sigma|u|^{2p-2}u.
		\]
		Each is bounded in absolute value by $C_\tau\rho$, so its
		integrated cutoff error is at most $\frac{C_\tau\tau M}{R}$.
		The bulk terms are integrable with measure $r^2\dd r\dd t$:
		they contain bounded powers of $u$ multiplying $u_r,\frac{u}{r},\sigma,\theta$
		or $\rho^\gamma$, and are bounded by $C_\tau\rho$.
		At the center $u=O(r)$ and $u_r,\frac{u}{r}$ are bounded, which also removes
		the origin flux. Thus the physical moment identity used below is valid.
		For the effective moment, the radial transport identity gives directly
		\[
		\partial_t\left(r^2\frac{\rho|v|^{2q}}{2q}\right)
		+\partial_r\left(r^2u\frac{\rho|v|^{2q}}{2q}\right)
		=-r^2\rho F_1|v|^{2q-2}v(v-u).
		\]
		The scalar flux and source after division by $r^2$ are bounded by
		$C_\tau\rho$, which allows us to let the cutoff radius tend to infinity.
		Nonintegral $p,q>\frac{3}{2}$ cause no singularity at zero velocity: the
		multiplier has the required differentiability, or one may first use
		$u(u^2+\varepsilon)^{p-1}$ and then let $\varepsilon\downarrow0$.
		
		\par\medskip\noindent\textit{Conversion back to $\R^3$.}\quad
		Set $\mathbf e_r=\frac{x}{r}$, $\bu=u\mathbf e_r$, and $\bv=v\mathbf e_r$.
		For radial scalar integrands $f$ and scalar fluxes $j$,
		\[
		\int_{\R^3}f\dd x=4\pi\int_0^\infty f r^2\dd r,\qquad
		\divg(j\mathbf e_r)=\frac{1}{r^2}\partial_r(r^2j),\qquad
		|\nabla\rho|^2=\rho_r^2.
		\]
		Moreover,
		\[
		\nabla\bu=u_r\mathbf e_r\otimes\mathbf e_r
		+\frac{u}{r}(I-\mathbf e_r\otimes\mathbf e_r),\qquad
		|\nabla\bu|^2=u_r^2+\frac{2u^2}{r^2},\qquad \divg\bu=d.
		\]
		The stress in \eqref{eq:system} therefore satisfies
		\[
		\mathsf S=\sigma\mathbf e_r\otimes\mathbf e_r
		+\theta(I-\mathbf e_r\otimes\mathbf e_r),\qquad
		\mathsf S\bu=\sigma u\mathbf e_r,\qquad
		\mathsf S:\nabla\bu=\rho^\alpha\bigl(|\nabla\bu|^2+(\alpha-1)d^2\bigr).
		\]
		Writing the energy densities in Cartesian variables, division of
		\eqref{eq:radial-energy-u}--\eqref{eq:radial-energy-v} by $r^2$ gives
		\begin{equation}\label{eq:cont:2.1}
			\begin{aligned}
				&\partial_t\left(\frac{\rho|\bu|^2}{2}
				+\frac{\rho^\gamma}{\gamma-1}\right)
				+\divg\left[\bu\left(\frac{\rho|\bu|^2}{2}
				+\frac{\gamma\rho^\gamma}{\gamma-1}\right)-\mathsf S\bu\right]\\
				&\quad+\rho^\alpha\bigl(|\nabla\bu|^2+(\alpha-1)(\divg\bu)^2\bigr)=0,
			\end{aligned}
		\end{equation}
		\begin{equation}\label{eq:cont:2.2}
			\begin{aligned}
				&\partial_t\left(\frac{\rho|\bv|^2}{2}
				+\frac{\rho^\gamma}{\gamma-1}\right)
				+\divg\left[\bu\left(\frac{\rho|\bv|^2}{2}
				+\frac{\gamma\rho^\gamma}{\gamma-1}\right)\right]\\
				&\quad+\alpha\gamma\rho^{\alpha+\gamma-3}|\nabla\rho|^2=0.
			\end{aligned}
		\end{equation}
		These extend through $x=0$ by smoothness.
		Multiplying \eqref{eq:radial-integrated-energies} by $4\pi$ therefore
		gives \eqref{eq:ordinary-energy} and \eqref{eq:bd-energy}.
		The radial mass identity likewise gives
		\[
		\int_{\R^3}\rho(t,x)\dd x=\int_{\R^3}\rho_0(x)\dd x=4\pi M,
		\]
		and the radial density-gradient estimate becomes
		\begin{equation}\label{eq:cont:2.3}
			\begin{aligned}
				\|\nabla\rho^{\alpha-\frac{1}{2}}(t)\|_2^2
				&\le\frac{2(\alpha-\frac{1}{2})^2}{\alpha^2}
				\int_{\R^3}\rho\bigl(|\bu|^2+|\bv|^2\bigr)\dd x\\
				&\le\frac{4(\alpha-\frac{1}{2})^2}{\alpha^2}
				\int_{\R^3}\left(
				\frac{\rho_0(|\bu_0|^2+|\bv_0|^2)}{2}
				+\frac{2\rho_0^\gamma}{\gamma-1}\right)\dd x
				\le C_T.
			\end{aligned}
		\end{equation}
		The radial moment balances likewise become their Cartesian versions by
		replacing $r^{-2}\partial_r(r^2j)$ with $\divg(j\mathbf e_r)$ and
		$r^2\dd r$ with $(4\pi)^{-1}\dd x$.
	\end{proof}
	
	Write $\beta=\alpha-\frac{1}{2}$ and $g=\rho^\beta$ in the next lemma.
	
	\begin{lemma}[Two radial density estimates]\label{lem:radialtail}
		The energy and mass bounds imply, for $r>0$,
		\[
		\rho(t,r)\le C_T r^{-\frac{1}{2\alpha-1}},\qquad
		r^2\rho(t,r)^\alpha\le C_T.
		\]
		In particular the density is uniformly bounded on $r\ge1$ and tends
		to zero there uniformly in $t<T$.
	\end{lemma}
	\begin{proof}
		Since $g(\infty)=0$, the fundamental theorem of calculus and Cauchy--Schwarz
		give, for every $r>0$,
		\[
		\begin{aligned}
			g(r)&=-\int_r^\infty g_s(s)\dd s
			\le\int_r^\infty |g_s(s)|\dd s\\
			&\le\left(\int_r^\infty|g_s|^2s^2\dd s\right)^{\frac{1}{2}}
			\left(\int_r^\infty s^{-2}\dd s\right)^{\frac{1}{2}}
			\le C_T r^{-\frac{1}{2}}.
		\end{aligned}
		\]
		The improper integral is absolutely convergent on every $[r,\infty)$ by the
		same calculation. Thus
		\begin{equation}\label{eq:hmz:2.1}
			\rho(r)\le C_T r^{-\frac{1}{2\alpha-1}}.
		\end{equation}
		Mass gives another useful estimate. Since $(\rho^\alpha)_r=\frac{\alpha}{\beta}\rho^{\frac{1}{2}}g_r,$
		we have
		\begin{equation}\label{eq:hmz:2.2}
			\begin{aligned}
				\rho^\alpha(r)
				&\le\frac{\alpha}{\beta}\int_r^\infty\rho^{\frac{1}{2}}|g_s|\dd s\\
				&\le\frac{\alpha}{\beta}
				\left(\int_r^\infty\rho s^2\dd s\right)^{\frac{1}{2}}
				\left(\int_r^\infty|g_s|^2s^{-2}\dd s\right)^{\frac{1}{2}}\\
				&\le\frac{\alpha}{\beta r^2}M^{\frac{1}{2}}
				\left(\int_r^\infty|g_s|^2s^2\dd s\right)^{\frac{1}{2}}
				\le\frac{C_T}{r^2}.
			\end{aligned}
		\end{equation}
		The last step uses $s^{-2}\le r^{-4}s^2$ for $s\ge r$.
		In particular, $\rho(t,r)\le C_T$ on $r\ge1$.
	\end{proof}
	
	\begin{lemma}[Physical velocity moment]\label{lem:physicalmoment}
		For the choice \eqref{eq:momentchoice},
		\[
		\sup_{t<T}\int_0^\infty\rho|u|^{2p}r^2\dd r
		+\int_0^T\int_0^\infty\rho^\alpha
		\bigl(|u|^{2p}+|u|^{2p-2}u_r^2r^2\bigr)\dd r\dd t\le C_T.
		\]
	\end{lemma}
	\begin{proof}
		The radial momentum multiplier $|u|^{2p-2}ur^2$ gives
		\begin{equation}\label{eq:hmz:3.1}
			\begin{aligned}
				&\frac{1}{2p}\frac{\dd}{\dd t}
				\int_0^\infty\rho|u|^{2p}r^2\dd r\\
				&\quad+\int_0^\infty\rho^\alpha\bigl[
				(2p-1)\alpha|u|^{2p-2}u_r^2r^2
				+(4\alpha-2)|u|^{2p}\\
				&\hspace{9em}+4p(\alpha-1)|u|^{2p-2}u\,u_r\,r\bigr]\dd r\\
				&=\int_0^\infty\rho^\gamma\bigl[
				(2p-1)|u|^{2p-2}u_r\,r^2+2|u|^{2p-2}u\,r\bigr]\dd r.
			\end{aligned}
		\end{equation}
		The quadratic form is positive definite by \eqref{eq:momentchoice}.
		Thus the dissipation on the left is bounded below by
		\[
		c_p\int_0^\infty\rho^\alpha
		\bigl(|u|^{2p}+|u|^{2p-2}u_r^2r^2\bigr)\dd r
		\]
		for some $c_p>0$ depending only on $\alpha,p$.
		No second derivative occurs in \eqref{eq:hmz:3.1}. Before expansion,
		the integrated viscous expression is
		\[
		\begin{aligned}
			&-\alpha\int_0^\infty\rho^\alpha\left(u_r+\frac{2u}{r}\right)
			\bigl[(2p-1)|u|^{2p-2}u_r\,r^2+2|u|^{2p-2}u\,r\bigr]\dd r\\
			&\quad+2\int_0^\infty\rho^\alpha
			\bigl[2p|u|^{2p-2}u\,u_r\,r+|u|^{2p}\bigr]\dd r.
		\end{aligned}
		\]
		The cross coefficient is
		$-2\alpha(2p-1)-2\alpha+4p=-4p(\alpha-1)$,
		the $|u|^{2p}$ coefficient is $-4\alpha+2$, and the derivative-square
		coefficient is $-(2p-1)\alpha$. Moving these to the left gives \eqref{eq:hmz:3.1}.
		
		On $0<r<1$, put $E=2p(\gamma-\alpha)+\alpha$. For the first pressure term,
		\[
		\rho^\gamma|u|^{2p-2}|u_r|r^2 =\bigl(\rho^{\frac{\alpha}{2}}|u|^{p-1}|u_r|r\bigr) \bigl(\rho^{\gamma-\frac{\alpha}{2}}|u|^{p-1}r\bigr)
		\]
		and Young's inequality first gives a derivative dissipation plus
		$C_\varepsilon\rho^{2\gamma-\alpha}|u|^{2p-2}r^2$.
		Apply Young again with conjugate exponents $\frac{p}{p-1}$ and $p$:
		\[
		\begin{aligned}
			\rho^{2\gamma-\alpha}|u|^{2p-2}r^2
			&=(\rho^\alpha|u|^{2p})^{\frac{p-1}{p}}
			\rho^{2\gamma-2\alpha+\frac{\alpha}{p}}r^2\\
			&\le\varepsilon\rho^\alpha|u|^{2p}
			+C_\varepsilon\rho^{2p(\gamma-\alpha)+\alpha}r^{2p}.
		\end{aligned}
		\]
		The second pressure term admits the same bound by factoring it as
		\[
		\rho^\gamma|u|^{2p-1}r
		=(\rho^{\frac{\alpha}{2}}|u|^p)(\rho^{\gamma-\frac{\alpha}{2}}|u|^{p-1}r).
		\]
		By \eqref{eq:hmz:2.1},
		\[
		\int_0^1\rho^Er^{2p}\dd r
		\le C_T\int_0^1r^{2p-\frac{E}{2\alpha-1}}\dd r\le C_T,
		\]
		because
		\[
		2p-\frac{E}{2\alpha-1}>-1
		\quad\Longleftrightarrow\quad
		\gamma<3\alpha-1+\frac{\alpha-1}{2p}.
		\]
		The range \eqref{eq:range} satisfies this strict inequality.
		
		On $r\ge1$, use only the first Young step for each pressure term. The
		remaining expression is
		\[
		\begin{aligned}
			\int_1^\infty\rho^{2\gamma-\alpha}|u|^{2p-2}r^2\dd r
			&=\int_1^\infty\rho\,\rho^{2\gamma-\alpha-1}
			|u|^{2p-2}r^2\dd r\\
			&\le C_T\int_1^\infty\rho(1+|u|^{2p})r^2\dd r\\
			&\le C_T\left(1+\int_0^\infty\rho|u|^{2p}r^2\dd r\right).
		\end{aligned}
		\]
		Here $2\gamma-\alpha-1\ge3\alpha-3>0$, so the exterior density upper
		bound suffices. After absorbing the dissipation terms, we obtain
		\[
		\begin{aligned}
			&\frac{\dd}{\dd t}\int_0^\infty\rho|u|^{2p}r^2\dd r
			+c\int_0^\infty\rho^\alpha
			\bigl(|u|^{2p}+|u|^{2p-2}u_r^2r^2\bigr)\dd r\\
			&\qquad\le C_T\left(1+\int_0^\infty\rho|u|^{2p}r^2\dd r\right).
		\end{aligned}
		\]
		Gronwall and time integration yield
		\begin{equation}\label{eq:hmz:3.2}
			\begin{aligned}
				&\sup_{t<T}\int_0^\infty\rho|u|^{2p}r^2\dd r\\
				&\quad+\int_0^T\int_0^\infty\rho^\alpha
				\bigl(|u|^{2p}+|u|^{2p-2}u_r^2r^2\bigr)\dd r\dd t\le C_T.
			\end{aligned}
		\end{equation}
	\end{proof}
	
	\begin{lemma}[Effective velocity moment]\label{lem:effectivemoment}
		For $\frac{3}{2}<q<p$ as above,
		\[
		\sup_{t<T}\int_0^\infty\rho|v|^{2q}r^2\dd r
		+\int_0^T\int_0^\infty\rho^{\gamma-\alpha+1}
		|v-u|^{2q}r^2\dd r\dd t\le C_T.
		\]
	\end{lemma}
	\begin{proof}
		The effective equation and continuity imply
		\[
		\begin{aligned}
			&\frac{1}{2q}\frac{\dd}{\dd t}
			\int_0^\infty\rho|v|^{2q}r^2\dd r\\
			&\quad+\frac{\gamma}{\alpha}\int_0^\infty
			\rho^{\gamma-\alpha+1}|v|^{2q-2}v(v-u)r^2\dd r=0.
		\end{aligned}
		\]
		For real numbers $v,u$ and $2q>1$, Young's inequality gives
		\[
		|v|^{2q-2}v(v-u) \ge\frac{1}{2}|v|^{2q}-C_q|u|^{2q} \ge c_q|v-u|^{2q}-C_q'|u|^{2q}.
		\]
		Consequently,
		\begin{equation}\label{eq:hmz:4.1}
			\begin{aligned}
				&\frac{\dd}{\dd t}\int_0^\infty\rho|v|^{2q}r^2\dd r
				+c\int_0^\infty\rho^{\gamma-\alpha+1}|v-u|^{2q}r^2\dd r\\
				&\qquad\le C\int_0^\infty\rho^{\gamma-\alpha+1}|u|^{2q}r^2\dd r.
			\end{aligned}
		\end{equation}
		On $0<r<1$, set $k=\gamma-2\alpha+1$ and
		$\theta=\frac{q-1}{p-1}\in(0,1)$. We have $0\le k<\alpha$, and
		\eqref{eq:hmz:2.1} gives
		\[
		r^2\rho^k\le C_T r^{2-\frac{k}{2\alpha-1}}\le C_T.
		\]
		Also,
		\[
		\rho^\alpha|u|^{2q} =(\rho^\alpha|u|^{2p})^\theta (\rho^\alpha|u|^2)^{1-\theta}.
		\]
		Thus H\"older in space-time gives
		\[
		\begin{aligned}
			\int_0^T\int_0^1\rho^{\gamma-\alpha+1}|u|^{2q}r^2\dd r\dd t
			&\le C_T\left(\int_0^T\int_0^1\rho^\alpha|u|^{2p}\dd r\dd t\right)^\theta\\
			&\qquad\times
			\left(\int_0^T\int_0^1\rho^\alpha u^2\dd r\dd t\right)^{1-\theta}
			\le C_T.
		\end{aligned}
		\]
		On $r\ge1$, the exterior density upper bound and mass interpolation give
		\[
		\begin{aligned}
			\int_1^\infty\rho^{\gamma-\alpha+1}|u|^{2q}r^2\dd r
			&\le C_T\int_0^\infty\rho|u|^{2q}r^2\dd r\\
			&\le C_T\left(\int_0^\infty\rho|u|^{2p}r^2\dd r\right)^{\frac{q}{p}}
			M^{1-\frac{q}{p}}\le C_T.
		\end{aligned}
		\]
		Integrating \eqref{eq:hmz:4.1} proves
		\begin{equation}\label{eq:hmz:4.2}
			\sup_{t<T}\int_0^\infty\rho|v|^{2q}r^2\dd r
			+\int_0^T\int_0^\infty
			\rho^{\gamma-\alpha+1}|v-u|^{2q}r^2\dd r\dd t\le C_T.
		\end{equation}
	\end{proof}
	
	\begin{proposition}[Global upper density bound]\label{prop:upper}
		The density and the two velocity moments satisfy
		\[
		\begin{aligned}
			\sup_{t<T}\Biggl(&\|\rho(t)\|_\infty
			+\int_0^\infty\rho|u|^{2p}r^2\dd r
			+\int_0^\infty\rho|v|^{2q}r^2\dd r\\
			&+\|\nabla\rho^{\alpha-1+\frac{1}{2q}}(t)\|_{2q}\Biggr)\le C_T.
		\end{aligned}
		\]
		The constant uses no positive density lower bound.
	\end{proposition}
	\begin{proof}
		Set $s=\alpha-1+\frac{1}{2q}>0$ and $g=\rho^s$. Then
		
		\[
		\partial_rg=\frac{s}{\alpha}\rho^{\frac{1}{2q}}(v-u),
		\]
		and hence
		\[
		\begin{aligned}
			\int_0^\infty|g_r|^{2q}r^2\dd r
			&=\left(\frac{s}{\alpha}\right)^{2q}
			\int_0^\infty\rho|v-u|^{2q}r^2\dd r\\
			&\le C\int_0^\infty\rho|v|^{2q}r^2\dd r\\
			&\quad+C\left(\int_0^\infty\rho|u|^{2p}r^2\dd r\right)^{\frac{q}{p}}
			M^{1-\frac{q}{p}}
			\le C_T.
		\end{aligned}
		\]
		Since $s<2\alpha-1\le\gamma$, the pressure energy controls the local mean.
		Sobolev--Poincare on $B_2$, with $2q>3$, then gives
		\[
		\begin{aligned}
			\int_{B_2}g\dd x
			&\le |B_2|^{1-\frac{s}{\gamma}}\left(\int_{B_2}\rho^\gamma\dd x\right)^{\frac{s}{\gamma}}\le C_T,\\
			\|g\|_{L^\infty(B_1)}
			&\le C\bigl(\|\nabla g\|_{L^{2q}(B_2)}+\|g\|_{L^1(B_2)}\bigr)\le C_T.
		\end{aligned}
		\]
		Together with \eqref{eq:hmz:2.2} on $r\ge1$, this gives the estimate
		\begin{equation}\label{eq:hmz:5.1}
			\begin{aligned}
				\sup_{t<T}\Biggl(&\|\rho(t)\|_\infty
				+\int_0^\infty\rho|u|^{2p}r^2\dd r
				+\int_0^\infty\rho|v|^{2q}r^2\dd r\\
				&+\|\nabla\rho^{\alpha-1+\frac{1}{2q}}(t)\|_{L^{2q}(\R^3)}
				\Biggr)\le C_T.
			\end{aligned}
		\end{equation}
	\end{proof}
	
	The initial energies and moments above are finite under the theorem's
	data: $u_0,v_0$ are bounded, $\rho_0\in L^1\cap L^\infty$, and
	\[
	v_0=u_0+\alpha\rho_0^{\alpha-1}(\log\rho_0)_r,\qquad
	|\nabla\rho_0^{\alpha-\frac{1}{2}}|^2
	\le C\rho_0^{2\alpha-1}|\nabla\log\rho_0|^2\le C\rho_0.
	\]
	Hence none of these initial integral quantities is an extra assumption.
	
	\section{Velocity bounds and local positivity of the density}\label{sec:barriers}
	The purpose of this section is to show that the density stays uniformly
	positive on every fixed ball up to time $T$. We first establish one-sided
	velocity bounds, then control the reciprocal density inside a material
	sphere. A final transport argument connects this estimate on a moving
	region to the desired lower bound on a fixed spatial region.
	
	Let $\bar\rho_T$ denote the upper density bound on $[0,T)$ supplied by
	Proposition~\ref{prop:upper}. It depends only on $T$, $\alpha,\gamma$,
	and the initial data. We write $C_T$ for other finite positive constants
	with these dependencies. Their value may increase from one
	estimate to the next; a constant used in a comparison equation is fixed
	throughout that comparison.
	
	The radial mass satisfies $0\le m<M$ and
	\begin{equation}\label{eq:bar:2.2}
		m\le\frac{\bar\rho_T r^3}{3},\qquad
		\sigma=\frac{r}{m^{\frac{1}{3}}}
		\ge c_\sigma=\left(\frac{3}{\bar\rho_T}\right)^{\frac{1}{3}}.
	\end{equation}
	The upper-density estimate and Lemma~\ref{lem:radialtail} imply
	\begin{equation}\label{eq:bar:2.3}
		r^2\rho^\alpha\le C_T.
	\end{equation}
	Set
	\[
	k=\frac{2(\alpha-1)}{\alpha},\qquad
	\bar F_1=\frac{\gamma}{\alpha}\bar\rho_T^{\gamma-\alpha}.
	\]
	Using \eqref{eq:bar:2.3}, we obtain
	\begin{equation}\label{eq:bar:2.4}
		\begin{aligned}
			F_1r^k
			&=\frac{\gamma}{\alpha}\rho^{\gamma-2\alpha+1}
			(r^2\rho^\alpha)^{\frac{\alpha-1}{\alpha}}\\
			&\le C_T\bar\rho_T^{\gamma-2\alpha+1}\le C_T.
		\end{aligned}
	\end{equation}
	The exponent $\gamma-2\alpha+1$ is nonnegative, including at the lower pressure endpoint.
	
	The following comparison controls both outward velocity and the positive
	parts of $\frac{u}{r}$ and $\frac{v}{r}$. The latter bounds will control growth terms in
	the reciprocal-density argument and already give a lower density bound
	at the center.
	
	\begin{lemma}[Positive velocity and transverse strain rate bounds]\label{lem:records}
		There is a constant $C_T<\infty$ such that
		\[
		u_+,v_+\le C_T,\qquad
		\left(\frac{u}{r}\right)_+,\left(\frac{v}{r}\right)_+\le C_T.
		\]
		The estimate holds for every pair $(\alpha,\gamma)$ satisfying \eqref{eq:range}.
	\end{lemma}
	\begin{proof}
		Define
		\[
		P=\frac{u}{m^{\frac{1}{3}}},\qquad
		Z=\frac{v}{m^{\frac{1}{3}}},\qquad
		e=\frac{\rho r^3}{3m},\qquad C_e=2{(\alpha-1)}+(\alpha+1)e.
		\]
		These quantities have finite center limits, and $P,Z\to0$ at infinity.
		The radial equations give
		\begin{equation}\label{eq:bar:3.1}
			\begin{aligned}
				\Dt P={}&\alpha {\rho^{\alpha-1}}\left(P_{rr}+\frac{2e+2}{r}P_r\right)
				+\alpha(v-u)P_r
				-\frac{2\alpha {\rho^{\alpha-1}}(1-e)^2}{r^2}P\\
				&+(P-Z)\left(F_1-\frac{C_eP}{\sigma}\right),\\
				\Dt Z={}&F_1(P-Z),\qquad \left.\sigma_t\right|_m=P.
			\end{aligned}
		\end{equation}
		At a positive spatial maximum of $P$ with $r>0$, $P_r=0$, $P_{rr}\le0$,
		so only the last reaction can be positive. At the center $e=1+O(r^2)$ and
		\[
		P=P(t,0)+\frac{P_{rr}(t,0)}{2}r^2+o(r^2).
		\]
		Consequently,
		\[
		P_{rr}+\frac{2e+2}{r}P_r\longrightarrow5P_{rr}(t,0)\le0,
		\]
		while $\frac{(1-e)^2P}{r^2}\to0$ and $(v-u)P_r\to0$. The same differential
		inequality therefore holds when $P$ attains a positive spatial maximum
		at the center.
		
		Let
		\[
		N(t)=\max\{0,\sup_{0\le m<M}P(t,m),\sup_{0\le m<M}Z(t,m)\}.
		\]
		The initial value is finite. At a point where $Z(t,m)=N(t)>0$,
		we have $Z\ge P$, hence $Z_t\le0$. At a point where $P(t,m)=N(t)>0$,
		we have $P\ge Z$. Next we split into two cases.
		
		\par\medskip\noindent\textit{The range $\gamma-\alpha\ge\frac{1}{3}$.}\quad
		Temporarily denote $y=\rho^{\frac{1}{3}}\sigma$, $c = 2(\alpha-1)$ and $b=\frac{\alpha+1}{3}$. Then $e=\frac{y^3}{3}$, we have
		\[
		\frac{F_1\sigma}{C_e}
		=\frac{\gamma}{\alpha}\rho^{\gamma-\alpha-\frac{1}{3}}
		\frac{y}{2(\alpha-1)+\frac{\alpha+1}{3}y^3}.
		\]
		For $c>0$ and $b>0$,
		\[
		\left(\frac{y}{c+by^3}\right)'
		=\frac{c-2by^3}{(c+by^3)^2},\qquad
		\sup_{y\ge0}\frac{y}{c+by^3}
		=\frac{2}{3c}\left(\frac{c}{2b}\right)^{\frac{1}{3}}.
		\]
		Hence $\frac{F_1\sigma}{C_e}\le C_T$, with a constant determined by
		$\alpha,\gamma,\bar\rho_T$. Fix such a constant. At a point where $P(t,m)=N(t)>C_T$, both
		$P-Z\ge0$ and $F_1-\frac{C_eP}{\sigma}\le0$. Equation \eqref{eq:bar:3.1} gives
		\begin{equation}\label{eq:bar:3.2}
			N(t)\le\max\{N(0),C_T\}.
		\end{equation}
		\par\medskip\noindent\textit{The remaining wedge: coupled upper and lower bounds.}\quad
		Suppose $0<\gamma-\alpha<\frac{1}{3}$. Since $\gamma-\alpha\ge\alpha-1$, necessarily
		\[
		1<\alpha<\frac{4}{3},\qquad
		0<k=\frac{2(\alpha-1)}{\alpha}<\frac{1}{2}.
		\]
		Define
		\[
		J=r^k u,\qquad
		\chi(m)=\left(\frac{3m}{\bar\rho_T}\right)^{\frac{k}{3}},\qquad H=\chi v.
		\]
		The weight $\chi$ is constant along material trajectories and
		$0\le\chi\le r^k$. For arbitrary $k$, differentiating $u=r^{-k}J$ gives
		\[
		\begin{aligned}
			u_r&=r^{-k}J_r-kr^{-k-1}J,\\
			u_{rr}&=r^{-k}J_{rr}-2kr^{-k-1}J_r+k(k+1)r^{-k-2}J.
		\end{aligned}
		\]
		Since $\Dt (r^ku)=r^k\Dt u+kr^{k-1}u^2$, \eqref{eq:radial} becomes
		\begin{equation}\label{eq:bar:3.3}
			\begin{aligned}
				\Dt J={}&\alpha {\rho^{\alpha-1}}\left[J_{rr}+\frac{2-2k}{r}J_r
				+\frac{(k-2)(k+1)}{r^2}J\right]
				+\alpha(v-u)J_r\\
				&+(2(\alpha-1)-\alpha k)(v-u)\frac{J}{r}
				+F_1(J-r^kv)+\frac{kJ^2}{r^{k+1}}.
			\end{aligned}
		\end{equation}
		Our choice makes $2{(\alpha-1)}-\alpha k=0$. Also
		\begin{equation}\label{eq:bar:3.4}
			\Dt H=F_1(\chi u-H).
		\end{equation}
		Set
		\[
		A(t)=\max\{0,\sup(-J(t)),\sup(-H(t))\}.
		\]
		At a negative $J$ joint minimum, $J=-A<0$, $J_r=0$, $J_{rr}\ge0$.
		Since $0<k<2$, both $(k-2)(k+1)J$ and $kJ^2$ are nonnegative. Thus
		\[
		J_t\ge-\bar F_1 A-F_1r^kv_+\ge-\bar F_1 A-C_TN.
		\]
		At a negative $H$ joint minimum, $H=-A$ and $J\ge-A$. Therefore
		\[
		\chi u=\frac{\chi}{r^k}J\ge-\frac{\chi}{r^k}A\ge-A=H,
		\]
		so \eqref{eq:bar:3.4} gives $H_t\ge0$. Consequently the joint lower comparison is
		\begin{equation}\label{eq:bar:3.5}
			D^+A\le\bar F_1 A+C_TN.
		\end{equation}
		Endpoint justification introduces no density profile: $J,H\to0$ at the
		center by parity, and $H\to0$ at infinity because
		$\chi\le\left(\frac{3M}{\bar\rho_T}\right)^{\frac{k}{3}}$
		and $v\to0$. For $J$, radial vector $H^1$ regularity gives
		\[
		\int_0^\infty(r^2u^2+r^2u_r^2+2u^2)\,dr<\infty.
		\]
		For $g=ru$,
		\[
		\int_0^\infty g^2\,dr=\int_0^\infty r^2u^2\,dr,\qquad
		\int_0^\infty g_r^2\,dr
		\le2\int_0^\infty(u^2+r^2u_r^2)\,dr.
		\]
		Hence $g\in H^1(0,\infty)$ and is bounded, uniformly on each time
		interval $[0,\tau]$, $\tau<T$. Since $k<1$,
		$|J|=r^{k-1}|g|\to0$ uniformly there. This also proves $A(0)<\infty$
		from the stated initial regularity.
		
		Now choose $m_c=\frac{M}{2}$, and set
		\[
		S(t)=\sup_{0\le m\le m_c}\sigma(t,m),\qquad
		C_A=\frac{\bar F_1}{2(\alpha-1)}>0.
		\]
		Initially $S(0)<\infty$ by positivity on the corresponding finite initial
		ball. Equation \eqref{eq:bar:3.1} gives $D^+S\le N$.
		
		At a point where $P(t,m)=N(t)>0$ and $m\le m_c$,
		if $N\ge C_AS$, then
		\[
		F_1-\frac{C_eP}{\sigma}\le\bar F_1-\frac{2(\alpha-1)N}{S}\le0,
		\]
		so $P_t\le0$. At a point where $P(t,m)=N(t)>0$ and $m\ge m_c$,
		drop the nonpositive diffusion and $-\frac{C_e(P-Z)P}{\sigma}$ terms. Since $H\ge-A$,
		\[
		Z_-\le\frac{A}{\chi(m)m^{\frac{1}{3}}}
		\le\frac{A}{\chi(m_c)m_c^{\frac{1}{3}}},
		\]
		\begin{equation}\label{eq:bar:3.6}
			P_t\le F_1(P-Z)
			\le\bar F_1\left(N+\frac{A}{\chi(m_c)m_c^{\frac{1}{3}}}\right).
		\end{equation}
		Define the simultaneous barrier size by
		\[
		\mathcal H(t)=\max\{1,N(t),C_AS(t),A(t)\}.
		\]
		If $\mathcal H(t)=N(t)$, then at any point where $P(t,m)=N(t)$,
		we have $P_t\le0$ for $m\le m_c$ because $N\ge C_AS$, and
		$P_t\le C(N+A)$ for $m\ge m_c$. At any point where $Z(t,m)=N(t)$,
		we have $Z_t\le0$. If $\mathcal H(t)=C_AS(t)$, then
		$D^+(C_AS)\le C_AN$. If $\mathcal H(t)=A(t)$,
		use \eqref{eq:bar:3.5}. Thus in every case
		\begin{equation}\label{eq:bar:3.7}
			D^+\mathcal H\le C\mathcal H,\qquad
			\mathcal H(t)\le\mathcal H(0)e^{Ct}.
		\end{equation}
		This closes the small-exponent wedge without assuming a positive angular
		bound or a negative velocity bound in advance.
		
		\par\medskip\noindent\textit{Consequences and comparison interpretation.}\quad
		Combining \eqref{eq:bar:3.2} and \eqref{eq:bar:3.7} gives $N(t)\le C_T$.
		Since $m<M$ and $\sigma\ge c_\sigma$, the definitions imply
		\[
		u_+,v_+\le M^{\frac{1}{3}}N(t),\qquad
		U_+,V_+\le\frac{N(t)}{c_\sigma}.
		\]
		Absorbing the fixed factors into $C_T$, we conclude that
		\begin{equation}\label{eq:bar:3.8}
			u_+,v_+,U_+,V_+\le C_T.
		\end{equation}
		The Dini notation above can equivalently be read as first-contact
		comparisons with strictly larger scalar ODE barriers. On any $[0,\tau]$,
		positive or negative contact levels bounded away from zero lie in a compact
		spatial region, by the uniform endpoint limits. The interval defining $S$
		is already compact in mass. On these contact sets classical regularity
		justifies differentiation. Fractional regularity of $r^k$ at the center
		does not matter for negative contacts of $J,H$, whose center values are zero.
		Use barriers $\mathcal H(0)+\varepsilon$ and derivative
		$C\mathcal H+\varepsilon$, then let $\varepsilon\downarrow0$ and
		$\tau\uparrow T$.
		
		The bound at the center follows directly: parity gives
		$u_r(t,0)=U(t,0)$, so continuity implies
		\begin{equation}\label{eq:bar:3.9}
			\rho(t,0)=\rho_0(0)
			\exp\left(-3\int_0^tU(s,0)\,ds\right)
			\ge\rho_0(0)e^{-3C_TT}>0.
		\end{equation}
		Excluding vacuum at all other finite radii requires the next two steps.
	\end{proof}
	
	The estimate at the center does not yet give a lower density bound on a
	whole ball. We also need to keep a sufficiently distant material sphere
	outside the ball throughout $[0,T)$. For this purpose, the next lemma
	bounds the inward speeds away from the origin by comparing the shifted
	fluxes $r^2(u+G)$ and $r^2(v+G)$.
	
	\begin{lemma}[Shifted flux and inward transport]\label{lem:inward}
		There is a constant $C_T<\infty$ such that
		\[
		u_-(t,r),v_-(t,r)\le C_T\left(1+\frac{1}{r^2}\right)
		\qquad(0<t<T,\ r>0).
		\]
		This estimate is used only away from the center.
	\end{lemma}
	\begin{proof}
		The construction in this section uses only \eqref{eq:bar:3.8}, the upper density bound, and bounded initial $u_0,v_0$.
		Fix a value of $C_T$ for which \eqref{eq:bar:3.8} holds, and choose
		\[
		G_0\ge\max\{1,C_T,\|(u_0)_-\|_\infty,\|(v_0)_-\|_\infty\},
		\qquad G(t)=G_0e^{(2\bar F_1+1)t}.
		\]
		Then $v_+\le C_T\le G$ and
		\begin{equation}\label{eq:bar:4.1}
			G'-\bar F_1(C_T+G)\ge G.
		\end{equation}
		Define
		\[
		\mathsf P=r^2(u+G),\qquad \mathsf Z=r^2(v+G),\qquad
		\mathcal L\varphi=\alpha r^4\partial_m(\rho^{\alpha+1}\varphi_m).
		\]
		For an Eulerian function $\varphi$,
		\begin{equation}\label{eq:bar:4.2}
			\mathcal L\varphi=\alpha {\rho^{\alpha-1}} \varphi_{rr}
			+\left[\alpha(v-u)-\frac{2\alpha {\rho^{\alpha-1}}}{r}\right]\varphi_r.
		\end{equation}
		Indeed, $\varphi_m=\frac{\varphi_r}{\rho r^2}$, so
		\[
		\begin{aligned}
			\mathcal L\varphi
			&=\frac{\alpha r^2}{\rho}\partial_r(\rho^\alpha r^{-2}\varphi_r)\\
			&=\alpha\rho^{\alpha-1}\varphi_{rr}
			+\alpha^2\rho^{\alpha-2}\rho_r \varphi_r
			-\frac{2\alpha\rho^{\alpha-1}\varphi_r}{r}.
		\end{aligned}
		\]
		In particular $\mathcal L\varphi\ge0$ at an interior spatial minimum. Put
		$Q=r^2u$. Direct substitution of \eqref{eq:radial} gives
		\begin{equation}\label{eq:bar:4.3}
			\Dt Q-\mathcal LQ
			=\frac{4Q^2}{r^3}-\frac{2Qv}{r}+F_1(Q-r^2v).
		\end{equation}
		Also, \eqref{eq:bar:4.2} gives
		\[
		\mathcal L(r^2)=2\alpha r(v-u)-2\alpha {\rho^{\alpha-1}},
		\qquad \Dt (Gr^2)=G'r^2+2Gru.
		\]
		Substitute $Q=\mathsf P-Gr^2$, $v=\frac{\mathsf Z}{r^2}-G$ in \eqref{eq:bar:4.3}, and add the
		last two identities. The individual expansions are
		\[
		\begin{aligned}
			\frac{4Q^2}{r^3}
			&=\frac{4\mathsf P^2}{r^3}-\frac{8G\mathsf P}{r}+4G^2r,\\
			-\frac{2Qv}{r}
			&=-\frac{2\mathsf P\mathsf Z}{r^3}+\frac{2G\mathsf P}{r}+\frac{2G\mathsf Z}{r}-2G^2r,\\
			F_1(Q-r^2v)&=F_1(\mathsf P-\mathsf Z),\\
			2Gru-G\mathcal L(r^2)
			&=\frac{2G\mathsf P}{r}-2G^2r
			-\frac{2\alpha G(\mathsf Z-\mathsf P)}{r}+2\alpha\rho^{\alpha-1}G.
		\end{aligned}
		\]
		Thus the $G^2r$ terms cancel, the $\frac{G\mathsf P}{r}$ coefficient is
		$-8+2+2+2\alpha=2(\alpha-2)$, and the $\frac{G\mathsf Z}{r}$ coefficient is
		$2-2\alpha=-2{(\alpha-1)}$. The resulting equation is
		\begin{equation}\label{eq:bar:4.4}
			\begin{aligned}
				\Dt \mathsf P-\mathcal L\mathsf P={}&
				\frac{4\mathsf P^2}{r^3}-\frac{2\mathsf P\mathsf Z}{r^3}
				+\frac{2(\alpha-2)G\mathsf P}{r}-\frac{2{(\alpha-1)}G\mathsf Z}{r}\\
				&+F_1(\mathsf P-\mathsf Z)+G'r^2+2\alpha {\rho^{\alpha-1}}G,
			\end{aligned}
		\end{equation}
		\begin{equation}\label{eq:bar:4.5}
			\Dt \mathsf Z=F_1(\mathsf P-\mathsf Z)+\frac{2u\mathsf Z}{r}+G'r^2.
		\end{equation}
		Both shifted fluxes are initially nonnegative, vanish at the center, and
		tend to $+\infty$ at infinity uniformly on each time interval $[0,\tau]$, $\tau<T$.
		The last fact uses $u,v\to0$ and $G\ge1$; no bound on $r^2u$ or $r^2v$ is
		needed. Their negative minima are attained at finite positive radii.
		
		Set
		\[
		A_s(t)=\max\{0,\sup(-\mathsf P(t)),\sup(-\mathsf Z(t))\}.
		\]
		At a negative $\mathsf Z$ joint minimum, $\mathsf Z=-A_s\le\mathsf P$.
		Equation \eqref{eq:bar:4.5} implies
		\begin{equation}\label{eq:bar:4.6}
			\Dt \mathsf Z\ge-2U_+A_s\ge-2C_TA_s.
		\end{equation}
		At a negative $\mathsf P$ joint minimum, $\mathsf P=-A_s$ and
		$\mathsf Z\ge-A_s$. Equation \eqref{eq:bar:4.4} and
		$\mathcal L\mathsf P\ge0$ give the reaction lower bound
		\begin{equation}\label{eq:bar:4.7}
			\Dt \mathsf P\ge
			\frac{2A_s^2}{r^3}+\frac{2GA_s}{r}
			+(\mathsf Z+A_s)\left(\frac{2A_s}{r^3}-\frac{2{(\alpha-1)}G}{r}-F_1\right)
			+G'r^2+2\alpha {\rho^{\alpha-1}}G.
		\end{equation}
		To control the growth of $A_s$, we seek a lower bound for $\Dt\mathsf P$
		that does not involve $\mathsf Z$. Since $\mathsf Z+A_s\ge0$, the sign of
		the product in \eqref{eq:bar:4.7} is determined by its coefficient in
		parentheses. If
		\[
		\frac{2A_s}{r^3}-\frac{2(\alpha-1)G}{r}-F_1\ge0,
		\]
		then this product is nonnegative. All remaining terms in
		\eqref{eq:bar:4.7} are also nonnegative, because $A_s\ge0$, $G>0$, and
		$G'>0$. Hence $\Dt\mathsf P\ge0$ in this case.
		
		It remains to consider the case in which this coefficient is negative.
		The previously proved bound $v\le C_T$ gives
		\[
		0\le\mathsf Z+A_s=r^2(v+G)+A_s\le r^2(C_T+G)+A_s.
		\]
		Multiplying this upper bound by the negative coefficient gives the lower
		bound
		\[
		\begin{aligned}
			&(\mathsf Z+A_s)
			\left(\frac{2A_s}{r^3}-\frac{2(\alpha-1)G}{r}-F_1\right)\\
			&\qquad\ge\bigl[r^2(C_T+G)+A_s\bigr]
			\left(\frac{2A_s}{r^3}-\frac{2(\alpha-1)G}{r}-F_1\right).
		\end{aligned}
		\]
		Substituting into \eqref{eq:bar:4.7} and expanding removes $\mathsf Z$.
		Dropping the nonnegative term $2\alpha\rho^{\alpha-1}G$ and using
		$F_1\le\bar F_1$ in the terms $-F_1A_s$ and $-F_1(C_T+G)r^2$, we obtain
		\begin{equation}\label{eq:bar:4.8}
			\begin{aligned}
				\Dt \mathsf P\ge{}&\frac{4A_s^2}{r^3}
				+\frac{[2C_T+(6-2\alpha)G]A_s}{r}-\bar F_1 A_s\\
				&-2{(\alpha-1)}G(C_T+G)r
				+[G'-\bar F_1(C_T+G)]r^2.
			\end{aligned}
		\end{equation}
		Let $c_\alpha=(2\alpha-6)_+$ and
		$C_\alpha=4{(\alpha-1)}+\frac{c_\alpha^2}{4}$. Since $C_T\le G$, \eqref{eq:bar:4.1} and
		\[
		\frac{c_\alpha GA_s}{r} \le\frac{A_s^2}{r^3}+\frac{c_\alpha^2}{4}G^2r
		\]
		yield
		\begin{equation}\label{eq:bar:4.9}
			\begin{aligned}
				\Dt \mathsf P
				&\ge\frac{3A_s^2}{r^3}-\bar F_1 A_s-C_\alpha G^2r+Gr^2\\
				&=\frac{3A_s^2}{r^3}-\bar F_1 A_s
				+G\left(r-\frac{C_\alpha G}{2}\right)^2-\frac{C_\alpha^2G^3}{4}\\
				&\ge-\bar F_1 A_s-\frac{C_\alpha^2G^3}{4}.
			\end{aligned}
		\end{equation}
		The final lower bound also holds when the coefficient in parentheses is
		nonnegative, since then $\Dt\mathsf P\ge0$. Combining the two cases with
		\eqref{eq:bar:4.6}, and setting $C_1=\max\{2C_T,\bar F_1\}$ and
		$C_2=\frac{C_\alpha^2}{4}$,
		we obtain
		\[
		D^+A_s\le C_1A_s+C_2G^3,\qquad A_s(0)=0,
		\]
		\begin{equation}\label{eq:bar:4.10}
			A_s(t)\le C_2e^{C_1t}\int_0^t e^{-C_1s}G(s)^3\,ds
			\le C_2T G_0^3e^{[C_1+3(2\bar F_1+1)]T}<\infty.
		\end{equation}
		When the comparison starts from zero, we use strictly negative comparison
		functions in the first contact argument, as in the center argument of Lemma~\ref{lem:records}. The endpoint behavior excludes contacts at zero
		or infinity. After this comparison, enlarge $C_T$ to bound both $A_s(t)$
		and $G(t)$ on $[0,T)$. Since both shifted fluxes are at least $-C_T$,
		\begin{equation}\label{eq:bar:4.11}
			u_-(t,r),v_-(t,r)\le C_T\left(1+\frac{1}{r^2}\right).
		\end{equation}
		This proves the asserted estimate. The singular $r^{-2}$ bound is used only away from the
		center; it is not an assertion that the velocities are singular there.
	\end{proof}
	
	The inward velocity bound will be used when we pass from material
	coordinates to a fixed ball. We first prove positivity inside an arbitrary
	material sphere, using the positive velocity bounds of
	Lemma~\ref{lem:records}. Weighting the reciprocal density by a factor that
	vanishes on the sphere allows the comparison to close without requiring
	a lower density bound at this artificial outer surface.
	
	\begin{proposition}[Positivity on a material interior]\label{prop:shielding}
		For every $0<m_0<m_1<M$ there exists $C_{T,m_1}<\infty$ such that
		\[
		\rho(t,m)\ge \frac{(m_1-m_0)^{\frac{1}{\alpha}}}{C_{T,m_1}},\qquad
		0\le t<T,\quad0\le m\le m_0.
		\]
		The lower bound holds up to the spatial origin.
	\end{proposition}
	\begin{proof}
		Fix $0<m_0<m_1<M$. By initial positivity,
		$S_1=\sup_{m\le m_1}\sigma(0,m)<\infty$. Since
		$\sigma_t=U\sigma$ at fixed $m$ and $U_+\le C_T$,
		\begin{equation}\label{eq:bar:5.1}
			c_\sigma\le\sigma(t,m)\le S_1e^{C_TT},\qquad
			r(t,m)\le m_1^{\frac{1}{3}}S_1e^{C_TT}=:R_1,
		\end{equation}
		\begin{equation}\label{eq:bar:5.2}
			\begin{aligned}
				\int_0^tU_-(s,m)\,ds
				&=\int_0^tU_+(s,m)\,ds
				-\log\frac{\sigma(t,m)}{\sigma(0,m)}\\
				&\le C_TT+\log\left(\frac{S_1}{c_\sigma}\right)=:L_1.
			\end{aligned}
		\end{equation}
		No later-time density lower bound was used. Define
		\[
		z=\rho^{-1},\quad\delta=m_1-m,\quad
		\phi=\delta^{\frac{1}{\alpha}},\quad X=\phi z,
		\]
		and choose the following primitive of $\frac{F_1}{\rho^2}$:
		\[
		F_2(\rho)=\begin{cases}
			\dfrac{\gamma\rho^{\gamma-\alpha-1}}{\alpha(\gamma-\alpha-1)},&\gamma-\alpha\ne1,\\[6pt]
			\dfrac{\gamma}{\alpha}\log\rho,&\gamma-\alpha=1,
		\end{cases}
		\qquad F_2'(\rho)=\frac{F_1(\rho)}{\rho^2}.
		\]
		We define the effective pressure by
		\[
		Y=zv_r+F_2(\rho)=\frac{v_r}{\rho}+F_2(\rho).
		\]
		\par\medskip\noindent\textit{Derive the effective-pressure equation.}\quad
		Differentiating $\Dt v=F_1(u-v)$ gives
		\[
		\Dt v_r=F_1(u_r-v_r)+F_1'\rho_r(u-v)-u_rv_r.
		\]
		Since $\Dt z=z(u_r+2U)$, the product and chain rules yield
		\[
		\begin{aligned}
			\Dt (zv_r)&=zF_1(u_r-v_r)-zF_1'\rho_r(v-u)+2Uzv_r,\\
			\Dt F_2&=-F_1z(u_r+2U).
		\end{aligned}
		\]
		
		Adding and using $zv_r=Y-F_2$ and
		$\rho_r=\frac{v-u}{\alpha\rho^{\alpha-2}}$ yields
		\begin{equation}\label{eq:bar:5.3}
			\Dt Y+(F_1-2U)Y
			=F_1F_2-2U(F_1z+F_2)
			-\frac{F_1'}{\alpha\rho^{\alpha-1}}(v-u)^2.
		\end{equation}
		The last term is nonpositive because ${(\gamma-\alpha)}>0$. Its integrating-factor kernel
		along a fixed mass label satisfies
		\begin{equation}\label{eq:bar:5.4}
			\exp\left(-\int_s^t(F_1-2U)\,d\tau\right)\le e^{2C_TT}.
		\end{equation}
		\par\medskip\noindent\textit{Derive the critical reciprocal inequality, including the center.}\quad
		The relation between physical and effective velocity gives
		\[
		\rho_r=-\rho^3r^2z_m,\qquad
		u=v+\alpha r^2z^{-\alpha-1}z_m,\qquad r_m=\frac{z}{r^2}.
		\]
		Substituting this into $z_t=z\left(u_r+\frac{2u}{r}\right)$, differentiating the product in
		$u_m$, and using $u_r=\rho r^2u_m$, gives explicitly
		\[
		\begin{aligned}
			u_m={}&\frac{z}{r^2}v_r
			+2\alpha r^{-1}z^{-\alpha}z_m
			+\alpha r^2z^{-\alpha-1}z_{mm}
			-\alpha(\alpha+1)r^2z^{-\alpha-2}z_m^2,\\
			zu_r={}&zv_r+2\alpha r z^{-\alpha}z_m
			+\alpha r^4z^{-\alpha-1}z_{mm}
			-\alpha(\alpha+1)r^4z^{-\alpha-2}z_m^2,\\
			\frac{2zu}{r}={}&2Vz+2\alpha r z^{-\alpha}z_m.
		\end{aligned}
		\]
		Adding the last two lines and writing $zv_r=Y-F_2$ yields
		\begin{equation}\label{eq:bar:5.5}
			z_t=Y-F_2+2Vz
			+\alpha r^4z^{-\alpha-1}z_{mm}
			+4\alpha r z^{-\alpha}z_m
			-\alpha(\alpha+1)r^4z^{-\alpha-2}z_m^2.
		\end{equation}
		At an interior maximum of $X$, $X_m=0$, $X_{mm}\le0$. Since
		$\frac{\phi_m}{\phi}=-\frac{1}{\alpha\delta}$ and
		$\frac{\phi_{mm}}{\phi}=\frac{1-\alpha}{\alpha^2\delta^2}$,
		\[
		z_m=\frac{z}{\alpha\delta},\qquad
		z_{mm}\le\frac{(\alpha+1)z}{\alpha^2\delta^2}.
		\]
		Consequently the second-derivative and squared-gradient terms in $X_t$
		have sum at most
		\[
		\frac{\alpha r^4\phi z^{-\alpha}}{\delta^2}
		\left[\frac{\alpha+1}{\alpha^2}-\frac{\alpha+1}{\alpha^2}\right]=0.
		\]
		The remaining first-derivative term equals
		\[
		4\alpha r\phi z^{-\alpha}z_m =\frac{4r\phi z^{1-\alpha}}{\delta} =4rX^{1-\alpha},
		\]
		where $\phi^\alpha=\delta$ was used. Therefore
		\begin{equation}\label{eq:bar:5.6}
			X_t\le\phi Y-\phi F_2+2VX+4rX^{1-\alpha}.
		\end{equation}
		At the artificial outer material surface $m=m_1$, $X=0$ at every
		time $t<T$, without any velocity condition there. At a center
		maximum, parity gives
		\[
		m=\frac{\rho(t,0)r^3}{3}+O(r^5),\qquad
		X=m_1^{\frac{1}{\alpha}}\left[z(t,0)+\frac{z_{rr}(t,0)r^2}{2}\right]+O(r^3).
		\]
		Thus $z_{rr}(t,0)\le0$. In Eulerian form the reciprocal diffusion has
		principal part $\alpha\rho^{\alpha-1}\left(z_{rr}+\frac{2z_r}{r}\right)$ and a nonpositive
		squared-gradient term. Its center limit is
		$3\alpha\rho(t,0)^{\alpha-1}z_{rr}(t,0)\le0$. Hence \eqref{eq:bar:5.6} holds at a
		center maximum with the final term zero. No singular mass derivative is
		evaluated at the center.
		
		\par\medskip\noindent\textit{Bounds for the effective pressure and reciprocal density.}\quad
		Set $Q_X(t)=\max_{m\le m_1}X(t,m)$ and $H_X=\max\{1,Q_X\}$. The actual
		initial data on the finite initial material sphere bound $H_X(0)$ and
		$\sup_{m\le m_1}(Y(0,m))_+$.
		
		If $0<{(\gamma-\alpha)}<1$, then
		\[
		F_1F_2\le0,\qquad F_1z+F_2=\frac{\gamma {(\gamma-\alpha)}}{\alpha(\gamma-\alpha-1)}z^{1-\gamma+\alpha}<0,
		\]
		\[
		[-2U(F_1z+F_2)]_+\le C C_T z^{1-\gamma+\alpha},\qquad
		\phi z^{1-\gamma+\alpha}=\phi^{\gamma-\alpha}X^{1-\gamma+\alpha}\le C H_X.
		\]
		Multiply the pathwise estimate from \eqref{eq:bar:5.3}--\eqref{eq:bar:5.4} by the time-independent
		$\phi(m)$. It follows that
		\begin{equation}\label{eq:bar:5.7}
			\phi Y(t,m)\le C+C\int_0^tH_X(s)\,ds,\qquad
			-\phi F_2\le CH_X(t).
		\end{equation}
		If ${(\gamma-\alpha)}=1$, then $(F_1F_2)_+$ and $(F_1z+F_2)_+$ are bounded by the upper
		density bound, while $[-2U(F_1z+F_2)]_+ \le CC_T\log(1+z)+CU_-.$
		Use $\phi\log(1+z)\le\phi z=X$ and \eqref{eq:bar:5.2} in \eqref{eq:bar:5.3}--\eqref{eq:bar:5.4} to obtain
		\eqref{eq:bar:5.7} again. Also $-\phi F_2\le\frac{\gamma}{\alpha}X$ follows from
		$\log z\le z$.
		
		If ${(\gamma-\alpha)}>1$, $F_1F_2$ and $F_1z+F_2$ are nonnegative and bounded. Their positive
		source in \eqref{eq:bar:5.3} is at most $C+CU_-$. Equations \eqref{eq:bar:5.2}--\eqref{eq:bar:5.4} give
		$\phi Y\le C$, and $-\phi F_2\le0$. Thus the upper inequalities \eqref{eq:bar:5.7}
		remain valid in this case too.
		
		At an active maximum $Q_X\ge1$, \eqref{eq:bar:5.1}, \eqref{eq:bar:5.6}, and $V_+\le C_T$ imply
		\[
		D^+H_X\le C\left[H_X(t)+\int_0^tH_X(s)\,ds\right],
		\]
		because $X^{1-\alpha}\le1$ there. For
		$\mathcal K=H_X+\int_0^tH_X(s)\,ds$,
		\[
		D^+\mathcal K\le(C+1)\mathcal K,
		\qquad H_X(t)\le H_X(0)e^{(C+1)T}=:C_{T,m_1}.
		\]
		Each comparison is on a compact material interval. The zero endpoint at
		$m_1$ excludes first contact with a positive upper barrier there, even though no uniform
		density lower bound is known at that moving surface. We conclude
		\begin{equation}\label{eq:bar:5.8}
			\rho(t,m)\ge\frac{(m_1-m_0)^{\frac{1}{\alpha}}}{C_{T,m_1}}
			\qquad(0\le m\le m_0<m_1<M).
		\end{equation}
	\end{proof}
	
	Proposition~\ref{prop:shielding} controls a fixed interval of mass labels,
	whose spatial image moves with the flow. To obtain a lower bound on a
	prescribed fixed ball, we choose a material sphere sufficiently far out
	initially and use Lemma~\ref{lem:inward} to ensure that it never reaches
	the ball. The material lower bound then applies throughout the ball.
	
	\begin{corollary}[Positivity on every fixed ball]\label{cor:compactpositive}
		For every finite $L>0$,
		\begin{equation}\label{eq:compactpositive}
			\inf_{t<T,\,r\le L}\rho(t,r)\ge c_{T,L}>0.
		\end{equation}
	\end{corollary}
	\begin{proof}
		Fix $L>0$. By \eqref{eq:bar:4.11},
		\[
		\sup_{t<T,r\ge L}u_-(t,r)\le C_T\left(1+\frac{1}{L^2}\right)=:K_L<\infty.
		\]
		Choose a finite initial radius and its material label by
		\[
		R_0=L+TK_L+1,\qquad m_0=\int_0^{R_0}\rho_0(r)r^2\,dr.
		\]
		Strict positivity at every finite initial radius and finite mass imply
		$0<m_0<M$: for example the initial mass on $(R_0,R_0+1)$ is positive.
		If this particle had a first time $t_*<T$ with $r(t_*,m_0)=L$, then up to
		that time its radius was at least $L$, so
		\[
		r(t_*,m_0)=R_0+\int_0^{t_*}u(s,r(s,m_0))\,ds \ge R_0-t_*K_L>L,
		\]
		a contradiction. Hence $r(t,m_0)>L$ and $m(t,L)<m_0$ uniformly in $t<T$.
		Choose $m_1=(m_0+M)/2$. Apply \eqref{eq:bar:5.8} to obtain
		\[
		\rho(t,r)\ge\frac{(m_1-m_0)^{1/\alpha}}{C_{T,m_1}}>0
		\qquad(0\le t<T,\ 0\le r\le L).
		\]
		This proves the assertion, including the center.
	\end{proof}
	
	Together with Proposition~\ref{prop:upper}, the corollary gives two-sided
	density bounds on each fixed ball. In particular, the viscous coefficient
	$\alpha\rho^{\alpha-1}$ is bounded away from zero there. This local
	nondegeneracy is the starting point for the regularity and deformation
	estimates in the next section.
	
	\section{Local regularity and global velocity bounds}\label{sec:core}
	The goal of this section is to turn the density bounds from the preceding
	sections into control of velocity gradients on fixed balls and global
	bounds for both velocities. The main local conclusion is
	Proposition~\ref{prop:coredeformation}, which bounds physical deformation
	on every fixed ball in a square-integrable time norm. We reach it through a
	first local parabolic estimate, a two-sided effective deformation bound,
	and estimates for two density derivatives. Combining the local estimates
	with the exterior comparisons then yields the global bounds in
	Corollary~\ref{cor:amplitudes}. Along the way, we also establish uniform
	far-field velocity decay for the later exterior analysis.
	
	All constants in a local estimate may depend on the radius through
	\eqref{eq:compactpositive}. No local ellipticity constant will be used
	uniformly as the radius tends to infinity.
	
	We begin by completing the global bound for the effective velocity.
	
	\begin{lemma}[Global effective velocity amplitude]\label{lem:boundedv}
		The effective velocity is bounded globally:
		$\sup_{t<T}\|\bv(t)\|_\infty\le C_T$.
	\end{lemma}
	\begin{proof}
		The preceding effective-pressure argument gives more than a lower density
		bound. For every fixed Eulerian ball $B_R$, it also bounds $Y$ above there:
		\[
		Y=\frac{v_r}{\rho}+F_2(\rho),\qquad
		F_2(\rho)=\begin{cases}
			\frac{\gamma\rho^{\gamma-\alpha-1}}{\alpha(\gamma-\alpha-1)},&{(\gamma-\alpha)}\ne1,\\
			\frac{\gamma}{\alpha}\log\rho,&{(\gamma-\alpha)}=1.
		\end{cases}
		\]
		To see this carefully, select material labels $m_0<m_1<M$ which cover
		$B_R$ as in Corollary~\ref{cor:compactpositive}. Its estimate
		$\phi Y\le C$, where $\phi=(m_1-m)^{\frac{1}{\alpha}}$, then gives
		\[
		Y(t,r)\le C(m_1-m_0)^{-\frac{1}{\alpha}}\quad(r\le R).
		\]
		On this ball, \eqref{eq:compactpositive} bounds $F_2$ from below, including the cases ${(\gamma-\alpha)}\le1$.
		Using the upper density bound,
		\begin{equation}\label{eq:cont:3.1}
			v_r=\rho(Y-F_2)\le C_{T,R}.
		\end{equation}
		Choose $R=1$. For $0\le r\le1$,
		\[
		v(t,r)=v(t,1)-\int_r^1v_s(t,s)\,ds \ge v(t,1)-C_{T,1}(1-r)\ge-C_T,
		\]
		where the last step uses the shifted exterior bound at $r=1$. For $r\ge1$,
		the same exterior bound already gives $v\ge-C_T$. Together with the proved
		global upper bound for $v$,
		\begin{equation}\label{eq:cont:3.2}
			\sup_{t<T}\|\mathbf v(t)\|_\infty\le C_T.
		\end{equation}
		This uses only an upper bound for $v_r$ on a fixed ball. It does not assert
		a global bound for $v_r$ or for $\frac{v}{r}$ in absolute value.
	\end{proof}
	
	With the effective velocity bounded, the identity \eqref{eq:effective}
	links the density gradient to the physical velocity. We now use positivity
	on fixed balls to obtain the first local parabolic estimate for $\bu$.
	Its global consequences are square-integrability in time of the physical
	velocity amplitude and of the supremum norms of
	$\nabla\rho^{\alpha-1}$ and $\nabla F_1$.
	
	\begin{proposition}[First local parabolic estimate]\label{prop:firstorder}
		For every finite $R$,
		\[
		\sup_{t<T}\|\bu(t)\|_{H^1(B_R)}^2
		+\int_0^T\bigl(\|\bu\|_{H^2(B_R)}^2
		+\|\bu_t\|_{L^2(B_R)}^2\bigr)\dd t\le C_{T,R}.
		\]
		In addition, $\int_0^T\|u(t)\|_\infty^2\dd t\le C_T$ and
		$\nabla {\rho^{\alpha-1}},\nabla F_1\in L^2(0,T;L^\infty)$.
	\end{proposition}
	\begin{proof}
		\par\medskip\noindent\textit{The lower-order coefficients on a finite ball.}\quad
		Fix $R>0$. Let $q>3/2$ be the moment exponent of Lemma~\ref{lem:effectivemoment}, and put $s=2q$ and $\ell={(\alpha-1)}+1/s$. On $B_{2R}$, Proposition~\ref{prop:upper} and Corollary~\ref{cor:compactpositive}
		imply $\sup_{t<T}\|\mathbf u(t)\|_{L^s(B_{2R})}\le C_{T,R}.$
		Indeed, the $2p$ moment controls the $s$ moment by mass interpolation, then
		$\rho\ge c_{T,2R}$ removes the mass weight on this ball. Also
		\[
		\begin{aligned}
			\nabla {\rho^{\alpha-1}}&=\frac{\alpha-1}{\ell}\rho^{{(\alpha-1)}-\ell}\nabla\rho^\ell,\\
			\frac{\gamma}{\gamma-1}\nabla\rho^{\gamma-1}
			&=\frac{\gamma}{\ell}\rho^{\gamma-1-\ell}\nabla\rho^\ell.
		\end{aligned}
		\]
		Both density powers are bounded on the fixed positive-density ball. Hence
		\begin{equation}\label{eq:cont:4.1}
			\sup_{t<T}\bigl(
			\|\mathbf u\|_{L^s(B_{2R})}
			+\|\nabla {\rho^{\alpha-1}}\|_{L^s(B_{2R})}
			+\frac{\gamma}{\gamma-1}\|\nabla\rho^{\gamma-1}\|_{L^s(B_{2R})}\bigr)\le C_{T,R}.
		\end{equation}
		Ordinary energy and \eqref{eq:compactpositive} also give
		\begin{equation}\label{eq:cont:4.2}
			\sup_{t<T}\|\mathbf u(t)\|_{L^2(B_{2R})}^2
			+\int_0^T\|\mathbf u(t)\|_{H^1(B_{2R})}^2\,dt\le C_{T,R}.
		\end{equation}
		\par\medskip\noindent\textit{Localize the vector equation, including the commutators.}\quad
		Radial vector fields are curl-free, so $\nabla \left(\operatorname{div}\mathbf u\right)=\Delta\mathbf u$. Expanding
		the stress \eqref{eq:system} and dividing by $\rho$ gives
		\begin{equation}\label{eq:cont:4.3}
			\begin{aligned}
				\mathbf u_t-\alpha {\rho^{\alpha-1}}\Delta\mathbf u
				={}&-\mathbf u\cdot\nabla\mathbf u
				+\frac{\alpha}{\alpha-1}\nabla {\rho^{\alpha-1}}\cdot\nabla\mathbf u\\
				&+\alpha\nabla {\rho^{\alpha-1}}\,\operatorname{div}\mathbf u
				-\frac{\gamma}{\gamma-1}\nabla\rho^{\gamma-1}.
			\end{aligned}
		\end{equation}
		The radial component of the right-hand side is
		\[
		-uu_r+(v-u)\left(\alpha u_r+2(\alpha-1)\frac{u}{r}-F_1\right),
		\]
		in agreement with \eqref{eq:radial}. Holding the solution coefficients
		fixed, define the first-order linear operator on vector test fields by
		\[
		\begin{aligned}
			\mathcal B\nabla\mathbf f
			={}&-\mathbf u\cdot\nabla\mathbf f
			+\frac{\alpha}{\alpha-1}\nabla\rho^{\alpha-1}\cdot\nabla\mathbf f\\
			&+\alpha\nabla\rho^{\alpha-1}\,\operatorname{div}\mathbf f.
		\end{aligned}
		\]
		The coefficient array therefore satisfies
		\[
		|\mathcal B|\le C_\alpha\bigl(|\mathbf u|+|\nabla\rho^{\alpha-1}|\bigr).
		\]
		Choose a cutoff and define the localized velocity by
		\[
		\chi\in C_c^\infty(B_{2R}),\qquad
		\chi=1\text{ on }B_R,\qquad \mathbf w=\chi\mathbf u.
		\]
		The product rule gives
		\begin{equation}\label{eq:cont:4.4}
			\mathbf w_t-\alpha {\rho^{\alpha-1}}\Delta\mathbf w
			=\mathcal B\nabla\mathbf w+\mathbf f,
		\end{equation}
		where
		\[
		\begin{aligned}
			\mathbf f={}&-\frac{\gamma}{\gamma-1}\chi\nabla\rho^{\gamma-1}
			-\mathcal B(\mathbf u\otimes\nabla\chi)\\
			&-\alpha\rho^{\alpha-1}
			\bigl[2\nabla\chi\cdot\nabla\mathbf u+(\Delta\chi)\mathbf u\bigr].
		\end{aligned}
		\]
		
		By \eqref{eq:cont:4.1}, the coefficient norm of $\mathcal B$ is uniformly bounded in
		$L^s$. Moreover $r_s=2s/(s-2)<6$, so \eqref{eq:cont:4.2}, bounded-domain Sobolev embedding,
		and H\"older give
		\[
		\begin{aligned}
			\int_0^T\|\mathcal B(\mathbf u\otimes\nabla\chi)\|_2^2\,dt
			&\le C_{T,R}\int_0^T\|\mathbf u\|_{L^{r_s}(B_{2R})}^2\,dt\\
			&\le C_{T,R}\int_0^T\|\mathbf u\|_{H^1(B_{2R})}^2\,dt\le C_{T,R}.
		\end{aligned}
		\]
		The other terms in $\mathbf f$ are controlled by \eqref{eq:cont:4.1}--\eqref{eq:cont:4.2}. Thus
		\begin{equation}\label{eq:cont:4.5}
			\int_0^T\|\mathbf f(t)\|_2^2\,dt\le C_{T,R}.
		\end{equation}
		No derivative of $\rho$ beyond the already proved $L^s$ gradient is used.
		
		\par\medskip\noindent\textit{Test by $-\Delta\mathbf w$.}\quad
		Since ${\rho^{\alpha-1}}\ge c_{T,2R}^{\alpha-1}$ on the support of $\mathbf w$, \eqref{eq:cont:4.4} gives
		\[
		\frac{1}{2}\frac{d}{dt}\|\nabla\mathbf w\|_2^2
		+\alpha c_{T,2R}^{\alpha-1}\|\Delta\mathbf w\|_2^2
		\le\|\mathcal B\|_s\|\nabla\mathbf w\|_{r_s}\|\Delta\mathbf w\|_2
		+\|\mathbf f\|_2\|\Delta\mathbf w\|_2.
		\]
		
		Let $\theta=3/s<1$. Interpolating $L^2$ and $L^6$ for
		$\nabla\mathbf w$, then using the full-space $L^2$ elliptic identity, gives
		\[
		\|\nabla\mathbf w\|_{r_s} \le C\|\nabla\mathbf w\|_2^{1-\theta} \|\Delta\mathbf w\|_2^\theta.
		\]
		Young's inequality with conjugate exponents $2/(1+\theta)$ and
		$2/(1-\theta)$ therefore gives
		\[
		\|\mathcal B\|_s\|\nabla\mathbf w\|_{r_s}\|\Delta\mathbf w\|_2
		\le\varepsilon\|\Delta\mathbf w\|_2^2
		+C_\varepsilon\|\mathcal B\|_s^{2/(1-\theta)}
		\|\nabla\mathbf w\|_2^2.
		\]
		After absorbing and applying Gronwall using \eqref{eq:cont:4.5},
		\begin{equation}\label{eq:cont:4.6}
			\sup_{t<T}\|\mathbf w(t)\|_{H^1}^2
			+\int_0^T\|\mathbf w(t)\|_{H^2}^2\,dt\le C_{T,R}.
		\end{equation}
		The $L^2$ part was already controlled by \eqref{eq:cont:4.2}. Equation \eqref{eq:cont:4.4} also gives
		$\mathbf w_t\in L^2_tL^2_x$: the drift term is in that space by the same
		interpolation and \eqref{eq:cont:4.6}. Hence
		\begin{equation}\label{eq:cont:4.7}
			\sup_{t<T}\|\mathbf u(t)\|_{H^1(B_R)}^2
			+\int_0^T\left(\|\mathbf u\|_{H^2(B_R)}^2
			+\|\mathbf u_t\|_{L^2(B_R)}^2\right)dt\le C_{T,R}.
		\end{equation}
		
		\par\medskip\noindent\textit{Global consequences.}\quad
		For $R=1$, $H^2(\mathbb R^3)\hookrightarrow L^\infty$ applied to the cutoff
		field in \eqref{eq:cont:4.6} bounds $\int_0^T\|\mathbf u\|_{L^\infty(B_1)}^2\,dt$.
		On $r\ge1$, Lemmas~\ref{lem:records} and~\ref{lem:inward} already bounds $|u|$ uniformly. Consequently
		\begin{equation}\label{eq:cont:4.8}
			\int_0^T\|\mathbf u(t)\|_\infty^2\,dt\le C_T.
		\end{equation}
		The effective velocity relation now yields
		\begin{equation}\label{eq:cont:4.9}
			\nabla {\rho^{\alpha-1}}=\frac{\alpha-1}{\alpha}(\mathbf v-\mathbf u),\qquad
			\int_0^T\|\nabla {\rho^{\alpha-1}}(t)\|_\infty^2\,dt\le C_T.
		\end{equation}
		Also $F_1=(\gamma/\alpha){\rho^{\gamma-\alpha}}$ and ${(\gamma-\alpha)}/{(\alpha-1)}\ge1$, so
		\begin{equation}\label{eq:cont:4.10}
			\nabla F_1=\frac{\gamma {(\gamma-\alpha)}}{\alpha {(\alpha-1)}}{\rho^{\gamma-2\alpha+1}}\nabla {\rho^{\alpha-1}},
			\qquad\int_0^T\|\nabla F_1(t)\|_\infty^2\,dt\le C_T.
		\end{equation}
		This last bound includes equality at ${(\gamma-\alpha)}={(\alpha-1)}$ and contains no negative density
		power. It is useful additional control on the coefficient of $\Dt \mathbf v$.
	\end{proof}
	
	The global time-integrability just obtained also bounds the distance
	travelled by each particle. Before returning to local derivative estimates,
	we use this displacement bound and an exterior comparison to propagate
	the initial decay of both velocities uniformly up to $T$. This prepares
	the estimates at infinity needed in the later exterior argument.
	
	\begin{lemma}[Uniform far-field velocity decay]\label{lem:velocitytails}
		Both velocities have uniformly vanishing tails:
		\[
		\lim_{R\to\infty}\sup_{t<T,\,r\ge R}(|u(t,r)|+|v(t,r)|)=0.
		\]
	\end{lemma}
	\begin{proof}
		\par\medskip\noindent\textit{Effective velocity by bounded displacement.}\quad
		Equation \eqref{eq:cont:4.8} gives a uniform bound on particle displacement:
		\[
		|r(t;r_0)-r_0|\le\int_0^T\|\mathbf u(s)\|_\infty\,ds
		\le\sqrt T\left(\int_0^T\|\mathbf u(s)\|_\infty^2\,ds\right)^{\frac{1}{2}}
		=:D_T<\infty.
		\]
		From \eqref{eq:bar:2.3},
		\begin{equation}\label{eq:cont:5.1}
			\sup_{t<T,r\ge R}F_1(t,r)\le C R^{-\frac{2(\gamma-\alpha)}{\alpha}}.
		\end{equation}
		If a particle is at radius $r\ge R$ at time $t$, every previous point of its
		trajectory has radius at least $R-D_T$, and $r_0\ge R-D_T$. The
		integrating-factor formula for $\Dt v+F_1v=F_1u$ then gives, for $R>D_T+1$,
		\begin{equation}\label{eq:cont:5.2}
			\sup_{t<T,r\ge R}|v(t,r)|
			\le\sup_{r_0\ge R-D_T}|v_0(r_0)|
			+C(R-D_T)^{-\frac{2(\gamma-\alpha)}{\alpha}}\int_0^T\|\mathbf u(s)\|_\infty\,ds
			\longrightarrow0.
		\end{equation}
		This uses the actual initial modulus of decay, not a prescribed profile.
		
		\par\medskip\noindent\textit{Physical velocity by an exterior comparison.}\quad
		On $r\ge1$, Lemmas~\ref{lem:records} and~\ref{lem:inward} and \eqref{eq:cont:3.2} give $|u|+|v|\le C_T$. Write \eqref{eq:radial} as
		\[
		u_t-\alpha\rho^{\alpha-1}u_{rr}+B_u u_r=S_u,
		\]
		where
		\[
		\begin{aligned}
			B_u&=u-\alpha(v-u)-\frac{2\alpha\rho^{\alpha-1}}{r},\\
			S_u&=-\frac{2\alpha\rho^{\alpha-1}u}{r^2}
			+(v-u)\left(\frac{2(\alpha-1)u}{r}-F_1\right).
		\end{aligned}
		\]
		Choose $C_T$, independent of $R$ and fixed throughout the comparison,
		large enough that on $r\ge R\ge1$,
		\[
		|B_u|\le C_T,\qquad 0<\alpha {\rho^{\alpha-1}}\le C_T,\qquad
		|S_u|\le C_T\left(R^{-2}+R^{-1}+R^{-\frac{2(\gamma-\alpha)}{\alpha}}\right)=:q_R.
		\]
		Let $C_0\ge\sup_{t<T,r\ge1}|u(t,r)|$ and
		$e_R=\sup_{r\ge R}|u_0(r)|$. Set
		\[
		\lambda_R=C_T\left(\frac{2}{R^2}+\frac{1}{R}\right),\qquad
		\mathcal E_R(t,r)=e_R+tq_R+C_0e^{\lambda_Rt}\frac{R}{r}.
		\]
		The initial and boundary inequalities are
		$\mathcal E_R(0,r)\ge|u_0(r)|$ and
		$\mathcal E_R(t,R)\ge C_0\ge|u(t,R)|$. Direct differentiation gives
		\[
		\begin{aligned}
			&\partial_t\mathcal E_R-\alpha {\rho^{\alpha-1}}\partial_{rr}\mathcal E_R
			+B_u\partial_r\mathcal E_R\\
			&\quad=q_R+C_0e^{\lambda_Rt}\frac{R}{r}
			\left(\lambda_R-\frac{2\alpha {\rho^{\alpha-1}}}{r^2}-\frac{B_u}{r}\right)
			\ge q_R.
		\end{aligned}
		\]
		Thus comparison with $u$ and $-u$ gives $|u|\le\mathcal E_R$. Work first on
		a fixed time interval $[0,\tau]$, $\tau<T$, where $u\to0$ uniformly at
		infinity; adding
		$\varepsilon(1+t)$ to the barrier makes all far-field first contacts strict.
		The constants are independent of $\tau$. Pass to $\tau\uparrow T$ and
		$\varepsilon\downarrow0$.
		
		For $S\ge R$ this gives
		\[
		\sup_{t<T,r\ge S}|u(t,r)| \le e_R+Tq_R+C_0e^{\lambda_RT}\frac{R}{S}.
		\]
		First let $S\to\infty$ with $R$ fixed, then $R\to\infty$. Since
		$e_R,q_R\to0$,
		\begin{equation}\label{eq:cont:5.3}
			\lim_{S\to\infty}\sup_{t<T,r\ge S}(|u|+|v|)=0.
		\end{equation}
		Consequently ${\partial_r \rho^{\alpha-1}}=\frac{\alpha-1}{\alpha}(v-u)$ also vanishes uniformly at spatial
		infinity up to $T$.
	\end{proof}
	
	We now return to fixed balls to strengthen the one-sided estimate for
	$v_r$ used in Lemma~\ref{lem:boundedv} to a bound for the full effective
	velocity gradient. The global bounds from Lemma~\ref{lem:boundedv} and
	Proposition~\ref{prop:firstorder} make the source in the effective-pressure
	equation time-integrable. A comparison on a material ball then gives both
	sides of the derivative estimate.
	
	\begin{lemma}[Two-sided compact effective deformation]\label{lem:localvgradient}
		For every finite $R$,
		\[
		\sup_{t<T}\|\bv(t)\|_{W^{1,\infty}(B_R)}\le C_{T,R}.
		\]
		The estimate uses no global integral of $\|(u_r)_-\|_\infty$.
	\end{lemma}
	\begin{proof}
		We use the effective pressure $Y$ and the primitive $F_2$ from
		Proposition~\ref{prop:shielding}. Define the nonnegative source
		\begin{equation}\label{eq:core:2.1}
			Q=(\partial_r F_1)(v-u)
			=\frac{\gamma(\gamma-\alpha)}{\alpha^2}
			\rho^{\gamma-2\alpha+1}(v-u)^2\ge0.
		\end{equation}
		The expression follows from $\rho_r=\frac{v-u}{\alpha\rho^{\alpha-2}}$.
		Set $q(t)=\|Q(t)\|_\infty$. Since $\gamma-2\alpha+1\ge0$,
		the upper density bound, Lemma~\ref{lem:boundedv}, and
		Proposition~\ref{prop:firstorder} give
		\begin{equation}\label{eq:core:2.2}
			\begin{aligned}
				\int_0^Tq(t)\,dt
				&\le\frac{\gamma(\gamma-\alpha)}{\alpha^2}\bar\rho_T^{\gamma-2\alpha+1}
				\int_0^T\|v-u\|_\infty^2\,dt\\
				&\le\frac{2\gamma(\gamma-\alpha)}{\alpha^2}\bar\rho_T^{\gamma-2\alpha+1}
				\left(T\sup_{t<T}\|v(t)\|_\infty^2
				+\int_0^T\|u(t)\|_\infty^2\,dt\right)
				\le C_T.
			\end{aligned}
		\end{equation}
		The effective-pressure equation \eqref{eq:bar:5.3} can therefore be written as
		\begin{equation}\label{eq:core:2.4}
			\Dt Y=2U\left(Y-F_2-\frac{F_1}{\rho}\right)
			-F_1(Y-F_2)-\frac{Q}{\rho}.
		\end{equation}
		The $2U$ term is nonpositive when $U<0$ and $Y\ge F_2+\frac{F_1}{\rho}$,
		and nonnegative when $U<0$ and $Y\le F_2+\frac{F_1}{\rho}$.
		We use this sign in the two-sided comparison below.
		
		Fix a finite radius $R$. Choose $C_T$ large enough to bound
		$\int_0^Tq(t)\,dt$, $U_+$ in \eqref{eq:bar:3.8}, and the displacement:
		\[
		\int_0^T\|u(t)\|_\infty\,dt
		\le\sqrt T\left(\int_0^T\|u(t)\|_\infty^2\,dt\right)^{\frac{1}{2}}
		\le C_T.
		\]
		Keep this choice fixed and set $L=R+C_T+1$. Follow the physical trajectory
		$R_L$ starting at $L$. Then
		\begin{equation}\label{eq:core:3.1}
			L-C_T\le R_L(t)\le L+C_T\quad(t<T).
		\end{equation}
		The radial flow preserves ordering, and the center is stationary. Thus
		the material ball $0\le r\le R_L(t)$ contains $[0,R]$ and stays inside
		$[0,L+C_T]$. Compact positivity gives, for a sufficiently large $C_{T,R}$,
		\begin{equation}\label{eq:core:3.2}
			\frac{1}{C_{T,R}}\le\rho(t,r)\le\bar\rho_T
			\quad(0\le r\le R_L(t)).
		\end{equation}
		On this positive-density interval, both $F_2$ and
		$F_2+\frac{F_1}{\rho}$ are bounded. Smoothness and positivity on the initial
		ball also bound $Y(0,r)$. Enlarge $C_{T,R}$ once so that
		\[
		\begin{aligned}
			\max\left\{|F_2(\rho)|,
			\left|F_2(\rho)+\frac{F_1(\rho)}{\rho}\right|\right\}
			&\le C_{T,R}\quad(0\le r\le R_L(t)),\\
			\sup_{r\le L}|Y(0,r)|&\le C_{T,R}.
		\end{aligned}
		\]
		These bounds use the actual density range on the material ball.
		No global bound on $\frac{v_{0r}}{\rho_0}$ is assumed.
		
		Fix this value of $C_{T,R}$ throughout the comparison, and let $B$ solve
		\begin{equation}\label{eq:core:3.3}
			B'=2C_T(B+C_{T,R})+C_{T,R}q(t),\qquad B(0)=C_{T,R}.
		\end{equation}
		The solution $B$ is nondecreasing and satisfies $B\ge C_{T,R}$.
		We compare $Y$ with $B$ and $-B$ along each material characteristic.
		
		At a point where $Y=B$, the coefficient bounds imply
		$B-F_2-\frac{F_1}{\rho}\ge0$ and $B-F_2\ge0$.
		Since $F_1,Q\ge0$, equation
		\eqref{eq:core:2.4} yields
		\[
		\begin{aligned}
			\Dt Y
			&=2U\left(B-F_2-\frac{F_1}{\rho}\right)-F_1(B-F_2)-\frac{Q}{\rho}\\
			&\le 2U_+\left(B-F_2-\frac{F_1}{\rho}\right)\\
			&\le 2C_T(B+C_{T,R})
			\le B'.
		\end{aligned}
		\]
		At a point where $Y=-B$, we instead have
		$B+F_2+\frac{F_1}{\rho}\ge0$ and $B+F_2\ge0$.
		Using also $\frac{1}{\rho}\le C_{T,R}$ and $Q\le q(t)$, we obtain
		\[
		\begin{aligned}
			\Dt Y
			&=-2U\left(B+F_2+\frac{F_1}{\rho}\right)+F_1(B+F_2)-\frac{Q}{\rho}\\
			&\ge -2U_+\left(B+F_2+\frac{F_1}{\rho}\right)-C_{T,R}q(t)\\
			&\ge -2C_T(B+C_{T,R})-C_{T,R}q(t)
			=-B'.
		\end{aligned}
		\]
		Thus $B$ and $-B$ are respectively upper and lower comparison functions.
		Since $|Y(0,r)|\le B(0)$ on the initial ball, comparison along its
		material characteristics gives
		\begin{equation}\label{eq:core:3.4}
			|Y(t,r)|\le B(t)\quad(0\le r\le R_L(t)).
		\end{equation}
		For non-strict contact, first increase the initial value and the right
		side of \eqref{eq:core:3.3} by $\varepsilon>0$, then let $\varepsilon\downarrow0$.
		The outer boundary is itself a material characteristic, so there is no
		unknown inflow datum. At the center the equation extends by smooth parity.
		
		The scalar comparison function is explicit:
		\[
		B(t)+C_{T,R}
		=C_{T,R}e^{2C_Tt}\left(2+\int_0^t e^{-2C_Ts}q(s)\,ds\right).
		\]
		Since $\int_0^Tq(t)\,dt\le C_T$, this implies
		\begin{equation}\label{eq:core:3.5}
			\sup_{t<T}B(t)\le C_{T,R}e^{2C_TT}(2+C_T).
		\end{equation}
		The comparison is now complete, so we may enlarge $C_{T,R}$ in the final
		bounds. The effective radial derivative satisfies
		\[
		|v_r|=|\rho(Y-F_2)|\le\bar\rho_T\bigl(B(t)+|F_2|\bigr)\le C_{T,R}.
		\]
		Radial averaging then gives the effective transverse strain rate:
		\begin{equation}\label{eq:core:3.6}
			V(t,r)=\frac{1}{r}\int_0^r{v_r}(t,s)\,ds,\qquad
			|V(t,r)|\le C_{T,R}\quad(r\le R).
		\end{equation}
		The eigenvalues of $\nabla\mathbf v$ are ${v_r},V,V$, so this proves the assertion.
	\end{proof}
	
	The effective-velocity relation now converts the derivative bounds for
	$\bv$ and the local $H^1$ bound for $\bu$ into control of two derivatives
	of $\rho^{\alpha-1}$. We record this consequence, together with the
	corresponding pressure estimate, to provide the coefficient bounds needed
	for the next parabolic step.
	
	\begin{lemma}[Two compact density derivatives]\label{lem:localh2}
		For every finite $R$,
		\[
		\sup_{t<T}\left(\|{\rho^{\alpha-1}}(t)\|_{H^2(B_R)}
		+\frac{\gamma}{\gamma-1}\|\nabla\rho^{\gamma-1}(t)\|_{H^1(B_R)}\right)\le C_{T,R}.
		\]
	\end{lemma}
	\begin{proof}
		In Cartesian form the identity is
		\begin{equation}\label{eq:core:4.1}
			\nabla {\rho^{\alpha-1}}=\frac{\alpha-1}{\alpha}(\mathbf v-\mathbf u),\qquad
			\nabla^2{\rho^{\alpha-1}}=\frac{\alpha-1}{\alpha}(\nabla\mathbf v-\nabla\mathbf u).
		\end{equation}
		Consequently,
		\begin{equation}\label{eq:core:4.2}
			\begin{aligned}
				\|{\rho^{\alpha-1}}(t)\|_{H^2(B_R)}
				&\le\|{\rho^{\alpha-1}}(t)\|_{L^2(B_R)}
				+C_\alpha\bigl(
				\|\mathbf v(t)\|_{H^1(B_R)}+\|\mathbf u(t)\|_{H^1(B_R)}
				\bigr)\\
				&\le C_{T,R}.
			\end{aligned}
		\end{equation}
		Here use the bounded density, Proposition~\ref{prop:firstorder}, and the newly obtained local
		$W^{1,\infty}$ bound for $\mathbf v$. In particular,
		$\nabla {\rho^{\alpha-1}}\in L^\infty_tL^6_x$ and
		$\nabla^2{\rho^{\alpha-1}}\in L^\infty_tL^2_x$ on every fixed ball.
		
		For the pressure potential $\frac{\gamma}{\gamma-1}\rho^{\gamma-1}$,
		the chain rule gives
		\[
		\begin{aligned}
			\frac{\gamma}{\gamma-1}\nabla\rho^{\gamma-1}
			&=\frac{\gamma}{\alpha-1}\rho^{\gamma-\alpha}\nabla\rho^{\alpha-1},\\
			\frac{\gamma}{\gamma-1}\nabla^2\rho^{\gamma-1}
			&=\frac{\gamma}{\alpha-1}\rho^{\gamma-\alpha}\nabla^2\rho^{\alpha-1}\\
			&\quad+\frac{\gamma(\gamma-\alpha)}{(\alpha-1)^2}\rho^{\gamma-2\alpha+1}
			\nabla\rho^{\alpha-1}\otimes\nabla\rho^{\alpha-1}.
		\end{aligned}
		\]
		The coefficients are bounded on the compact positive-density interval. Hence
		\begin{equation}\label{eq:core:4.3}
			\frac{\gamma}{\gamma-1}\|\nabla\rho^{\gamma-1}\|_{H^1(B_R)}
			\le C_{T,R}\left(
			\|\nabla {\rho^{\alpha-1}}\|_2+\|\nabla^2{\rho^{\alpha-1}}\|_2+\|\nabla {\rho^{\alpha-1}}\|_4^2
			\right)\le C_{T,R}.
		\end{equation}
	\end{proof}
	
	We can now differentiate the localized velocity equation from
	Proposition~\ref{prop:firstorder} once in space.
	Lemma~\ref{lem:localh2} supplies the density and pressure regularity needed
	to close this estimate, yielding a local $H^2$ bound and a deformation
	bound that is square-integrable in time. This is the main local conclusion
	of this section.
	
	\begin{proposition}[Compact physical deformation]\label{prop:coredeformation}
		For every finite $R$,
		\begin{equation}\label{eq:coredeformation}
			\sup_{t<T}\|\bu(t)\|_{H^2(B_R)}^2
			+\int_0^T\bigl(\|\bu\|_{H^3(B_R)}^2
			+\|\bu_t\|_{H^1(B_R)}^2\bigr)\dd t\le C_{T,R}.
		\end{equation}
		Consequently $\int_0^T\|\nabla\bu(t)\|_{L^\infty(B_R)}^2\dd t<\infty$.
		
	\end{proposition}
	\begin{proof}
		\par\medskip\noindent\textit{Reuse the localization and strengthen the coefficient bounds.}\quad
		Fix $R>0$ and use the localization from the proof of
		Proposition~\ref{prop:firstorder}:
		\[
		\chi\in C_c^\infty(B_{2R}),\qquad
		\chi=1\text{ on }B_R,\qquad \mathbf w=\chi\mathbf u.
		\]
		With the same first-order operator $\mathcal B$ and cutoff forcing
		$\mathbf f$, the localized equation \eqref{eq:cont:4.4} already holds.
		We estimate its first spatial derivatives using
		Proposition~\ref{prop:firstorder} and Lemma~\ref{lem:localh2}.
		
		Use the known estimates on $B_{3R}$. Extend ${\rho^{\alpha-1}}$ to a positive constant
		outside this ball, agreeing with it near the support of $\chi$; convex
		cutoff interpolation preserves a positive lower bound. Extend the
		coefficient arrays of $\mathcal B$ by a cutoff. Then
		\begin{equation}\label{eq:core:5.4}
			\alpha {\rho^{\alpha-1}}\ge\nu_{T,R}>0,\qquad
			\|\nabla {\rho^{\alpha-1}}\|_6+\|\mathcal B\|_6+\|\nabla\mathcal B\|_2\le C_{T,R}.
		\end{equation}
		Indeed, $\mathcal B$ consists of $\mathbf u$ and $\nabla {\rho^{\alpha-1}}$, and its
		spatial derivatives consist of $\nabla\mathbf u$ and $\nabla^2{\rho^{\alpha-1}}$.
		The bounds follow from Proposition~\ref{prop:firstorder},
		Lemma~\ref{lem:localh2}, and Sobolev embedding. The extensions
		alter no equation for the compactly supported $\mathbf w$.
		
		\par\medskip\noindent\textit{One derivative of the cutoff forcing.}\quad
		The cutoff forcing $\mathbf f$ from Proposition~\ref{prop:firstorder} satisfies
		\begin{equation}\label{eq:core:5.5}
			\|\mathbf f(t)\|_{H^1}
			\le C_{T,R}\left(
			1+\|\mathbf u(t)\|_{H^2(B_{3R})}
			+\|\mathbf u(t)\|_{L^\infty(B_{3R})}
			\right).
		\end{equation}
		To verify this estimate, first control the pressure term by
		\eqref{eq:core:4.3}. The differentiated first-order remainder is a sum of fixed-cutoff
		multiples of $\nabla\mathcal B\,\mathbf u$,
		$\mathcal B\nabla\mathbf u$, and $\mathcal B\mathbf u$. Here
		\[
		\begin{aligned}
			\|\nabla\mathcal B\,\mathbf u\|_2
			&\le\|\nabla\mathcal B\|_2\|\mathbf u\|_\infty,\\
			\|\mathcal B\nabla\mathbf u\|_2
			&\le\|\mathcal B\|_6\|\nabla\mathbf u\|_3
			\le C_{T,R}\|\mathbf u\|_{H^2(B_{3R})}.
		\end{aligned}
		\]
		The differentiated diffusion remainder contains
		${\rho^{\alpha-1}}\nabla^2\mathbf u$, $\nabla {\rho^{\alpha-1}}\,\nabla\mathbf u$,
		${\rho^{\alpha-1}}\nabla\mathbf u$, and lower terms. In particular,
		\[
		\begin{aligned}
			\|\rho^{\alpha-1}\nabla^2\mathbf u\|_2
			&\le C_{T,R}\|\mathbf u\|_{H^2(B_{3R})},\\
			\|\nabla\rho^{\alpha-1}\,\nabla\mathbf u\|_2
			&\le\|\nabla\rho^{\alpha-1}\|_6\|\nabla\mathbf u\|_3
			\le C_{T,R}\|\mathbf u\|_{H^2(B_{3R})}.
		\end{aligned}
		\]
		The remaining products use the uniform local $H^1$ bound.
		Proposition~\ref{prop:firstorder} controls both velocity norms on the
		right-hand side in $L^2(0,T)$, so
		\begin{equation}\label{eq:core:5.6}
			\int_0^T\|\mathbf f(t)\|_{H^1}^2\,dt\le C_{T,R}.
		\end{equation}
		No new boundary derivative estimate has been assumed.
		
		\par\medskip\noindent\textit{Differentiate once and test by minus the Laplacian.}\quad
		Let $\mathbf w_j=\partial_j\mathbf w$. Differentiating \eqref{eq:cont:4.4} gives
		\begin{equation}\label{eq:core:5.7}
			(\mathbf w_j)_t-\alpha {\rho^{\alpha-1}}\Delta\mathbf w_j
			=\alpha {\partial_j \rho^{\alpha-1}}\Delta\mathbf w+
			\mathcal B\nabla\mathbf w_j+
			(\partial_j\mathcal B)\nabla\mathbf w+\partial_j\mathbf f.
		\end{equation}
		Multiply by $-\Delta\mathbf w_j$, integrate over all space, and sum in $j$.
		Compact support eliminates spatial boundary terms. Hence
		\begin{equation}\label{eq:core:5.8}
			\begin{aligned}
				\frac{1}{2}\frac{d}{dt}\|\nabla^2\mathbf w\|_2^2
				+\nu_{T,R}\|\nabla\Delta\mathbf w\|_2^2
				\le{}&
				C_\alpha\bigl(\|\nabla {\rho^{\alpha-1}}\|_6+\|\mathcal B\|_6\bigr)
				\|\nabla^2\mathbf w\|_3\|\nabla\Delta\mathbf w\|_2\\
				&+C\|\nabla\mathcal B\|_2\|\nabla\mathbf w\|_\infty
				\|\nabla\Delta\mathbf w\|_2\\
				&+\|\nabla\mathbf f\|_2\|\nabla\Delta\mathbf w\|_2.
			\end{aligned}
		\end{equation}
		To estimate the right-hand side, we will use the following two
		interpolation inequalities:
		\begin{equation}\label{eq:core:5.9}
			\|\nabla^2\mathbf w\|_3
			\le C\|\mathbf w\|_{H^2}^{\frac{1}{2}}\|\mathbf w\|_{H^3}^{\frac{1}{2}},
			\qquad
			\|\nabla\mathbf w\|_\infty
			\le C\|\mathbf w\|_{H^2}^{\frac{1}{2}}\|\mathbf w\|_{H^3}^{\frac{1}{2}}.
		\end{equation}
		The first follows from $L^2$--$L^6$ interpolation and Sobolev embedding,
		as in Proposition~\ref{prop:firstorder}. We give the proof of the second.
		Let $g$ be a smooth compactly supported scalar function on $\mathbb R^3$.
		For any $N>0$, Fourier inversion gives
		\[
		\|g\|_\infty
		\le C\int_{|\eta|\le N}|\widehat g(\eta)|\,d\eta
		+C\int_{|\eta|>N}|\widehat g(\eta)|\,d\eta.
		\]
		We bound the first integral using one derivative of $g$ and the second
		using two derivatives. By Cauchy--Schwarz and Plancherel,
		\[
		\begin{aligned}
			\int_{|\eta|\le N}|\widehat g(\eta)|\,d\eta
			&\le
			\left(\int_{|\eta|\le N}\frac{d\eta}{|\eta|^2}\right)^{\frac{1}{2}}
			\left(\int_{|\eta|\le N}|\eta|^2|\widehat g(\eta)|^2\,d\eta\right)^{\frac{1}{2}}\\
			&\le C N^{\frac{1}{2}}\|\nabla g\|_2,\\
			\int_{|\eta|>N}|\widehat g(\eta)|\,d\eta
			&\le
			\left(\int_{|\eta|>N}\frac{d\eta}{|\eta|^4}\right)^{\frac{1}{2}}
			\left(\int_{|\eta|>N}|\eta|^4|\widehat g(\eta)|^2\,d\eta\right)^{\frac{1}{2}}\\
			&\le C N^{-\frac{1}{2}}\|\nabla^2g\|_2.
		\end{aligned}
		\]
		Here the integrals of the weights are $4\pi N$ and $\frac{4\pi}{N}$,
		respectively. Combining these bounds yields
		\[
		\|g\|_\infty
		\le C\left(N^{\frac{1}{2}}\|\nabla g\|_2
		+N^{-\frac{1}{2}}\|\nabla^2g\|_2\right).
		\]
		If both derivative norms are nonzero, choose
		$N=\frac{\|\nabla^2g\|_2}{\|\nabla g\|_2}$ so that the two terms are equal.
		Then
		\[
		\|g\|_\infty
		\le C\|\nabla g\|_2^{\frac{1}{2}}\|\nabla^2g\|_2^{\frac{1}{2}}.
		\]
		If either derivative norm vanishes, compact support implies $g=0$, and
		the inequality still holds. Applying it to each component of
		$\nabla\mathbf w$ and combining the component bounds gives
		\[
		\|\nabla\mathbf w\|_\infty
		\le C\|\nabla^2\mathbf w\|_2^{\frac{1}{2}}
		\|\nabla^3\mathbf w\|_2^{\frac{1}{2}}
		\le C\|\mathbf w\|_{H^2}^{\frac{1}{2}}\|\mathbf w\|_{H^3}^{\frac{1}{2}},
		\]
		which proves the second inequality in \eqref{eq:core:5.9}.
		
		Put $N_2=\|\mathbf w\|_{H^2}$ and $N_3=\|\mathbf w\|_{H^3}$. The first
		two terms of \eqref{eq:core:5.8} are bounded by
		\begin{equation}\label{eq:core:5.10}
			C_{T,R}N_2^{\frac{1}{2}}N_3^{\frac{3}{2}}
			\le\varepsilon N_3^2+C_{\varepsilon,T,R}N_2^2.
		\end{equation}
		The forcing term is bounded by $\varepsilon N_3^2+C_\varepsilon\|\mathbf f\|_{H^1}^2.$
		The $H^1$ norm of $\mathbf w$ is already uniformly bounded. Also,
		\[
		N_3^2\le C\left(\|\nabla\Delta\mathbf w\|_2^2+
		\|\mathbf w\|_{H^1}^2\right),\qquad
		N_2^2\le C\left(\|\nabla^2\mathbf w\|_2^2+
		\|\mathbf w\|_{H^1}^2\right).
		\]
		Absorb the third derivatives by choosing $\varepsilon$ sufficiently small.
		Then
		\begin{equation}\label{eq:core:5.11}
			\frac{d}{dt}\|\nabla^2\mathbf w\|_2^2+
			c_{T,R}\|\nabla^3\mathbf w\|_2^2
			\le C_{T,R}(1+\|\nabla^2\mathbf w\|_2^2)+
			C_{T,R}\|\mathbf f\|_{H^1}^2.
		\end{equation}
		The initial local $H^2$ norm is finite. Gronwall and \eqref{eq:core:5.6} give
		\begin{equation}\label{eq:core:5.12}
			\sup_{t<T}\|\mathbf w(t)\|_{H^2}^2+
			\int_0^T\|\mathbf w(t)\|_{H^3}^2\,dt\le C_{T,R}.
		\end{equation}
		First perform the calculation on each time interval $[0,\tau]$, $\tau<T$,
		then let $\tau\uparrow T$ using the uniform bounds.
		The target $\int\|\nabla\mathbf u\|_\infty$ is never used as a coefficient:
		$\nabla\mathcal B$ is bounded in $L^\infty_tL^2_x$, and local ellipticity
		absorbs the highest derivative.
		
		\par\medskip\noindent\textit{Time derivatives and the center-inclusive deformation estimate.}\quad
		Returning to \eqref{eq:cont:4.4}, we have
		\[
		\|\alpha\rho^{\alpha-1}\Delta\mathbf w\|_{H^1}
		\le C_{T,R}\left(N_3+N_2^{\frac{1}{2}}N_3^{\frac{1}{2}}\right),
		\]
		because its derivative contains $\nabla\rho^{\alpha-1}\,\Delta\mathbf w$,
		estimated in $L^6\cdot L^3$. The drift satisfies
		\[
		\|\mathcal B\nabla\mathbf w\|_{H^1}\le C_{T,R}(N_2+N_2^{\frac{1}{2}}N_3^{\frac{1}{2}}),
		\]
		using $\nabla\mathcal B\in L^2$ with $\nabla\mathbf w\in L^\infty$,
		and $\mathcal B\in L^6$ with $\nabla^2\mathbf w\in L^3$. Both right
		sides are in $L^2_t$ by \eqref{eq:core:5.12}. Together with \eqref{eq:core:5.6}, this proves
		$\mathbf w_t\in L^2_tH^1_x$, and hence the stated time-derivative estimate.
		
		Finally $H^3\subset W^{1,\infty}$ gives
		\[
		\int_0^T\|\nabla\mathbf u(t)\|_{L^\infty(B_R)}^2\,dt \le C\int_0^T\|\mathbf w(t)\|_{H^3}^2\,dt\le C_{T,R}.
		\]
		Cauchy--Schwarz also gives
		\[
		\int_0^T\|(u_r)_-(t)\|_{L^\infty(0,R)}\dd t
		\le\sqrt T\left(\int_0^T\|\nabla\bu(t)\|_{L^\infty(B_R)}^2\dd t\right)^{\frac{1}{2}}
		<\infty.
		\]
		This entire
		parabolic calculation is Cartesian, so it includes the center without
		a singular radial boundary argument.
	\end{proof}
	
	It remains to combine these local estimates with the exterior velocity
	comparisons from Section~\ref{sec:barriers}. Control near the center
	handles the quotients $\frac{u}{r}$ and $\frac{v}{r}$, while the exterior
	amplitude bounds handle large radii. The following corollary collects
	the global bounds used in the subsequent sections.
	
	\begin{corollary}[Global amplitudes and transverse strain rates]\label{cor:amplitudes}
		The following bounds hold:
		\begin{equation}\label{eq:amplitudes}
			\sup_{t<T}\Bigl(\|u\|_\infty+\|v\|_\infty+
			\|\nabla {\rho^{\alpha-1}}\|_\infty+\left\|\frac{v}{r}\right\|_\infty\Bigr)\le C_T,
			\qquad \int_0^T\left\|\frac{u}{r}\right\|_\infty^2\dd t\le C_T.
		\end{equation}
		In particular, the nonnegative source $Q$ defined in \eqref{eq:core:2.1}
		is bounded uniformly in time and space.
	\end{corollary}
	\begin{proof}
		On $B_1$, Proposition~\ref{prop:coredeformation} and $H^2\subset L^\infty$ give a uniform velocity bound.
		For $r\ge1$, Lemmas~\ref{lem:records} and~\ref{lem:inward} already bound the amplitude. Consequently $\sup_{t<T}\|u(t)\|_\infty\le C_T.$
		The identity \eqref{eq:core:4.1} and bounded effective velocity imply
		\[
		\sup_{t<T}\|\nabla {\rho^{\alpha-1}}(t)\|_\infty \le\frac{\alpha-1}{\alpha}\sup_{t<T} (\|\mathbf v(t)\|_\infty+\|\mathbf u(t)\|_\infty)\le C_T.
		\]
		The local $V$ bound in \eqref{eq:core:3.6}, together with
		$|V|\le\frac{C_T}{r}$ for $r\ge1$, gives $\sup_{t<T}\|V(t)\|_\infty\le C_T.$
		For $r\le1$, radial averaging gives
		\[
		|U(t,r)|\le\frac{1}{r}\int_0^r|u_r(t,s)|\,ds
		\le\|\nabla\mathbf u(t)\|_{L^\infty(B_1)}.
		\]
		For $r\ge1$, $|U|\le\|u(t)\|_\infty$. Thus
		\begin{equation}\label{eq:core:6.1}
			\int_0^T\|U(t)\|_\infty^2\,dt
			\le\int_0^T\|\nabla\mathbf u(t)\|_{L^\infty(B_1)}^2\,dt+
			T\sup_{t<T}\|u(t)\|_\infty^2\le C_T.
		\end{equation}
		In particular, the spatial supremum of the negative part of the physical
		transverse strain rate is now time-integrable, not merely controlled along material paths. The global source $Q$ in \eqref{eq:core:2.1} is also uniformly bounded in time, since ${(v-u)}$ is now uniformly bounded.
	\end{proof}
	
	The estimates above control velocity amplitudes globally and deformation
	on each fixed ball. We next study the radial derivatives in the exterior
	through their conservative equations, with the aim of bounding their total
	variation and proving that its tails vanish uniformly.
	
	\section{Exterior variation bounds and vanishing tails}\label{sec:variation}
	Our goal is to bound the total radial variation of $u$ and $v$ on the
	exterior uniformly for $t<T$, and to show that this variation becomes
	arbitrarily small sufficiently far from the origin. We use conservative
	equations for the radial strains to derive a weighted variation inequality,
	first applied on a fixed exterior region and then on increasingly distant
	tails. The resulting tail control will be used in
	Section~\ref{sec:jacobian} to select a material boundary with small time
	variation of the reference-volume Jacobian.
	
	In this section all one-dimensional exterior integrals use the measure
	$\dd r$, without either the mass weight or the spherical volume weight.
	We begin by checking that the initial assumptions provide finite radial
	variation with vanishing tails. The accompanying pointwise and
	H\"older bounds record the initial strain regularity needed in the later
	exterior analysis.
	
	\begin{lemma}[Initial effective strain and radial variation]\label{lem:initialstrain}
		Under \eqref{eq:data}--\eqref{eq:logdata}, $v_{0r}$ is bounded and tends
		to zero at infinity. If $\psi_0=(\log\rho_0)'$, then $\psi_0'$ and
		$v_{0r}$ have uniformly bounded $C^{\frac{1}{4}}$ norms on $[1,\infty)$.
		Moreover, the initial velocities have finite radial variation with
		vanishing tails:
		\[
		\int_1^\infty(|u_{0r}|+|v_{0r}|)\dd r<\infty,
		\qquad
		\int_R^\infty(|u_{0r}|+|v_{0r}|)\dd r\longrightarrow0
		\quad(R\to\infty).
		\]
	\end{lemma}
	\begin{proof}
		For $x\ne0$, write $r=|x|$ and $e=\frac{x}{r}$. Since
		\[
		\partial_i r=e_i,\qquad
		\partial_i e_j=\frac{\delta_{ij}-e_i e_j}{r},\qquad
		\partial_j\log\rho_0=\psi_0 e_j,
		\]
		direct differentiation gives
		\[
		\partial_i\partial_j\log\rho_0
		=\psi_0'e_i e_j+\frac{\psi_0}{r}(\delta_{ij}-e_i e_j),
		\qquad
		\nabla^2\log\rho_0
		=\psi_0'e\otimes e+\frac{\psi_0}{r}(I-e\otimes e).
		\]
		The projections $e\otimes e$ and $I-e\otimes e$ act on the radial
		direction and its two-dimensional orthogonal complement, respectively.
		Thus the radial eigenvalue is $\psi_0'$ and both tangential eigenvalues
		are $\frac{\psi_0}{r}$. In particular,
		\begin{equation}\label{eq:var:2.1}
			\|\nabla^2\log\rho_0\|_{L^2(\R^3)}^2
			=4\pi\int_0^\infty
			\left(r^2|\psi_0'|^2+2|\psi_0|^2\right)\dd r.
		\end{equation}
		Along a fixed ray $x=re$, the unit vector $e$ is constant, so
		\[
		\psi_0''(r)
		=\frac{\mathrm d^3}{\mathrm d r^3}\log\rho_0(re)
		=\sum_{i,j,k=1}^3 e_i e_j e_k
		\partial_i\partial_j\partial_k\log\rho_0(re),
		\qquad
		|\psi_0''(r)|\le|\nabla^3\log\rho_0(re)|.
		\]
		The assumptions \eqref{eq:logdata} therefore imply
		\[
		\int_1^\infty r^2\left(|\psi_0'|^2+|\psi_0''|^2\right)\dd r
		\le\frac{1}{4\pi}
		\left(\|\nabla^2\log\rho_0\|_2^2
		+\|\nabla^3\log\rho_0\|_2^2\right)<\infty.
		\]
		
		The same radial contraction controls the derivatives of $u_0$.
		Indeed, $u_0(r)=e\cdot\bu_0(re)$ on a fixed ray, and hence
		\[
		|\partial_r^k u_0(r)|
		=\left|e\cdot(e\cdot\nabla)^k\bu_0(re)\right|
		\le|\nabla^k\bu_0(re)|,\qquad k=1,2.
		\]
		Integrating in spherical coordinates and using $\bu_0\in H^3$ gives
		\[
		\int_1^\infty r^2\left(|u_{0r}|^2+|u_{0rr}|^2\right)\dd r
		\le\frac{1}{4\pi}
		\left(\|\nabla\bu_0\|_2^2+\|\nabla^2\bu_0\|_2^2\right)<\infty.
		\]
		Since $r\ge1$, these two estimates show that both $\psi_0'$ and
		$u_{0r}$ belong to $H^1(1,\infty)$ with respect to $\dd r$.
		
		For either $h=\psi_0'$ or $h=u_{0r}$, membership in $L^2(1,\infty)$
		provides a sequence $R_j\to\infty$ with $h(R_j)\to0$.
		Integrating $(h^2)'=2hh'$ from $r$ to $R_j$ and then letting
		$j\to\infty$, we obtain
		\[
		\begin{aligned}
			|h(r)|^2
			&=-2\int_r^\infty h(s)h'(s)\dd s\\
			&\le2\|h\|_{L^2(r,\infty)}\|h'\|_{L^2(r,\infty)}
			\longrightarrow0\qquad(r\to\infty).
		\end{aligned}
		\]
		In particular both functions are bounded and tend to zero at infinity.
		For $r,s\ge1$, the fundamental theorem of calculus also gives
		\[
		|h(r)-h(s)|
		\le\|h'\|_{L^2(1,\infty)}|r-s|^{\frac{1}{2}}.
		\]
		When $|r-s|\le1$, this implies a $C^{\frac{1}{4}}$ bound; when
		$|r-s|\ge1$, use
		\[
		|h(r)-h(s)|\le2\|h\|_{L^\infty(1,\infty)}
		\le2\|h\|_{L^\infty(1,\infty)}|r-s|^{\frac{1}{4}}.
		\]
		Thus $\psi_0'$ and $u_{0r}$ have bounded $C^{\frac{1}{4}}$ norms
		on the entire exterior interval.
		
		We next transfer these bounds to $v_{0r}$. By \eqref{eq:logdata},
		$|\psi_0|\le\|\nabla\log\rho_0\|_\infty$; the bound for $\psi_0'$
		also makes $\psi_0$ Lipschitz on $[1,\infty)$. Moreover,
		\[
		(\rho_0^{\alpha-1})'
		=(\alpha-1)\rho_0^{\alpha-1}\psi_0,
		\]
		so $\rho_0^{\alpha-1}$ is bounded and Lipschitz there, and tends
		to zero because $\alpha>1$ and $\rho_0\to0$.
		Differentiating \eqref{eq:effective} at time zero now yields
		\begin{equation}\label{eq:var:2.3}
			\begin{aligned}
				v_0&=u_0+\alpha\rho_0^{\alpha-1}\psi_0,\\
				v_{0r}&=u_{0r}+\alpha\rho_0^{\alpha-1}\psi_0'
				+\alpha(\alpha-1)\rho_0^{\alpha-1}\psi_0^2.
			\end{aligned}
		\end{equation}
		Consequently, for a constant $C$ depending only on the initial data,
		\[
		|v_{0r}(r)|\le|u_{0r}(r)|+C\rho_0^{\alpha-1}(r),\qquad r\ge1.
		\]
		The right-hand side is bounded and tends to zero as $r\to\infty$,
		which proves both assertions for $v_{0r}$ on the exterior.
		The bounded Lipschitz factors in its formula also give, for
		$r,s\ge1$ with $|r-s|\le1$,
		\[
		\begin{aligned}
			|v_{0r}(r)-v_{0r}(s)|
			&\le|u_{0r}(r)-u_{0r}(s)|
			+C|\psi_0'(r)-\psi_0'(s)|+C|r-s|\\
			&\le C|r-s|^{\frac{1}{4}}.
		\end{aligned}
		\]
		For $|r-s|\ge1$, boundedness of $v_{0r}$ gives the same estimate
		after enlarging $C$. Smoothness and positivity of $\rho_0$
		on the closed unit ball make $\bv_0$ smooth there by
		\eqref{eq:effective}; its radial derivative $v_{0r}$ therefore has
		a bounded smooth center extension. This proves the global bound.
		
		Finally, the weighted estimates above also imply finite radial variation.
		For either $h=\psi_0'$ or $h=u_{0r}$ and every $R\ge1$,
		Cauchy--Schwarz gives
		\begin{equation}\label{eq:var:2.2}
			\begin{aligned}
				\int_R^\infty |h|\dd r
				&\le\left(\int_R^\infty r^2|h|^2\dd r\right)^{\frac{1}{2}}
				\left(\int_R^\infty\frac{1}{r^2}\dd r\right)^{\frac{1}{2}}\\
				&=\frac{1}{\sqrt R}\|rh\|_{L^2(R,\infty)}.
			\end{aligned}
		\end{equation}
		Also, \eqref{eq:var:2.1} gives $\psi_0\in L^2(1,\infty)$.
		Since $\rho_0^{\alpha-1}$ is bounded, the last two terms in
		\eqref{eq:var:2.3} satisfy
		\[
		\begin{aligned}
			\int_R^\infty\rho_0^{\alpha-1}|\psi_0'|\dd r
			&\le\frac{\|\rho_0^{\alpha-1}\|_\infty}{\sqrt R}
			\|r\psi_0'\|_{L^2(R,\infty)},\\
			\int_R^\infty\rho_0^{\alpha-1}|\psi_0|^2\dd r
			&\le\|\rho_0^{\alpha-1}\|_\infty\|\psi_0\|_{L^2(R,\infty)}^2.
		\end{aligned}
		\]
		Combining these estimates yields
		\[
		\begin{aligned}
			\int_R^\infty(|u_{0r}|+|v_{0r}|)\dd r
			&\le\frac{2}{\sqrt R}\|ru_{0r}\|_{L^2(R,\infty)}\\
			&\quad+\frac{\alpha\|\rho_0^{\alpha-1}\|_\infty}{\sqrt R}
			\|r\psi_0'\|_{L^2(R,\infty)}\\
			&\quad+\alpha(\alpha-1)\|\rho_0^{\alpha-1}\|_\infty
			\|\psi_0\|_{L^2(R,\infty)}^2.
		\end{aligned}
		\]
		The right-hand side is finite at $R=1$ and tends to zero as
		$R\to\infty$, proving both remaining assertions.
	\end{proof}
	
	To propagate the initial radial variation, we write the equations for
	$u_r$ and $v_r$ in conservative form and estimate their source terms.
	Let $\mathscr V=|u_r|+|v_r|$. Corollary~\ref{cor:amplitudes} bounds
	$u$, $v$, and $\rho$ uniformly, and the density-tail estimate gives
	\begin{equation}\label{eq:var:1.2}
		F_{1,R}:=\sup_{t<T,\,r\ge R}F_1(t,r)
		\le C_TR^{-\frac{2(\gamma-\alpha)}{\alpha}}\longrightarrow0
		\qquad(R\to\infty).
	\end{equation}
	
	\begin{lemma}[Conservative derivative equations]\label{lem:conservativestrain}
		Define the pressure force and the geometric source by
		\[
		\mathcal P=F_1(v-u),\qquad
		\mathcal G=-\frac{2\alpha {\rho^{\alpha-1}}u}{r^2}
		+\frac{2(\alpha-1)(v-u)u}{r}.
		\]
		Then the physical and effective velocity derivatives satisfy
		\[
		\begin{aligned}
			u_{rt}+\left[\left(2u-v-\frac{2\alpha {\rho^{\alpha-1}}}{r}\right)u_r\right]_r
			&=(\alpha {\rho^{\alpha-1}}u_r)_{rr}+\mathcal G_r-\mathcal P_r,\\
			v_{rt}+(uv_r)_r&=-\mathcal P_r.
		\end{aligned}
		\]
		For every nonnegative smooth compactly supported weight $\chi$ in $(0,\infty)$,
		\[
		\int\chi|\mathcal P_r|\dd r\le2\int\chi F_1\mathscr V\dd r
		+\int|\chi'||\mathcal P|\dd r.
		\]
		Moreover
		\[
		|\mathcal G_r|\le C_T\left(\frac{1}{r}+\frac{1}{r^2}\right)\mathscr V
		+C_T\left(\frac{1}{r^2}+\frac{1}{r^3}\right)
		\]
		on $r\ge1$.
	\end{lemma}
	\begin{proof}
		The radial momentum and effective velocity equations are
		\[
		\begin{aligned}
			u_t+uu_r
			&=\alpha {\rho^{\alpha-1}}\left(u_{rr}+\frac{2u_r}{r}-\frac{2u}{r^2}\right)\\
			&\quad+(v-u)\left(\alpha u_r+\frac{2(\alpha-1)u}{r}-F_1\right),\\
			v_t+uv_r&=-F_1(v-u).
		\end{aligned}
		\]
		Separating the pressure force and the remaining geometric terms gives
		\begin{equation}\label{eq:var:3.1}
			\mathcal P=F_1(v-u),\qquad
			\mathcal G=-\frac{2\alpha {\rho^{\alpha-1}}u}{r^2}
			+\frac{2(\alpha-1)(v-u)u}{r}.
		\end{equation}
		Thus the physical velocity equation becomes
		\begin{equation}\label{eq:var:3.2}
			u_t=\alpha {\rho^{\alpha-1}}u_{rr}
			+\left(\frac{2\alpha {\rho^{\alpha-1}}}{r}+\alpha(v-u)-u\right)u_r
			+\mathcal G-\mathcal P.
		\end{equation}
		The effective-velocity identity yields
		\[
		(\alpha {\rho^{\alpha-1}})_r
		=\alpha(\alpha-1)\rho^{\alpha-2}\rho_r
		=(\alpha-1)(v-u).
		\]
		Using this identity, we put the derivative terms in conservative form:
		\[
		\begin{aligned}
			&\alpha {\rho^{\alpha-1}}u_{rr}
			+\left(\frac{2\alpha {\rho^{\alpha-1}}}{r}+\alpha(v-u)-u\right)u_r\\
			&\quad=(\alpha {\rho^{\alpha-1}}u_r)_r
			+\left(\frac{2\alpha {\rho^{\alpha-1}}}{r}+\alpha(v-u)-u-(\alpha-1)(v-u)\right)u_r\\
			&\quad=(\alpha {\rho^{\alpha-1}}u_r)_r
			-\left(2u-v-\frac{2\alpha {\rho^{\alpha-1}}}{r}\right)u_r.
		\end{aligned}
		\]
		Differentiating both velocity equations in $r$ now gives
		\begin{equation}\label{eq:var:3.3}
			\begin{aligned}
				u_{rt}+\left[\left(2u-v-\frac{2\alpha {\rho^{\alpha-1}}}{r}\right)u_r\right]_r
				&=(\alpha {\rho^{\alpha-1}}u_r)_{rr}+\mathcal G_r-\mathcal P_r,\\
				v_{rt}+(uv_r)_r&=-\mathcal P_r.
			\end{aligned}
		\end{equation}
		The diffusion coefficient $\alpha {\rho^{\alpha-1}}$ and the transport
		coefficients $2u-v-\frac{2\alpha {\rho^{\alpha-1}}}{r}$ and $u$ are bounded
		on $r\ge1$ by Corollary~\ref{cor:amplitudes}.
		
		We next estimate the differentiated geometric source. Using
		$(v-u)_r=v_r-u_r$ and
		$(\alpha {\rho^{\alpha-1}})_r=(\alpha-1)(v-u)$ in \eqref{eq:var:3.1}, we obtain
		\begin{equation}\label{eq:var:3.4}
			\begin{aligned}
				\mathcal G_r
				={}&-\frac{2(\alpha-1)(v-u)u}{r^2}
				-\frac{2\alpha {\rho^{\alpha-1}}u_r}{r^2}
				+\frac{4\alpha {\rho^{\alpha-1}}u}{r^3}\\
				&+\frac{2(\alpha-1)(v_r-u_r)u}{r}
				+\frac{2(\alpha-1)(v-u)u_r}{r}
				-\frac{2(\alpha-1)(v-u)u}{r^2}\\
				={}&\frac{2(\alpha-1)u}{r}v_r
				+\left[\frac{2(\alpha-1)(v-2u)}{r}
				-\frac{2\alpha {\rho^{\alpha-1}}}{r^2}\right]u_r\\
				&+\frac{4\alpha {\rho^{\alpha-1}}u}{r^3}
				-\frac{4(\alpha-1)(v-u)u}{r^2}.
			\end{aligned}
		\end{equation}
		It follows from \eqref{eq:amplitudes}, for $r\ge1$, that
		\begin{equation}\label{eq:var:3.5}
			|\mathcal G_r|\le C_T\left(\frac{1}{r}+\frac{1}{r^2}\right)\mathscr V
			+C_T\left(\frac{1}{r^2}+\frac{1}{r^3}\right).
		\end{equation}
		The last two powers are integrable in $dr$ at infinity. All constants
		here are independent of $\tau<T$.
		
		Finally, the pressure force is itself the derivative of the enthalpy:
		\[
		\mathcal P=\frac{\gamma}{\alpha}\rho^{\gamma-\alpha}
		\alpha\rho^{\alpha-2}\rho_r
		=\gamma\rho^{\gamma-2}\rho_r
		=\left(\frac{\gamma}{\gamma-1}\rho^{\gamma-1}\right)_r.
		\]
		More usefully for variation,
		\begin{equation}\label{eq:var:3.6}
			\begin{aligned}
				\mathcal P_r&=F_1(v_r-u_r)+Q,\\
				Q&=(\partial_r F_1)(v-u)
				=\frac{\gamma(\gamma-\alpha)}{\alpha^2}\rho^{\gamma-2\alpha+1}(v-u)^2\ge0.
			\end{aligned}
		\end{equation}
		For every nonnegative smooth compactly supported weight $\chi$ in
		$(0,\infty)$, integration by parts gives
		\begin{equation}\label{eq:var:3.7}
			\begin{aligned}
				\int\chi Q\,dr
				&=\int\chi\mathcal P_r\,dr-\int\chi F_1({v_r}-{u_r})\,dr\\
				&=-\int\chi'\mathcal P\,dr+\int\chi F_1({u_r}-{v_r})\,dr\\
				&\le\int|\chi'||\mathcal P|\,dr+\int\chi F_1\mathscr V\,dr.
			\end{aligned}
		\end{equation}
		Since $Q\ge0$, \eqref{eq:var:3.6}--\eqref{eq:var:3.7} imply
		\begin{equation}\label{eq:var:3.8}
			\int\chi|\mathcal P_r|\,dr
			\le2\int\chi F_1\mathscr V\,dr+
			\int|\chi'||\mathcal P|\,dr.
		\end{equation}
		This is the key pressure estimate. Directly treating $Q$ as an arbitrary
		bounded function would give an infinite integral on the half-line.
		Equation \eqref{eq:var:3.7} avoids that loss. It contains no
		$\frac{Q}{\rho}$, $\frac{1}{\rho^{\alpha-1}}$, or
		initial density inverse.
	\end{proof}
	
	The conservative equations now allow us to estimate the absolute values
	of both strains. Testing against a nonnegative weight transfers the
	transport and diffusion derivatives onto that weight, while the pressure
	estimate in Lemma~\ref{lem:conservativestrain} controls the pressure
	contribution by the same weighted variation and a cutoff term. The next
	inequality is the common starting point for both exterior estimates below.
	
	\begin{lemma}[Variation inequality]\label{lem:kato}
		For every nonnegative compactly supported smooth weight $\chi$,
		\begin{align*}
			\frac{\dd}{\dd t}\int\chi\mathscr V\dd r
			\le{}&\int\alpha {\rho^{\alpha-1}}\chi''|{u_r}|\dd r+
			\int\chi'\left[\left(2u-v-\frac{2\alpha {\rho^{\alpha-1}}}{r}\right)|{u_r}|+u|{v_r}|\right]\dd r+\int\chi|\mathcal G_r|\dd r\\
			&+4\int\chi F_1\mathscr V\dd r+2\int|\chi'||\mathcal P|\dd r.
		\end{align*}
	\end{lemma}
	\begin{proof}
		On either open sign set of ${u_r}$, multiplication of \eqref{eq:var:3.3} by its sign
		gives
		\[
		\partial_t|{u_r}|+\partial_r\left[\left(2u-v-\frac{2\alpha {\rho^{\alpha-1}}}{r}\right)|{u_r}|\right] =\partial_{rr}(\alpha {\rho^{\alpha-1}}|{u_r}|) +\operatorname{sgn}({u_r})(\mathcal G_r-\mathcal P_r).
		\]
		At a zero of ${u_r}$, convexity produces a nonpositive diffusion defect
		in the inequality for $|{u_r}|$. This can be justified with
		$\eta_\varepsilon(s)=\sqrt{s^2+\varepsilon^2}$ on a compact spacetime
		set, then $\varepsilon\downarrow0$. In detail, writing
		\[
		\begin{aligned}
			u_{rt}={}&\alpha {\rho^{\alpha-1}}u_{rrr}
			+\left[2(\alpha {\rho^{\alpha-1}})_r
			-\left(2u-v-\frac{2\alpha {\rho^{\alpha-1}}}{r}\right)\right]u_{rr}\\
			&+\left[(\alpha {\rho^{\alpha-1}})_{rr}
			-\left(2u-v-\frac{2\alpha {\rho^{\alpha-1}}}{r}\right)_r\right]u_r
			+\mathcal G_r-\mathcal P_r,
		\end{aligned}
		\]
		the difference between the equation
		for $\eta_\varepsilon({u_r})$ and its conservative linear operator is
		\[
		\begin{aligned}
			&-\alpha {\rho^{\alpha-1}}\eta_\varepsilon''(u_r)u_{rr}^2\\
			&+\left[(\alpha {\rho^{\alpha-1}})_{rr}
			-\left(2u-v-\frac{2\alpha {\rho^{\alpha-1}}}{r}\right)_r\right]
			\bigl(u_r\eta_\varepsilon'(u_r)-\eta_\varepsilon(u_r)\bigr)\\
			&+\eta_\varepsilon'(u_r)(\mathcal G_r-\mathcal P_r).
		\end{aligned}
		\]
		The first term is nonpositive. The middle factor has absolute value
		at most $\varepsilon$, so the middle term tends to zero on the compact
		set. All these coefficients are smooth there. Thus, distributionally,
		\begin{equation}\label{eq:var:4.1}
			\partial_t|{u_r}|+\partial_r\left[\left(2u-v-\frac{2\alpha {\rho^{\alpha-1}}}{r}\right)|{u_r}|\right]
			\le\partial_{rr}(\alpha {\rho^{\alpha-1}}|{u_r}|)+|\mathcal G_r|+|\mathcal P_r|.
		\end{equation}
		The effective transport equation similarly gives
		\begin{equation}\label{eq:var:4.2}
			\partial_t|{v_r}|+\partial_r(u|{v_r}|)\le|\mathcal P_r|.
		\end{equation}
		There is no diffusion defect in \eqref{eq:var:4.2}. The transport regularization error
		is a bounded compact-set multiple of
		$\eta_\varepsilon({v_r})-{v_r}\eta_\varepsilon'({v_r})$ and vanishes.
		
		Adding, testing by $\chi\ge0$, and integrating yields
		\[
		\begin{aligned}
			\frac{d}{dt}\int\chi\mathscr V\,dr
			\le{}&\int\alpha {\rho^{\alpha-1}}\chi''|{u_r}|\,dr
			+\int\chi'\left[\left(2u-v-\frac{2\alpha {\rho^{\alpha-1}}}{r}\right)|{u_r}|+u|{v_r}|\right]\,dr\\
			&+\int\chi|\mathcal G_r|\,dr+2\int\chi|\mathcal P_r|\,dr.
		\end{aligned}
		\]
		Apply \eqref{eq:var:3.8}:
		\begin{equation}\label{eq:var:4.3}
			\begin{aligned}
				\frac{d}{dt}\int\chi\mathscr V\,dr
				\le{}&\int\alpha {\rho^{\alpha-1}}\chi''|{u_r}|\,dr
				+\int\chi'\left[\left(2u-v-\frac{2\alpha {\rho^{\alpha-1}}}{r}\right)|{u_r}|+u|{v_r}|\right]\,dr\\
				&+\int\chi|\mathcal G_r|\,dr
				+4\int\chi F_1\mathscr V\,dr
				+2\int|\chi'||\mathcal P|\,dr.
			\end{aligned}
		\end{equation}
	\end{proof}
	
	We first apply the weighted inequality to bound the total variation on
	the fixed exterior interval. The local estimates from
	Section~\ref{sec:core} control the inner cutoff terms, and the regularity
	of the branch on each interval $[0,\tau]$, $\tau<T$, permits removal
	of the outer cutoff. Together with the finite initial variation, these bounds give
	an estimate uniform up to $T$.
	
	\begin{proposition}[Uniform exterior variation]\label{prop:variation}
		There exists $C_T<\infty$, depending only on $T$, $\alpha,\gamma$,
		and the initial data, such that
		\[
		\sup_{t<T}\int_1^\infty(|u_r|+|v_r|)\dd r\le C_T.
		\]
	\end{proposition}
	\begin{proof}
		Choose a smooth nondecreasing $\chi$ with $0\le\chi\le1$, $\chi=0$
		for $r\le1$, and $\chi=1$ for $r\ge2$. For $N>3$, choose a smooth
		cutoff $\eta_N$ with $0\le\eta_N\le1$, $\eta_N=1$ for $r\le N$,
		$\eta_N=0$ for $r\ge2N$, and
		\[
		|\eta_N'|\le\frac{C}{N},\qquad |\eta_N''|\le\frac{C}{N^2}.
		\]
		Define the truncated variation by
		\[
		E_N(t)=\int_0^\infty\chi\eta_N\mathscr V(t,r)\,dr.
		\]
		The derivatives of $\chi$ are supported in $[1,2]$, where $\eta_N=1$,
		and those of $\eta_N$ are supported in $[N,2N]$, where $\chi=1$.
		In particular, $\chi'\eta_N'=0$, so the product rule gives
		\[
		\begin{aligned}
			(\chi\eta_N)'&=\chi'\eta_N+\chi\eta_N'=\chi'+\eta_N',\\
			(\chi\eta_N)''&=\chi''\eta_N+2\chi'\eta_N'+\chi\eta_N''
			=\chi''+\eta_N'',\\
			|(\chi\eta_N)'|&=|\chi'|+|\eta_N'|.
		\end{aligned}
		\]
		The last equality uses the disjoint derivative supports. Substituting
		$\chi\eta_N$ for the weight in \eqref{eq:var:4.3} therefore yields
		\[
		\begin{aligned}
			E_N'(t)
			\le{}&\int\alpha\rho^{\alpha-1}\chi''|u_r|\,dr
			+\int\alpha\rho^{\alpha-1}\eta_N''|u_r|\,dr\\
			&+\int\chi'\left[\left(2u-v-\frac{2\alpha\rho^{\alpha-1}}{r}\right)|u_r|
			+u|v_r|\right]\,dr\\
			&+\int\eta_N'\left[\left(2u-v-\frac{2\alpha\rho^{\alpha-1}}{r}\right)|u_r|
			+u|v_r|\right]\,dr\\
			&+\int\chi\eta_N|\mathcal G_r|\,dr
			+4\int\chi\eta_N F_1\mathscr V\,dr\\
			&+2\int|\chi'||\mathcal P|\,dr
			+2\int|\eta_N'||\mathcal P|\,dr.
		\end{aligned}
		\]
		We first estimate the outer cutoff terms, which contain derivatives of
		$\eta_N$, and then the inner cutoff terms, which contain derivatives of $\chi$.
		
		Fix $\tau<T$. The branch has bounded ${u_r}$ on $[0,\tau]$. It also has
		bounded ${v_r}$: the characteristic equation $\Dt {v_r} =(F_1-{v_r}){u_r}-F_1{v_r}-Q$
		has bounded coefficients and source on $[0,\tau]$, and bounded initial ${v_r}$.
		This assertion uses the finite norms on $[0,\tau]$, not a bound uniform up to $T$.
		
		The diffusion error at the outer cutoff is at most
		\[
		\frac{C}{N^2}\int_N^{2N}\alpha {\rho^{\alpha-1}}|{u_r}|\,dr
		\le\frac{C_\tau}{N}\to0.
		\]
		The outer transport errors are at most
		\[
		C\sup_{\substack{t\le\tau\\r\ge N}}\left(\left|2u-v-\frac{2\alpha {\rho^{\alpha-1}}}{r}\right|+|u|\right)
		\sup_{t\le\tau,r\ge0}(|{u_r}|+|{v_r}|)\to0.
		\]
		Here $2u-v-\frac{2\alpha {\rho^{\alpha-1}}}{r}\to0$ uniformly on the fixed interval. The pressure
		cutoff error is at most
		\[
		C\sup_{\substack{t\le\tau\\r\ge N}}|\mathcal P|\to0.
		\]
		All errors remain vanishing after integration over $[0,\tau]$.
		This argument does not assume variation integrability in order to
		discard the outer cutoff.
		
		The inner cutoff terms have support in $1\le r\le2$. Uniform local
		$H^2$ control of $\mathbf u$ bounds $\int_1^2|{u_r}|\,dr$, and uniform
		local $W^{1,\infty}$ control of $\mathbf v$ bounds the analogous
		${v_r}$ integral. The coefficients and $\mathcal P$ are uniformly bounded by
		\eqref{eq:amplitudes}. Therefore the sum of inner cutoff errors is at most $C_T$.
		
		Equation \eqref{eq:var:3.5} gives
		\[
		\int\chi\eta_N|\mathcal G_r|\,dr \le C_T\int\chi\eta_N\mathscr V\,dr+ C_T\int_1^\infty(r^{-2}+r^{-3})\,dr \le C_TE_N+C_T,
		\]
		Also,
		$4\int\chi\eta_N F_1\mathscr V\le4\bar F_1 E_N$. Thus the integrated
		inequality and Gronwall give, after the vanishing outer errors,
		\[
		E_N(t)\le e^{C_Tt}\bigl(E_N(0)+C_Tt+o_{N\to\infty}(1)\bigr)
		\quad(0\le t\le\tau).
		\]
		The initial variation is finite by Lemma~\ref{lem:initialstrain}. Fatou's lemma yields
		\begin{equation}\label{eq:var:5.1}
			\int\chi\mathscr V(t,r)\,dr
			\le e^{C_TT}\left(\int\chi\mathscr V(0,r)\,dr+C_TT\right).
		\end{equation}
		The constants left in \eqref{eq:var:5.1} are independent of $\tau$. Letting
		$\tau\uparrow T$ and adding the uniformly controlled annulus $[1,2]$
		gives, after enlarging $C_T$,
		\begin{equation}\label{eq:var:5.2}
			\sup_{t<T}\int_1^\infty\mathscr V(t,r)\,dr\le C_T.
		\end{equation}
		Once \eqref{eq:var:5.2} is known, limits of the weighted identities and inequalities
		can also be justified directly by dominated convergence in space-time.
	\end{proof}
	
	The uniform bound just obtained controls the total amount of exterior
	variation. We next show that the contribution from distant radii tends
	to zero uniformly in time. Moving the inner cutoff outward makes the
	cutoff errors and source contributions small, using
	Proposition~\ref{prop:variation} and the decay of $F_1$.
	The same pressure estimate then gives vanishing tails for $Q$ and
	the variation of the pressure force.
	
	\begin{proposition}[Tight variation tails]\label{prop:variationtails}
		The radial variation is uniformly tight:
		\begin{equation}\label{eq:variationtails}
			m_T(R):=\sup_{t<T}\int_R^\infty(|u_r|+|v_r|)\dd r
			\longrightarrow0\quad(R\to\infty).
		\end{equation}
		The exterior integrals of $Q$ and $|\mathcal P_r|$ also tend to zero uniformly
		in time. In particular $|u(t,r)|+|v(t,r)|\le m_T(r)$.
	\end{proposition}
	\begin{proof}
		Use the smooth nondecreasing cutoff $\chi$ from
		Proposition~\ref{prop:variation}, satisfying
		\[
		0\le\chi\le1,\qquad
		\chi(s)=0\quad(s\le1),\qquad
		\chi(s)=1\quad(s\ge2).
		\]
		For $R\ge2$, define
		\[
		\chi_R(r)=\chi\left(\frac{r}{R}\right).
		\]
		Then $\chi_R=0$ for $r\le R$ and $\chi_R=1$ for $r\ge2R$.
		Set
		\[
		E_R(t)=\int\chi_R\mathscr V\,dr,\qquad
		\mathcal P_R=\sup_{t<T,r\ge R}|\mathcal P(t,r)|.
		\]
		The derivatives of $\chi_R$ are supported in $[R,2R]$ and satisfy
		\[
		|\chi_R'|\le\frac{C}{R},\qquad
		|\chi_R''|\le\frac{C}{R^2}.
		\]
		Substitute $\chi_R$ for $\chi$ in \eqref{eq:var:4.3}, using the cutoff
		limit justified at the end of Proposition~\ref{prop:variation} to allow
		this weight, which equals one at infinity. We obtain
		\[
		\begin{aligned}
			E_R'(t)
			\le{}&\int\alpha\rho^{\alpha-1}\chi_R''|u_r|\,dr\\
			&+\int\chi_R'\left[\left(2u-v-\frac{2\alpha\rho^{\alpha-1}}{r}\right)|u_r|
			+u|v_r|\right]\,dr\\
			&+\int\chi_R|\mathcal G_r|\,dr
			+4\int\chi_R F_1\mathscr V\,dr
			+2\int|\chi_R'||\mathcal P|\,dr.
		\end{aligned}
		\]
		We estimate each term on the right. The uniform exterior variation
		bound \eqref{eq:var:5.2} and the cutoff derivative bounds give
		\[
		\begin{aligned}
			\left|\int\alpha {\rho^{\alpha-1}}\chi_R''|u_r|\,dr\right|&\le\frac{C_T}{R^2},\\
			\left|\int\chi_R'\left[\left(2u-v-\frac{2\alpha {\rho^{\alpha-1}}}{r}\right)|u_r|+u|v_r|\right]\,dr\right|&\le\frac{C_T}{R},\\
			\int|\chi_R'||\mathcal P|\,dr&\le C \mathcal P_R.
		\end{aligned}
		\]
		For the geometric source, \eqref{eq:var:3.5} and the fact that
		$\chi_R=0$ for $r\le R$ give
		\[
		\begin{aligned}
			\int\chi_R|\mathcal G_r|\,dr
			&\le C_T\left(\frac{1}{R}+\frac{1}{R^2}\right)E_R
			+C_T\int_R^\infty\left(\frac{1}{r^2}+\frac{1}{r^3}\right)\,dr\\
			&=C_T\left(\frac{1}{R}+\frac{1}{R^2}\right)E_R
			+C_T\left(\frac{1}{R}+\frac{1}{2R^2}\right)\\
			&\le C_T\left(\frac{1}{R}+\frac{1}{R^2}\right)E_R
			+C_T\left(\frac{1}{R}+\frac{1}{R^2}\right).
		\end{aligned}
		\]
		Finally, recall $F_{1,R}$ from \eqref{eq:var:1.2}. Since
		$F_1\le F_{1,R}$ wherever $\chi_R\ne0$,
		\[
		4\int\chi_R F_1\mathscr V\,dr\le4F_{1,R}E_R.
		\]
		Substituting these estimates into the inequality for $E_R'$ gives
		\[
		\begin{aligned}
			E_R'(t)
			\le{}&\frac{C_T}{R^2}+\frac{C_T}{R}
			+C_T\left(\frac{1}{R}+\frac{1}{R^2}\right)E_R(t)\\
			&+C_T\left(\frac{1}{R}+\frac{1}{R^2}\right)
			+4F_{1,R}E_R(t)+2C\mathcal P_R.
		\end{aligned}
		\]
		Collect the terms multiplying $E_R$ and enlarge $C_T$ and $C$ to absorb
		the remaining cutoff contributions. These constants are independent of $R$.
		The resulting inequality is
		\begin{equation}\label{eq:var:6.1}
			E_R'\le\lambda_R E_R+\varepsilon_R,
		\end{equation}
		where
		\begin{equation}\label{eq:var:6.2}
			\begin{aligned}
				\lambda_R&=4F_{1,R}+C_T\left(\frac{1}{R}+\frac{1}{R^2}\right),\\
				\varepsilon_R&=C_T\left(\frac{1}{R}+\frac{1}{R^2}\right)+C\mathcal P_R.
			\end{aligned}
		\end{equation}
		Both quantities tend to zero. In fact \eqref{eq:amplitudes} and \eqref{eq:var:1.2} give
		$\mathcal P_R\le C_TF_{1,R}\le C_TR^{-\frac{2(\gamma-\alpha)}{\alpha}}$; no already proved uniform
		velocity-decay modulus is needed here. Integrating \eqref{eq:var:6.1},
		\begin{equation}\label{eq:var:6.3}
			\sup_{t<T}E_R(t)
			\le e^{\lambda_RT}\bigl(E_R(0)+T\varepsilon_R\bigr)\longrightarrow0.
		\end{equation}
		Because $\chi_R=1$ on $r\ge2R$, this proves the variation-tail assertion. Quantitatively,
		\begin{equation}\label{eq:var:6.4}
			\sup_{t<T}\int_{2R}^\infty\mathscr V\,dr
			\le C_T\left[
			\int_R^\infty\mathscr V(0,r)\,dr+R^{-1}+R^{-\frac{2(\gamma-\alpha)}{\alpha}}
			\right].
		\end{equation}
		The powers on the right are derived estimates; they prescribe no
		pointwise density profile.
		
		There are two immediate consequences. Since $u,v\to0$ at infinity at
		each time $t<T$, the fundamental theorem of calculus gives
		\begin{equation}\label{eq:var:6.5}
			|u(t,r)|+|v(t,r)|
			\le\int_r^\infty\mathscr V(t,s)\,ds.
		\end{equation}
		Thus the new variation argument also recovers uniform far-field decay
		of both amplitudes. Secondly, \eqref{eq:var:3.7}--\eqref{eq:var:3.8} yield
		\begin{equation}\label{eq:var:6.6}
			\int\chi_R Q\,dr\le F_{1,R} E_R+C \mathcal P_R,
			\qquad
			\int\chi_R|\mathcal P_r|\,dr\le2F_{1,R} E_R+C \mathcal P_R.
		\end{equation}
		In particular both the nonnegative pressure square and the spatial
		variation of the pressure force have uniformly vanishing exterior
		integrals. Their time-integrated tails vanish as well, by multiplication
		of these uniform estimates by $T$.
	\end{proof}
	
	We have now obtained uniform exterior variation bounds and vanishing
	tails, which also control the far-field velocity amplitudes. To use this
	spatial integral control in the deformation argument, we next pass to
	material coordinates. There the variation estimate will allow us to
	choose one distant particle trajectory along which the reference-volume
	Jacobian stays close to its initial value throughout the time interval.
	
	\section{Two-sided bounds for the reference-volume Jacobian}\label{sec:jacobian}
	Our goal is to control the change of volume along exterior particle
	trajectories. The main conclusion, Proposition~\ref{prop:lowerjacobian},
	shows that the reference-volume Jacobian stays bounded above and away
	from zero, and approaches its initial value at infinity uniformly for
	$t<T$. This compares the evolving density with its initial value along
	trajectories and prepares the parabolic estimates in
	Section~\ref{sec:parabolic}. We first establish an upper bound, then use
	vanishing tail estimates and a carefully selected material boundary to
	obtain the lower bound.
	
	Recall the trajectory with initial radius $a$ from the proof of
	Lemma~\ref{lem:energies}:
	\[
	\partial_t r(t,a)=u(t,r(t,a)),\qquad r(0,a)=a.
	\]
	The initial-radius label $a$ is distinct from the mass coordinate $m$.
	In material coordinates, time derivatives are taken at fixed $a$, and
	Eulerian quantities in the formulas below are evaluated at $(t,r(t,a))$.
	Put
	\[
	D_T=T\sup_{t<T}\|u(t)\|_\infty,\qquad
	\psi_0=(\log\rho_0)',\quad
	F_{1,0}=\frac{\gamma}{\alpha}\rho_0^{\gamma-\alpha}.
	\]
	Then $|r(t,a)-a|\le D_T$. Fix $a_*=2D_T+3$, so that the entire material
	exterior $a\ge a_*$ remains away from the center and $a_*$ depends only
	on $T$, $\alpha,\gamma$, and the initial data.
	
	We begin by normalizing the effective strain with respect to the initial
	volume. The associated primitive removes the physical velocity gradient
	from its evolution equation, giving an estimate that can be used before
	any lower bound for the Jacobian is available.
	
	\begin{lemma}[Reference effective strain and its primitive]\label{lem:reference}
		For $a\ge a_*$, define
		\[
		S=\frac{r(t,a)}{a},\qquad
		J=S^2r_a=\frac{\rho_0(a)}{\rho(t,r(t,a))},\qquad \mathcal{R}=Jv_r.
		\]
		Here $S$ is the ratio of the current radius to the initial radius,
		$J$ is the corresponding volume ratio, and $\mathcal{R}$ is the reference
		effective strain. Then $\frac{1}{2}\le S\le\frac{3}{2}$, and
		\[
		|\mathcal{R}(t,a)|\le C_T\left(1+J(t,a)+\int_0^tJ(s,a)\dd s\right).
		\]
		This bound requires no positive lower bound for $J$.
	\end{lemma}
	\begin{proof}
		The displacement estimate and the choice $a_*>2D_T+2$ give
		\begin{equation}\label{eq:jac:3.1}
			|r(t,a)-a|\le D_T,\qquad
			\frac{1}{2}\le S=\frac{r(t,a)}{a}\le\frac{3}{2},\qquad r(t,a)>1
			\quad(a\ge a_*).
		\end{equation}
		The mass identity proved in Lemma~\ref{lem:energies} gives
		$\rho(t,r)r^2r_a=\rho_0(a)a^2$. Hence the definitions and the chain rule yield
		\begin{equation}\label{eq:jac:3.2}
			J=\frac{r^2r_a}{a^2}=S^2r_a=\frac{\rho_0(a)}{\rho(t,r)},
			\qquad
			\mathcal{R}=J{v_r}=S^2\partial_a[v(t,r(t,a))].
		\end{equation}
		Direct calculation at fixed $a$ gives
		\begin{equation}\label{eq:jac:3.3}
			\begin{aligned}
				J_t&=\frac{2ru r_a+r^2u_r r_a}{a^2}=J({u_r}+2U),\\
				r_a&=\frac{J}{S^2},\qquad \frac{S_a}{S}=\frac{J}{S^2r}-\frac{1}{a},\\
				{\rho^{\alpha-1}}&={\rho_0^{\alpha-1}}J^{1-\alpha},\qquad F_1=F_{1,0}J^{\alpha-\gamma},\qquad
				F_{1,0}(a)=\frac{\gamma}{\alpha}\rho_0(a)^{\gamma-\alpha}.
			\end{aligned}
		\end{equation}
		Use the effective derivative equation
		$\Dt  v_r=(F_1-v_r)u_r-F_1v_r-Q$, where $Q$ is defined in \eqref{eq:notation}. Then
		\begin{equation}\label{eq:jac:3.4}
			\begin{aligned}
				\mathcal{R}_t
				&=J({u_r}+2U){v_r}+J[(F_1-{v_r}){u_r}-F_1{v_r}-Q]\\
				&=F_1J{u_r}+(2U-F_1)\mathcal{R}-QJ\\
				&=F_1J_t+(2U-F_1)\mathcal{R}-2UF_1J-QJ.
			\end{aligned}
		\end{equation}
		For each fixed initial radius $a$, define the primitive
		$\Phi_a:(0,\infty)\to\mathbb R$. Its subscript denotes the fixed label,
		not differentiation:
		\begin{equation}\label{eq:jac:3.5}
			\Phi_a(j)=\int_1^j F_{1,0}(a)s^{\alpha-\gamma}\,ds
			=\begin{cases}
				F_{1,0}(a)\frac{j^{1-\gamma+\alpha}-1}{1-\gamma+\alpha},&{(\gamma-\alpha)}\ne1,\\
				F_{1,0}(a)\log j,&{(\gamma-\alpha)}=1.
			\end{cases}
		\end{equation}
		Set $\mathcal{Z}=\mathcal{R}-\Phi_a(J)$. Since $\partial_j\Phi_a(J)=F_1$,
		\begin{equation}\label{eq:jac:3.6}
			\mathcal{Z}_t=(2U-F_1)\mathcal{Z}+(2U-F_1)\Phi_a(J)-2UF_1J-QJ.
		\end{equation}
		Neither ${u_r}$ nor $J_t$ remains. Initially $J=1$ and $\mathcal{Z}=\mathcal{R}={v_{0r}}$,
		without assuming a bound for the initial quotient $\frac{v_{0r}}{\rho_0}$.
		
		This normalization is related to the effective pressure from
		Proposition~\ref{prop:shielding}. The primitive $F_2$ defined in its proof
		satisfies $F_2'(q)=\frac{F_1(q)}{q^2}$. Thus the change of variables
		$q=\frac{\rho_0(a)}{s}$, $ds=-\rho_0(a)q^{-2}\,dq$ gives
		\[
		\begin{aligned}
			\Phi_a(J)
			&=-\rho_0(a)\int_{\rho_0(a)}^{\rho(t,r)}\frac{F_1(q)}{q^2}\,dq\\
			&=-\rho_0(a)[F_2(\rho)-F_2(\rho_0(a))],\\
			\mathcal{Z}&=\rho_0(a)\left[\frac{v_r}{\rho}+F_2(\rho)-F_2(\rho_0(a))\right].
		\end{aligned}
		\]
		
		For $s$ between $1$ and $J$, the density $\frac{\rho_0}{s}$ lies between
		$\rho_0$ and $\frac{\rho_0}{J}=\rho(t,r)$. Both densities are bounded by
		$\bar\rho_T$, so $F_{1,0},F_1\le\bar F_1$, with $\bar F_1$ as defined in
		Section~\ref{sec:barriers}. Hence
		\begin{equation}\label{eq:jac:3.7}
			|\Phi_a(J)|\le\max\{F_{1,0},F_1\}|J-1|
			\le\bar F_1(1+J).
		\end{equation}
		Bounded $U,F_1,Q$, Lemma~\ref{lem:initialstrain}, and the integrating factor for \eqref{eq:jac:3.6} imply
		\begin{equation}\label{eq:jac:3.8}
			\begin{aligned}
				|\mathcal{Z}(t,a)|
				&\le e^{C_Tt}\left(|{v_{0r}}(a)|
				+C_T\int_0^t[1+J(s,a)]\,ds\right),\\
				|\mathcal{R}(t,a)|
				&\le C_T\left(1+J(t,a)+\int_0^tJ(s,a)\,ds\right).
			\end{aligned}
		\end{equation}
		This estimate does not yet assume an upper or lower bound for $J$.
	\end{proof}
	
	The reference-strain estimate depends only on $J$ and its time integral.
	Inserting it into the scalar equation for $J$ will close an upper
	comparison. The resulting bound controls volume expansion and gives a
	lower bound for the density relative to its initial value along each
	exterior trajectory.
	
	\begin{proposition}[Upper bound for the reference Jacobian]\label{prop:upperjacobian}
		There is a constant $C_T<\infty$, such that
		$0<J(t,a)\le C_T$ on $t<T$, $a\ge a_*$. Consequently
		\[
		\rho(t,r(t,a))\ge\frac{\rho_0(a)}{C_T},\qquad
		0<r_a\le4C_T.
		\]
	\end{proposition}
	\begin{proof}
		Use ${\partial_r \rho^{\alpha-1}}=\frac{\alpha-1}{\alpha}(v-u)$ and $\partial_r=\frac{S^2}{J}\partial_a$:
		\begin{equation}\label{eq:jac:4.1}
			\begin{aligned}
				{\partial_a \rho^{\alpha-1}}&={\rho_0^{\alpha-1}}'J^{1-\alpha}-{(\alpha-1)}{\rho_0^{\alpha-1}}J^{-\alpha}J_a,\\
				{\partial_r \rho^{\alpha-1}}&=S^2[{\rho_0^{\alpha-1}}'J^{-\alpha}-{(\alpha-1)}{\rho_0^{\alpha-1}}J^{-\alpha-1}J_a],\\
				u&=v-\frac{\alpha}{\alpha-1} {\partial_r \rho^{\alpha-1}},\\
				J_t&=S^2u_a+2UJ\\
				&=\mathcal{R}+2UJ-\frac{\alpha}{\alpha-1} S^2
				\partial_a\{S^2[{\rho_0^{\alpha-1}}'J^{-\alpha}-{(\alpha-1)}{\rho_0^{\alpha-1}}J^{-\alpha-1}J_a]\}.
			\end{aligned}
		\end{equation}
		Here ${\rho_0^{\alpha-1}}'={(\alpha-1)}{\rho_0^{\alpha-1}}\psi_0$ and ${\rho_0^{\alpha-1}}''={(\alpha-1)}{\rho_0^{\alpha-1}}(\psi_0'+{(\alpha-1)}\psi_0^2)$.
		Using $\frac{S_a}{S}=\frac{J}{S^2r}-\frac{1}{a}$ from
		\eqref{eq:jac:3.3}, the derivative inside braces in \eqref{eq:jac:4.1},
		divided by $\alpha-1$, equals
		\[
		\begin{aligned}
			S^2{\rho_0^{\alpha-1}}\Biggl[&
			J^{-\alpha}\left(\psi_0'+(\alpha-1)\psi_0^2
			+2\left(\frac{J}{S^2r}-\frac{1}{a}\right)\psi_0\right)\\
			&-J^{-\alpha-1}\left((2\alpha-1)\psi_0
			+2\left(\frac{J}{S^2r}-\frac{1}{a}\right)\right)J_a\\
			&+(\alpha+1)J^{-\alpha-2}J_a^2-J^{-\alpha-1}J_{aa}\Biggr].
		\end{aligned}
		\]
		Define the diffusion coefficient $\mathsf D_J$, drift coefficient
		$\mathsf B_J$, and source term $\mathsf F_J$ by
		\begin{equation}\label{eq:jac:coefficients}
			\begin{aligned}
				\mathsf D_J&=\alpha S^4\rho_0^{\alpha-1}J^{-\alpha-1}
				=\frac{\alpha\rho^{\alpha-1}S^4}{J^2},\\
				\mathsf B_J&=\mathsf D_J\left[(2\alpha-1)\psi_0
				+2\left(\frac{J}{S^2r}-\frac{1}{a}\right)\right],\\
				\mathsf F_J&=\mathcal{R}+2UJ-\alpha S^4\rho_0^{\alpha-1}J^{-\alpha}
				\left[\psi_0'+(\alpha-1)\psi_0^2
				+2\left(\frac{J}{S^2r}-\frac{1}{a}\right)\psi_0\right].
			\end{aligned}
		\end{equation}
		These functions will also be used in the tail and exterior regularity
		estimates. The Jacobian equation becomes
		\begin{equation}\label{eq:jac:4.2}
			J_t=\mathsf D_JJ_{aa}-\frac{(\alpha+1)\mathsf D_J}{J}J_a^2
			+\mathsf B_JJ_a+\mathsf F_J.
		\end{equation}
		At a spatial maximum with $J\ge1$, $J_a=0$, $J_{aa}\le0$. Bounded
		$\psi_0,\psi_0'$, \eqref{eq:jac:3.1}, and $\alpha>1$ give
		\begin{equation}\label{eq:jac:4.3}
			\begin{aligned}
				&\left|\alpha S^4{\rho_0^{\alpha-1}}J^{-\alpha}
				\left[\psi_0'+(\alpha-1)\psi_0^2
				+2\left(\frac{J}{S^2r}-\frac{1}{a}\right)\psi_0\right]\right|\\
				&\qquad\le C_T {\rho_0^{\alpha-1}}(J^{-\alpha}+J^{1-\alpha})\le C_T.
			\end{aligned}
		\end{equation}
		Together with \eqref{eq:jac:3.8}, this yields at such a maximum
		\begin{equation}\label{eq:jac:4.4}
			J_t\le C_T\left(1+J(t,a)+\int_0^tJ(s,a)\,ds\right).
		\end{equation}
		At the inner boundary label, \eqref{eq:compactpositive} and bounded displacement give
		\begin{equation}\label{eq:jac:4.5}
			J(t,a_*)\le B_*:=
			\max\left\{1,\frac{\rho_0(a_*)}{c_{T,a_*+D_T}}\right\}<\infty.
		\end{equation}
		On each time interval $[0,\tau]$, $\tau<T$, $\mathbf u\in C([0,\tau];H^{s'})$, $s'>\frac{5}{2}$, implies
		uniform vanishing of its gradient at infinity. Hence
		\begin{equation}\label{eq:jac:4.6}
			J(t,a)=\exp\left(\int_0^t d(s,r(s,a))\,ds\right)\longrightarrow1
			\quad(a\to\infty)
		\end{equation}
		uniformly for $t\le\tau$, and $J_t=Jd\to0$ uniformly there as well.
		Spatial maxima exceeding $\max\{1,B_*\}$ are therefore attained at
		finite interior labels. The Sobolev norm on $[0,\tau]$ only justifies this endpoint; it does not
		enter the comparison constants.
		
		For $M(t)=\max\{B_*,\sup_{a\ge a_*}J(t,a)\}$, the upper-Dini
		maximum argument gives
		\begin{equation}\label{eq:jac:4.7}
			D^+M(t)\le C_T\left(1+M(t)+\int_0^tM(s)\,ds\right).
		\end{equation}
		The time integral at a maximizing label is bounded by the integral of
		$M$; the constant obstacle contributes zero growth. Set
		$N=1+M+\int_0^tM$ and keep the comparison constant $C_T$ fixed. Then
		\begin{equation}\label{eq:jac:4.8}
			D^+N\le(C_T+1)N,\qquad
			M(t)\le(1+B_*)e^{(C_T+1)t}.
		\end{equation}
		The bound is independent of $\tau<T$. Enlarging $C_T$ after this
		comparison gives $J\le C_T$ and hence
		\begin{equation}\label{eq:jac:4.9}
			\rho(t,r(t,a))\ge\frac{\rho_0(a)}{C_T},\qquad
			0<r_a=\frac{J}{S^2}\le4C_T\quad(a\ge a_*).
		\end{equation}
		Moreover bounded $\psi_0$ and displacement imply
		\[
		\left|\log\frac{\rho_0(a)}{\rho_0(r(t,a))}\right| \le\|\psi_0\|_\infty D_T.
		\]
		Thus for physical radii $r\ge a_*+D_T$,
		\begin{equation}\label{eq:jac:4.10}
			\rho(t,r)\ge\frac{\rho_0(r)}{C_T}e^{-\|\psi_0\|_\infty D_T}.
		\end{equation}
		This relative lower bound still vanishes at infinity. It is not an upper
		bound for $\frac{\rho}{\rho_0}$.
	\end{proof}
	
	With the upper bound for $J$ established, we return to the normalized
	effective strain and sharpen its estimate at infinity. Its uniform
	decay will make the source terms small in the remaining Jacobian
	comparisons. At this stage we continue to work with the normalized
	strain, since recovering $v_r$ requires the lower bound for $J$ still
	to be proved.
	
	\begin{lemma}[Vanishing reference-strain tail]\label{lem:referencetails}
		The normalized effective derivative is bounded and has a vanishing
		tail:
		\[
		\sup_{t<T,a\ge a_*}|\mathcal{R}(t,a)|<\infty,\qquad
		\sup_{t<T,a\ge R}|\mathcal{R}(t,a)|\longrightarrow0.
		\]
	\end{lemma}
	\begin{proof}
		For $R\ge a_*$ put
		\[
		\epsilon_R=\sup_{a\ge R}\left(|v_{0r}(a)|+\rho_0(a)^{\alpha-1}\right)
		+\sup_{t<T,\ a\ge R}(|U|+F_1+Q).
		\]
		By Lemma~\ref{lem:velocitytails}, Lemma~\ref{lem:initialstrain}, bounded
		displacement, and the far-field condition on $\rho_0$, we have $\epsilon_R\to0$.
		The value $F_{1,0}$ is included at $t=0$, so \eqref{eq:jac:3.7} and \eqref{eq:jac:4.8} yield
		\[
		|\Phi_a(J)|\le\epsilon_R(1+C_T)\quad(a\ge R).
		\]
		For $\epsilon_R\le1$, equation \eqref{eq:jac:3.6} gives
		\[
		|\mathcal{Z}(t,a)|\le e^{3\epsilon_Rt} \left(\epsilon_R+C_Tt(\epsilon_R+\epsilon_R^2)\right) \le C_T\epsilon_R.
		\]
		Adding the primitive and using $S\ge1/2$,
		\begin{equation}\label{eq:jac:5.1}
			\sup_{t<T,\ a\ge R}
			\left(|\mathcal{R}(t,a)|+|\partial_a[v(t,r(t,a))]|\right)
			\le C_T\epsilon_R\longrightarrow0.
		\end{equation}
		The Eulerian derivative is $v_r=\frac{\mathcal{R}}{J}$. Recovering it would need additional
		information at small $J$; no such division is made in \eqref{eq:jac:5.1}.
	\end{proof}
	
	We now use this vanishing strain to refine the upper Jacobian bound:
	the expansion above the initial volume becomes uniformly small far
	from the origin. This gives the upper half of the eventual convergence
	of $J$ to one; the lower half will also require control at a material
	boundary.
	
	\begin{lemma}[Vanishing expansion in the tail]\label{lem:expansiontail}
		The excess reference Jacobian satisfies
		\[
		\sup_{t<T,a\ge R}(J(t,a)-1)_+\longrightarrow0.
		\]
	\end{lemma}
	\begin{proof}
		Use the coefficients defined in \eqref{eq:jac:coefficients} in
		Proposition~\ref{prop:upperjacobian}. On the subset $1\le J\le C_T$
		in the material exterior, they satisfy
		\begin{equation}\label{eq:jac:6.1}
			0<\mathsf D_J\le C_T,\qquad |\mathsf B_J|\le C_T,\qquad
			|\mathsf F_J|\le C_T\epsilon_R\quad(a\ge R),\qquad \epsilon_R\to0.
		\end{equation}
		Here $\epsilon_R$ is defined in the proof of Lemma~\ref{lem:referencetails}.
		For the source use \eqref{eq:jac:5.1}, the bounds
		$|U|,\rho_0^{\alpha-1}\le\epsilon_R$, bounded $\psi_0,\psi_0'$, bounded
		$S$, and $J\ge1$ in negative powers. These
		coefficient bounds are asserted only on $J\ge1$.
		
		Choose a smooth cutoff $\chi_R$ which is zero on $a\le R$, one on
		$a\ge2R$, and satisfies
		\begin{equation}\label{eq:jac:6.2}
			0\le\chi_R\le1,\quad |\chi_R'|\le \frac{C}{R},\quad
			|\chi_R''|+\frac{|\chi_R'|^2}{\chi_R}\le \frac{C}{R^2}
			\quad\text{where }\chi_R>0.
		\end{equation}
		The square of a scaled fixed smooth cutoff has these properties.
		Put $\zeta_R=\chi_R(J-1)$. At a positive spatial maximum of $\zeta_R$,
		$J-1>0$, $\partial_a\zeta_R=0$, $\partial_{aa}\zeta_R\le0$, and
		\[
		J_a=-\frac{\chi_R'}{\chi_R}(J-1),\qquad
		\chi_RJ_{aa}=\partial_{aa}\zeta_R-\chi_R''(J-1)
		+2\frac{|\chi_R'|^2}{\chi_R}(J-1).
		\]
		The gradient-square term in \eqref{eq:jac:4.2} is nonpositive. Therefore
		\begin{equation}\label{eq:jac:6.3}
			\begin{aligned}
				\partial_t\zeta_R
				&\le\mathsf D_J\partial_{aa}\zeta_R
				+\mathsf D_J\left(-\chi_R''+2\frac{|\chi_R'|^2}{\chi_R}\right)(J-1)\\
				&\quad-\mathsf B_J\chi_R'(J-1)+\chi_R\mathsf F_J\\
				&\le C_T(\epsilon_R+R^{-1}+R^{-2}),
			\end{aligned}
		\end{equation}
		since $0<J-1\le C_T$. There is no quadratic reaction divided by a vanishing
		cutoff. Initially $\zeta_R=0$, its inner value is zero, and \eqref{eq:jac:4.6} handles
		infinity on each interval $[0,\tau]$, $\tau<T$. Thus
		\begin{equation}\label{eq:jac:6.4}
			\sup_{t<T,\ a\ge2R}(J-1)_+
			\le C_TT(\epsilon_R+R^{-1}+R^{-2})\longrightarrow0.
		\end{equation}
	\end{proof}
	
	To control compression, we need an inner boundary for the exterior
	comparison along which $J$ stays close to one throughout the time
	interval. The variation tails from Section~\ref{sec:variation} supply
	such a boundary: averaging the total time variation of $J$ over distant
	labels allows us to select one label that works for the entire interval.
	
	\begin{lemma}[A fixed boundary with small time variation]\label{lem:materialboundary}
		For every sufficiently large $R$ there is one fixed label
		$a_R\in[R,2R]$ such that
		\[
		\sup_{t<T}|J(t,a_R)-1|\le\int_0^T|J_t(t,a_R)|\dd t
		\le\delta_R,\qquad \delta_R\longrightarrow0.
		\]
		The label is selected for the entire interval $[0,T)$.
	\end{lemma}
	\begin{proof}
		Choose $R\ge a_*$ and use the displacement and radius-ratio bounds in
		\eqref{eq:jac:3.1}. By the identities for $J_t$ and $r_a$ in
		\eqref{eq:jac:3.3}, the factor $J$ cancels when changing from the
		initial-radius label $a$ to the physical radius $r$:
		\begin{equation}\label{eq:bound:2.1}
			\begin{aligned}
				\int_R^{2R}|J_t(t,a)|\,da
				&\le\int_R^{2R}J\left(|u_r|+\frac{2|u|}{r}\right)\,da\\
				&=\int_{r(t,R)}^{r(t,2R)}
				S(t,a(t,r))^2\left(|u_r(t,r)|+\frac{2|u(t,r)|}{r}\right)\,dr\\
				&\le\frac{9}{4}\int_{R-D_T}^{\infty}|{u_r}(t,r)|\,dr\\
				&\quad+\frac{9}{2}
				\sup_{r\ge R-D_T}|u(t,r)|
				\log\frac{2R+D_T}{R-D_T}.
			\end{aligned}
		\end{equation}
		The tail bound $m_T$ from \eqref{eq:variationtails} controls both terms:
		Proposition~\ref{prop:variationtails} also gives
		\[
		\sup_{t<T,\,r\ge R-D_T}|u(t,r)|\le m_T(R-D_T).
		\]
		Since $R\ge a_*=2D_T+3$, the logarithm is at most $\log5$.
		Integrating \eqref{eq:bound:2.1} in time therefore yields
		\begin{equation}\label{eq:bound:2.2}
			\int_R^{2R}\int_0^T|J_t(t,a)|\,dt\,da
			\le B_R:=\left(\frac{9}{4}+\frac{9}{2}\log5\right)T m_T(R-D_T),
			\qquad B_R\longrightarrow0.
		\end{equation}
		Tonelli applies to the nonnegative integrand on the full open time interval.
		There is a label $a_R\in[R,2R]$ such that
		\begin{equation}\label{eq:bound:2.3}
			\int_0^T|J_t(t,a_R)|\,dt\le\delta_R:=\frac{2B_R}{R}.
		\end{equation}
		If $B_R=0$, the same conclusion with zero holds for almost every label.
		The factor two avoids any issue about an attained average. The smooth
		flow for $t<T$ and $J(0,a)=1$ now imply, for every $t<T$,
		\begin{equation}\label{eq:bound:2.4}
			|J(t,a_R)-1|
			\le\int_0^t|J_t(s,a_R)|\,ds\le\delta_R.
		\end{equation}
		This is one fixed label for the entire interval, selected using \eqref{eq:bound:2.2}.
		It is not a different label for each time or for each time cutoff $\tau<T$.
		No inverse Jacobian or maximum strain was used in its selection.
	\end{proof}
	
	The selected trajectory supplies the inner boundary value for the
	lower comparison. Together with the small exterior source terms, it
	allows a minimum principle for $J^2$ to control compression throughout
	the region beyond that trajectory. Combining this lower bound with the
	expansion estimate completes the two-sided Jacobian control and allows
	us to recover the Eulerian effective derivative from the reference strain.
	
	\begin{proposition}[Two-sided Jacobian control]\label{prop:lowerjacobian}
		For some finite $X_T$ and all $a\ge X_T$, $t<T$,
		\[
		\frac{1}{2}\le J(t,a)\le C_T,\qquad
		\sup_{t<T,a\ge R}|J(t,a)-1|\longrightarrow0.
		\]
		Furthermore $v_r=\frac{\mathcal{R}}{J}$ tends to zero uniformly in the exterior and
		$\sup_{t<T}\|v_r(t)\|_\infty<\infty$.
	\end{proposition}
	\begin{proof}
		Proposition~\ref{prop:upperjacobian} already gives $J\le C_T$.
		To obtain the lower bound, we apply the minimum principle to $K=J^2$.
		Multiplication by $2J$ converts the factor
		$\rho_0^{\alpha-1}J^{-\alpha}$ in $\mathsf F_J$ into
		$2\rho^{\alpha-1}$, by \eqref{eq:jac:3.3}.
		At a spatial minimum of $K$, positivity of $J$ gives
		$J_a=0$ and $K_{aa}=2JJ_{aa}\ge0$. Thus the Jacobian equation
		\eqref{eq:jac:4.2} and the source formula \eqref{eq:jac:coefficients} yield
		\begin{equation}\label{eq:bound:3.2}
			\begin{aligned}
				K_t&=\mathsf D_JK_{aa}+2J\mathsf F_J\\
				&=\mathsf D_JK_{aa}+2J\mathcal{R}+4UK\\
				&\quad-2\alpha S^4\rho^{\alpha-1}
				\left[\psi_0'+(\alpha-1)\psi_0^2
				+2\left(\frac{J}{S^2r}-\frac{1}{a}\right)\psi_0\right].
			\end{aligned}
		\end{equation}
		The bounds for $\psi_0,\psi_0'$ from Lemma~\ref{lem:initialstrain},
		together with \eqref{eq:jac:3.1} and $J\le C_T$, bound the bracket
		in \eqref{eq:bound:3.2} by $C_T$. Hence, throughout $a\ge R\ge a_*$,
		\begin{equation}\label{eq:bound:3.3}
			|2J\mathsf F_J|\le
			\eta_R:=C_T\sup_{\substack{t<T\\a\ge R}}
			(|\mathcal{R}|+|U|+\rho^{\alpha-1})
			\longrightarrow0.
		\end{equation}
		Here the reference-strain tail vanishes by Lemma~\ref{lem:referencetails};
		the velocity and density tails vanish by Lemma~\ref{lem:velocitytails}
		and \eqref{eq:bar:2.3}, using the displacement bound \eqref{eq:jac:3.1}.
		Since $\mathsf D_J>0$, equation \eqref{eq:bound:3.2} now gives
		$K_t\ge-\eta_R$ at every interior spatial minimum with $a\ge R$.
		
		Fix a large $R$ with $\delta_R<1$, and take the fixed label
		$a_R\in[R,2R]$ from Lemma~\ref{lem:materialboundary}.
		We compare $K$ on $a\ge a_R$ with
		\[
		q_R(t)=(1-\delta_R)^2-\eta_Rt.
		\]
		The initial value $K(0,a)=1$ lies above $q_R(0)$, and
		\eqref{eq:bound:2.4} gives the inner boundary value
		\[
		K(t,a_R)\ge(1-\delta_R)^2\ge q_R(t).
		\]
		On each time interval $[0,\tau]$, $\tau<T$, equation \eqref{eq:jac:4.6}
		gives $K\to1$ uniformly at infinity. A first lower contact with
		$q_R-\varepsilon(1+t)$ must therefore occur at a finite interior label.
		At such a contact the minimum inequality implies
		\[
		K_t\ge-\eta_R>q_R'(t)-\varepsilon,
		\]
		contradicting first contact.
		Letting $\varepsilon\downarrow0$ and then $\tau\uparrow T$ gives
		\begin{equation}\label{eq:bound:3.4}
			J(t,a)^2\ge(1-\delta_R)^2-T\eta_R
			\quad(t<T,\ a\ge a_R).
		\end{equation}
		Choose one sufficiently large $R$ with $\delta_R\le\frac{1}{4}$ and
		$T\eta_R\le\frac{1}{4}$. The right side is then at least
		$\frac{5}{16}$, so $J\ge\frac{1}{2}$ on $a\ge2R$.
		Taking $X_T=2R$ proves the stated two-sided bound.
		
		More generally, \eqref{eq:bound:3.4} holds on $a\ge2R$ for every large $R$ and its right
		side tends to one. Combining this lower bound with
		Lemma~\ref{lem:expansiontail} proves
		\begin{equation}\label{eq:bound:3.5}
			\sup_{t<T,\ a\ge R}|J-1|\longrightarrow0.
		\end{equation}
		Finally, the lower bound for $J$ and Lemma~\ref{lem:referencetails} give
		\begin{equation}\label{eq:bound:3.6}
			v_r=\frac{\mathcal{R}}{J},\qquad
			\sup_{t<T,\ a\ge R}|{v_r}(t,r(t,a))|\longrightarrow0.
		\end{equation}
		Bounded displacement transfers the tail estimate to physical radii.
		Lemma~\ref{lem:localvgradient} supplies the bound on the remaining fixed ball.
	\end{proof}
	
	The exterior density is now uniformly comparable to its initial value
	along trajectories, and the effective velocity gradient is controlled
	globally. We next use this volume control to estimate the physical
	velocity gradient in the exterior, applying parabolic estimates on
	spatial scales determined by the initial density.
	
	\section{Intrinsic exterior parabolic regularity}\label{sec:parabolic}
	Two-sided bounds for $J$ control the density relative to its initial value.
	The next estimates control the rate of deformation. They use fixed
	initial-density spatial scales and leave time unchanged.
	
	\begin{lemma}[Scalar transformation and uniform scaled coefficients]\label{lem:intrinsic}
		Set $\mathcal J=J^{-\alpha}$. For each sufficiently exterior center $a_c$,
		put $\ell=\rho_0(a_c)^{\frac{\alpha-1}{2}}$, $a=a_c+\ell y$, $|y|<4$.
		Then
		\[
		\mathcal J_t=\mathfrak a\mathcal J_{yy}+\mathfrak b\mathcal J_y+\mathfrak f,\qquad \mathcal J(0,y)=1,
		\]
		with $0<\frac{1}{C_T}\le\mathfrak a\le C_T$ and
		$|\mathfrak b|+|\mathfrak f|+|\mathcal J|\le C_T$ uniformly in the center
		and in $\ell^{-1}$. The geometric factors $S,\frac{1}{r}$ are uniformly
		Lipschitz in the scaled space and time variables.
	\end{lemma}
	\begin{proof}
		Define $\mathcal J=J^{-\alpha}$. Differentiate:
		\[
		\mathcal J_a=-\alpha J^{-\alpha-1}J_a,\qquad
		\mathcal J_{aa}=\alpha(\alpha+1)J^{-\alpha-2}J_a^2
		-\alpha J^{-\alpha-1}J_{aa}.
		\]
		Multiplying the Jacobian equation \eqref{eq:jac:4.2} by
		$-\alpha J^{-\alpha-1}$, with the coefficients from
		\eqref{eq:jac:coefficients} in Proposition~\ref{prop:upperjacobian}, gives
		\begin{equation}\label{eq:bound:4.1}
			\mathcal J_t=\mathsf D_J\mathcal J_{aa}+\mathsf B_J\mathcal J_a-\alpha J^{-\alpha-1}\mathsf F_J.
		\end{equation}
		The $J_a^2$ terms cancel.
		
		After \eqref{eq:bound:3.4}, $\mathcal J$ is bounded above and below by positive constants.
		The source in \eqref{eq:bound:4.1} is bounded without any derivative of $J,\mathcal{R},u$.
		
		Fix an exterior label $a_c$ and put
		\begin{equation}\label{eq:bound:5.1}
			\ell=\rho_0(a_c)^{\frac{\alpha-1}{2}}>0,\qquad a=a_c+\ell y,\qquad
			k(y)=\frac{{\rho_0^{\alpha-1}}(a_c+\ell y)}{{\rho_0^{\alpha-1}}(a_c)}.
		\end{equation}
		Time is not rescaled. We use nested fixed intervals in $|y|<4$.
		Since $\ell\le\|{\rho_0^{\alpha-1}}\|_\infty^{\frac{1}{2}}$, take all centers at least
		$4\|{\rho_0^{\alpha-1}}\|_\infty^{\frac{1}{2}}+1$ beyond the fixed label where
		\eqref{eq:bound:3.4} holds.
		
		Bounded $\psi_0$ gives, for $|y|\le4$,
		\begin{equation}\label{eq:bound:5.2}
			|\log k(y)|
			=\left|{(\alpha-1)}\int_{a_c}^{a_c+\ell y}\psi_0(s)\,ds\right|
			\le4{(\alpha-1)}\|\psi_0\|_\infty
			\|{\rho_0^{\alpha-1}}\|_\infty^{\frac{1}{2}}\le C_T.
		\end{equation}
		Thus $k$ is bounded above and below by positive constants independent
		of $a_c$ and of $\ell\downarrow0$. Also
		\begin{equation}\label{eq:bound:5.3}
			(k)_y={(\alpha-1)}\ell\psi_0 k,\qquad
			(k)_{yy}=\ell^2({(\alpha-1)}\psi_0'+{(\alpha-1)}^2\psi_0^2)k.
		\end{equation}
		We use the same symbols $J,S,\mathcal{R}$ for their compositions with this
		fixed affine label change. Equation \eqref{eq:bound:4.1} becomes
		\begin{equation}\label{eq:bound:5.4}
			\mathcal J_t=\mathfrak a\,\mathcal J_{yy}+\mathfrak b\,\mathcal J_y+\mathfrak f,
		\end{equation}
		where \eqref{eq:jac:coefficients} gives
		\begin{equation}\label{eq:bound:5.5}
			\begin{aligned}
				\mathfrak a&=\frac{\mathsf D_J}{\ell^2}
				=\alpha S^4 k J^{-\alpha-1},\\
				\mathfrak b&=\frac{\mathsf B_J}{\ell}
				=\ell\mathfrak a\left[(2\alpha-1)\psi_0
				+2\left(\frac{J}{S^2r}-\frac{1}{a}\right)\right],\\
				\mathfrak f&=-\alpha J^{-\alpha-1}\mathsf F_J\\
				&=-\alpha J^{-\alpha-1}\Biggl\{\mathcal{R}+2UJ-\alpha S^4\ell^2 k J^{-\alpha}
				\left[\psi_0'+(\alpha-1)\psi_0^2
				+2\left(\frac{J}{S^2r}-\frac{1}{a}\right)\psi_0\right]\Biggr\}.
			\end{aligned}
		\end{equation}
		There is a constant $C_T\ge1$, depending only on $T$, the exponents, and
		the initial data, such that
		\begin{equation}\label{eq:bound:5.6}
			0<\frac{1}{C_T}\le\mathfrak a\le C_T,\qquad
			|\mathfrak b|+|\mathfrak f|+|\mathcal J|\le C_T,\qquad \mathcal J(0,y)=1.
		\end{equation}
		As usual, $C_T$ may increase
		from line to line, but is independent of the center, $\ell^{-1}$, and
		the time $\tau<T$.
		
		The geometric coefficients have uniform space and time moduli already:
		\begin{equation}\label{eq:bound:5.7}
			S_t=US,\qquad
			S_y=\ell S\left(\frac{J}{S^2r}-\frac{1}{a}\right),\qquad
			(r^{-1})_t=-\frac{u}{r^2},\qquad
			(r^{-1})_y=-\frac{\ell J}{S^2r^2}.
		\end{equation}
		These are bounded.
	\end{proof}
	
	The initial functions $\psi_0'$ and $v_{0r}$ have uniform
	exterior $C^{\frac{1}{4}}$ norms by Lemma~\ref{lem:initialstrain}. Pullback to a
	scaled spatial interval multiplies their spatial seminorms by $\ell^{\frac{1}{4}}$,
	never by its inverse. The formulae for $k_y,k_{yy}$ in
	\eqref{eq:bound:5.3} also give a uniform $C^{2,\frac{1}{4}}$ norm for $k$.
	
	\begin{lemma}[First modulus without differentiating the coefficient]\label{lem:firstholder}
		On every cylinder $[0,\tau]\times(-4,4)$ in the scaled coordinates,
		\[
		\sup_{t\le\tau}\|\mathcal J_y(t)\|_{L^2(-3,3)}^2
		+\|\mathcal J_{yy}\|_{L^2((0,\tau)\times(-3,3))}^2
		+\|\mathcal J_t\|_{L^2((0,\tau)\times(-3,3))}^2\le C_T.
		\]
		The functions $\mathcal J$, $J$, and the leading coefficient $\mathfrak a$ have
		uniformly bounded $C^{\frac{1}{4},\frac{1}{8}}$ H\"older norms on the smaller
		space--time cylinder $[0,\tau]\times[-\frac{11}{4},\frac{11}{4}]$, including at $t=0$.
		The constant is independent of $\tau<T$ and of the center.
	\end{lemma}
	\begin{proof}
		Let $q=\mathcal J-1$, $I_s=(-s,s)$, and work first on a fixed interval
		$[0,\tau]$, $\tau<T$, on which the solution is smooth. Then
		\[
		q_t-\mathfrak aq_{yy}=\mathfrak b q_y+\mathfrak f,\qquad q(0)=0.
		\]
		For $3\le r<s\le4$ set $\delta=s-r$ and choose $\chi=1$ on $I_r$,
		supported in $I_s$, with $|\chi'|\le \frac{C}{\delta}$ and
		$|\chi''|\le \frac{C}{\delta^2}$. For $z=\chi q$,
		\[
		z_t-\mathfrak a z_{yy}=g,
		\quad
		g=\chi \mathfrak f+(\chi\mathfrak b-2\mathfrak a\chi')q_y-\mathfrak a\chi''q,
		\quad z(0)=0.
		\]
		Testing by $-z_{yy}$, integrating, and using the ellipticity bound in
		\eqref{eq:bound:5.6} and Young's inequality gives
		\[
		\sup_{t\le\tau}\|z_y(t)\|_2^2 +\|z_{yy}\|_{L^2_{t,y}}^2 \le C_T\|g\|_{L^2_{t,y}}^2.
		\]
		The equation also implies
		$\|z_t\|_2^2\le C_T(\|z_{yy}\|_2^2+\|g\|_2^2)$.
		Define
		\[
		\mathcal E(s)=\sup_{t\le\tau}\|q_y(t)\|_{L^2(I_s)}^2
		+\|q_{yy}\|_{L^2((0,\tau)\times I_s)}^2
		+\|q_t\|_{L^2((0,\tau)\times I_s)}^2.
		\]
		Using $\delta\le1$ and the amplitude bounds yields
		\begin{equation}\label{eq:para:A1}
			\mathcal E(r)\le C_T\left[
			\|\mathfrak f\|_{L^2((0,\tau)\times I_s)}^2
			+\delta^{-2}\|q_y\|_{L^2((0,\tau)\times I_s)}^2
			+\delta^{-4}\|q\|_{L^2((0,\tau)\times I_s)}^2\right].
		\end{equation}
		On intervals $I_s$ with $3\le s\le4$, the one-dimensional interpolation
		inequality has constants uniform in $s$:
		\[
		\|q_y(t)\|_{L^2(I_s)}^2
		\le\varepsilon^2\|q_{yy}(t)\|_{L^2(I_s)}^2
		+C\varepsilon^{-2}\|q(t)\|_{L^2(I_s)}^2,
		\qquad 0<\varepsilon\le1.
		\]
		Choose
		$\varepsilon^2$ proportional to $\theta\delta^2$ in \eqref{eq:para:A1}, integrate in
		time, and use $\tau\le T$. For any fixed $0<\theta<1$ this gives
		\begin{equation}\label{eq:para:A2}
			\mathcal E(r)\le\theta\mathcal E(s)
			+C_{T,\theta}\delta^{-4}.
		\end{equation}
		Set $r_n=4-2^{-n}$, so $r_0=3$ and
		$r_{n+1}-r_n=2^{-n-1}$. Iteration gives
		\[
		\mathcal E(3)\le\theta^N\mathcal E(r_N) +16C_{T,\theta}\sum_{n=0}^{N-1}(16\theta)^n.
		\]
		Choose $\theta=\frac{1}{32}$. For this fixed interval and center, smoothness gives
		$\mathcal E(r_N)\le\mathcal E(4)<\infty$, so the first term tends to
		zero. The resulting bound is independent of $\mathcal E(4)$. The geometric sum is bounded by two.
		Thus $\mathcal E(3)\le C_T$ uniformly in the center and $\tau$.
		
		We now turn this energy bound into a pointwise modulus. For $y,z\in I_3$,
		the fundamental theorem of calculus and Cauchy--Schwarz give
		\[
		|q(t,y)-q(t,z)|
		=\left|\int_z^y q_y(t,\xi)\dd\xi\right|
		\le\|q_y(t)\|_{L^2(I_3)}|y-z|^{\frac{1}{2}}
		\le C_T|y-z|^{\frac{1}{2}}.
		\]
		To compare two times at a fixed $y\in I_{\frac{11}{4}}$, choose
		$0<\eta\le\frac{1}{4}$. Then $(y-\eta,y+\eta)\subset I_3$, so the average
		\[
		\overline q_\eta(t,y)
		=\frac{1}{2\eta}\int_{y-\eta}^{y+\eta}q(t,\xi)\dd\xi
		\]
		satisfies
		\[
		|q(t,y)-\overline q_\eta(t,y)|
		\le\frac{1}{2\eta}\int_{y-\eta}^{y+\eta}
		|q(t,y)-q(t,\xi)|\dd\xi
		\le C_T\eta^{\frac{1}{2}}.
		\]
		For $0\le s<t\le\tau$, the time derivative bound controls the difference
		of these averages:
		\[
		\begin{aligned}
			|\overline q_\eta(t,y)-\overline q_\eta(s,y)|
			&=\left|\frac{1}{2\eta}\int_s^t\int_{y-\eta}^{y+\eta}
			q_t(\sigma,\xi)\dd\xi\dd\sigma\right|\\
			&\le\frac{|t-s|^{\frac{1}{2}}}{(2\eta)^{\frac{1}{2}}}
			\|q_t\|_{L^2((s,t)\times I_3)}
			\le C_T\eta^{-\frac{1}{2}}|t-s|^{\frac{1}{2}}.
		\end{aligned}
		\]
		Adding the two point-to-average errors, one at each time, gives
		\[
		|q(t,y)-q(s,y)|
		\le C_T\left(\eta^{\frac{1}{2}}
		+\eta^{-\frac{1}{2}}|t-s|^{\frac{1}{2}}\right).
		\]
		If $0<t-s\le\frac{1}{16}$, take $\eta=|t-s|^{\frac{1}{2}}$ to balance
		the two terms. For larger time gaps, the same conclusion follows from
		the uniform bound for $q$. Thus
		\[
		|q(t,y)-q(s,y)|\le C_T|t-s|^{\frac{1}{4}}.
		\]
		The averaging calculation is valid also for $s=0$, since the bound for
		$q_t$ holds on the whole interval $(0,\tau)$ and $q$ has its prescribed
		initial trace.
		
		Combining the spatial and temporal estimates, and recalling
		$\mathcal J=1+q$, we obtain
		\[
		|\mathcal J(t,y)-\mathcal J(s,z)|
		\le C_T\left(|y-z|^{\frac{1}{2}}+|t-s|^{\frac{1}{4}}\right)
		\quad\text{on }[0,\tau]\times I_{\frac{11}{4}}.
		\]
		The two-sided bound for $J$ places $\mathcal J$ in a fixed compact
		subinterval of $(0,\infty)$. The derivative of $x\mapsto x^{-\frac{1}{\alpha}}$
		is bounded there, so the mean value theorem and $J=\mathcal J^{-\frac{1}{\alpha}}$
		give the same modulus for $J$. These are uniform
		$C^{\frac{1}{2},\frac{1}{4}}$ bounds. In particular, on this bounded space--time cylinder
		they imply the parabolic $C^{\frac{1}{4}}$ bounds, by which we mean
		$C^{\frac{1}{4},\frac{1}{8}}$ in space and time.
		
		Finally, recall from \eqref{eq:bound:5.5} that
		$\mathfrak a=\alpha S^4kJ^{-\alpha-1}$.
		The two-sided bounds for $S,J$ and the bounds for
		$\ell,U,\frac{1}{r},\frac{1}{a}$ make both $S_y$ and $S_t$ in
		\eqref{eq:bound:5.7} uniformly bounded. Together with the bound for $k_y$
		in \eqref{eq:bound:5.3} and the fact that $k$ is independent of time,
		this gives
		\[
		|S(t,y)-S(s,z)|+|k(y)-k(z)|
		\le C_T\bigl(|y-z|+|t-s|\bigr).
		\]
		All factors in the formula for $\mathfrak a$ are bounded, and $J$ is
		bounded away from zero. Taking differences of the product and using
		the mean value theorem for $S^4$ and $J^{-\alpha-1}$ therefore yields
		\[
		\begin{aligned}
			|\mathfrak a(t,y)-\mathfrak a(s,z)|
			&\le C_T\bigl(|S(t,y)-S(s,z)|+|k(y)-k(z)|
			+|J(t,y)-J(s,z)|\bigr)\\
			&\le C_T\left(|y-z|^{\frac{1}{4}}+|t-s|^{\frac{1}{8}}\right).
		\end{aligned}
		\]
		This proves the required uniform parabolic H\"older norm for
		$\mathfrak a$; its ellipticity bounds
		$\frac{1}{C_T}\le\mathfrak a\le C_T$ were already established in
		\eqref{eq:bound:5.6}. All constants remain independent of the center,
		$\ell^{-1}$, and $\tau<T$.
	\end{proof}
	
	\begin{lemma}[Scaled first derivative modulus]\label{lem:lp}
		Uniformly for $\tau\in[\frac{T}{2},T)$,
		\[
		\|\mathcal J-1\|_{W^{2,1}_8((0,\tau)\times(-\frac{5}{2},\frac{5}{2}))}\le C_T,
		\qquad
		\|\mathcal J_y\|_{C^{\frac{1}{4},\frac{1}{8}}([0,\tau]\times[-2,2])}\le C_T.
		\]
	\end{lemma}
	\begin{proof}
		Fix a cylinder $[0,\tau]\times(-4,4)$ in the scaled coordinates, with
		$\frac{T}{2}\le\tau<T$, on which the solution is smooth.
		Put $q=\mathcal J-1$, $I_s=(-s,s)$, and $p=8$.
		By \eqref{eq:bound:5.4}--\eqref{eq:bound:5.6},
		\[
		q_t-\mathfrak a q_{yy}-\mathfrak b q_y=\mathfrak f,
		\qquad q(0,y)=0.
		\]
		We first obtain the $W^{2,1}_8$ bound by applying the interior
		$L^p$ estimate to spatial cutoffs of $q$. Interpolation and an iteration
		over nested intervals will control the derivatives introduced by these
		cutoffs. The H\"older estimate for $q_y$ will then follow by Sobolev
		embedding.
		
		To prepare for the interior estimate, extend $\mathfrak a$ from
		$[0,\tau]\times\overline I_{\frac{11}{4}}$ to all space and time by
		replacing each coordinate outside this rectangle by the nearest endpoint
		of its interval. This projection does not increase spatial or temporal
		distances, so it preserves the ellipticity bounds and parabolic H\"older
		modulus. Extend $\mathfrak b$ by zero outside the rectangle; the theorem
		requires only a bounded measurable drift. These extensions allow us to
		use fixed larger cylinders for functions supported in the original cylinder.
		
		For $\frac{5}{2}\le r<s\le\frac{11}{4}$, set $\delta=s-r$ and choose
		$\chi\in C_c^\infty(I_s)$ with $0\le\chi\le1$, $\chi=1$ on $I_r$, and
		$|\chi'|\le\frac{C}{\delta}$, $|\chi''|\le\frac{C}{\delta^2}$.
		The localized function $z=\chi q$, extended by zero in space outside
		$I_s$, agrees with $q$ on $I_r$. The product rule gives
		$z_y=\chi q_y+\chi' q$ and
		$z_{yy}=\chi q_{yy}+2\chi' q_y+\chi''q$, hence
		\begin{equation}\label{eq:para:A3}
			z_t-\mathfrak a z_{yy}-\mathfrak b z_y
			=\chi\mathfrak f-2\mathfrak a\chi' q_y
			-(\mathfrak a\chi''+\mathfrak b\chi')q=:G,
			\qquad z(0)=0.
		\end{equation}
		We apply Theorem~\ref{thm:interior-lp} with $p=8$,
		$\beta=\frac{1}{4}$, unknown $w=z$, and right-hand side $G$.
		We specify its domains to include the initial time while keeping
		the constant uniform in $\tau$.
		
		Extend $z$ and $G$ by zero to $-1<t<0$, and set
		\[
		\Omega=(-1,\tau)\times I_4,
		\qquad
		\Omega'=(0,\tau)\times I_3.
		\]
		The zero trace $z(0)=0$ ensures that this extension has no jump at
		$t=0$, so its weak time derivative acquires no term supported there.
		Smoothness on the fixed interval and the spatial cutoff therefore give
		$z\in W^{2,1}_p(\Omega)$, and \eqref{eq:para:A3} continues to hold
		almost everywhere on $\Omega$. The initial time is now an interior
		time of the larger cylinder.
		In Theorem~\ref{thm:interior-lp}, the time parameters are
		$s_0=-1$, $s_1=0$, and $s_2=\tau$; the spatial center is zero and
		the outer and inner radii are $4$ and $3$, respectively.
		Thus the spatial gap is $4-3=1$ and the gap from the bottom is
		$0-(-1)=1$, both independent of $\tau$.
		
		The ellipticity and drift bounds follow from \eqref{eq:bound:5.6}, and
		Lemma~\ref{lem:firstholder} supplies the uniform parabolic H\"older
		modulus of $\mathfrak a$. The coefficient extensions preserve these
		bounds, verifying the coefficient hypotheses of
		Theorem~\ref{thm:interior-lp} with constants independent of the center
		and inverse initial density.
		
		Write $\|\cdot\|_p$ for the norm in
		$L^p((0,\tau)\times\mathbb R)$. Both $z$ and $G$ are supported in
		$I_s\subset I_3$ and vanish at negative times.
		Theorem~\ref{thm:interior-lp} therefore gives
		\begin{equation}\label{eq:para:A4}
			\begin{aligned}
				\|z_t\|_p+\|z_{yy}\|_p
				&=\|z_t\|_{L^p(\Omega')}+\|z_{yy}\|_{L^p(\Omega')}\\
				&\le C_T\left(\|G\|_{L^p(\Omega)}+\|z\|_{L^p(\Omega)}\right)\\
				&=C_T\left(\|G\|_p+\|z\|_p\right).
			\end{aligned}
		\end{equation}
		The constant is uniform in $\tau$, the center, and $\ell^{-1}$ by the
		coefficient bounds and domain gaps just verified. No condition
		at $t=\tau$ is required. The remaining task is to estimate $G$ and
		$z$ in terms of $q$ on the nested intervals.
		
		To track the two highest derivatives, put
		\[
		\mathcal H_p(s)=\|q_t\|_{L^p((0,\tau)\times I_s)}
		+\|q_{yy}\|_{L^p((0,\tau)\times I_s)}.
		\]
		Since $\chi=1$ on $I_r$, equations
		\eqref{eq:para:A3}--\eqref{eq:para:A4} and the cutoff bounds give
		\[
		\begin{aligned}
			\mathcal H_p(r)&\le C_T\left[
			\|\mathfrak f\|_{L^p((0,\tau)\times I_s)}
			+\delta^{-1}\|q_y\|_{L^p((0,\tau)\times I_s)}
			+(1+\delta^{-2})\|q\|_{L^p((0,\tau)\times I_s)}\right]\\
			&\le C_T\left[
			\|\mathfrak f\|_{L^p((0,\tau)\times I_s)}
			+\delta^{-1}\|q_y\|_{L^p((0,\tau)\times I_s)}
			+\delta^{-2}\|q\|_{L^p((0,\tau)\times I_s)}\right].
		\end{aligned}
		\]
		Here $\|z\|_p\le\|q\|_{L^p((0,\tau)\times I_s)}$, and
		$0<\delta\le\frac{1}{4}$ implies $1+\delta^{-2}\le2\delta^{-2}$.
		Thus the term $\|z\|_p$ from the interior estimate is controlled by
		the same bound as the zeroth-order cutoff error. The amplitude bounds
		in \eqref{eq:bound:5.6}, together with $2s\le\frac{11}{2}$ and
		$\tau\le T$, give
		\[
		\|q\|_{L^p((0,\tau)\times I_s)}
		+\|\mathfrak f\|_{L^p((0,\tau)\times I_s)}\le C_T.
		\]
		Only the term involving $q_y$ still needs to be controlled.
		At each fixed time, the interval interpolation inequality bounds its
		$L^p$ norm by a small multiple of the norm of $q_{yy}$ and a multiple
		of the norm of $q$. Its constant is uniform for
		$\frac{5}{2}\le s\le\frac{11}{4}$, by rescaling a fixed interval,
		and it requires no boundary condition on $q$. Taking the $L^p$ norm
		in time gives, for $0<\varepsilon\le1$,
		\[
		\|q_y\|_{L^p((0,\tau)\times I_s)}
		\le\varepsilon\|q_{yy}\|_{L^p((0,\tau)\times I_s)}
		+C\varepsilon^{-1}\|q\|_{L^p((0,\tau)\times I_s)}.
		\]
		Substituting this inequality into the cutoff estimate and using the
		amplitude bounds yields
		\[
		\mathcal H_p(r)
		\le C_T\left[\frac{\varepsilon}{\delta}\mathcal H_p(s)
		+\frac{1}{\delta\varepsilon}+\frac{1}{\delta^2}\right].
		\]
		Fix the constant $C_T\ge1$ in this inequality. For any $0<\theta<1$,
		choose $\varepsilon=\frac{\theta\delta}{C_T}$, which lies in $(0,1]$.
		The coefficient of $\mathcal H_p(s)$ is then $\theta$, and the two
		remaining terms are bounded by $C_{T,\theta}\delta^{-2}$. Thus
		\begin{equation}\label{eq:para:A5}
			\mathcal H_p(r)\le\theta\mathcal H_p(s)+C_{T,\theta}\delta^{-2}.
		\end{equation}
		We now iterate this estimate from the target interval
		$I_{\frac{5}{2}}$ towards $I_{\frac{11}{4}}$. Set
		\[
		r_n=\frac{11}{4}-\frac{1}{2^{n+2}},
		\qquad r_0=\frac{5}{2},
		\qquad r_{n+1}-r_n=\frac{1}{2^{n+3}}.
		\]
		The inverse square of the $n$th gap is $64\cdot4^n$.
		Applying \eqref{eq:para:A5} successively with $r=r_n$ and $s=r_{n+1}$
		therefore gives, for every integer $N\ge1$,
		\[
		\mathcal H_p\left(\frac{5}{2}\right)
		\le\theta^N\mathcal H_p(r_N)
		+64C_{T,\theta}\sum_{n=0}^{N-1}(4\theta)^n.
		\]
		Choose $\theta=\frac{1}{8}$, so $4\theta=\frac{1}{2}$ and the geometric
		sum is at most two. For this fixed cylinder, smoothness gives
		\[
		0\le\theta^N\mathcal H_p(r_N)
		\le\theta^N\mathcal H_p\left(\frac{11}{4}\right)
		\longrightarrow0
		\quad\text{as }N\to\infty.
		\]
		Here smoothness gives finiteness of the larger-interval norm; its value
		does not enter the resulting bound. Letting $N\to\infty$ proves
		$\mathcal H_p(\frac{5}{2})\le C_T$ uniformly in the center and $\tau$.
		The bound for $q$ and the interpolation inequality with
		$s=\frac{5}{2}$ and $\varepsilon=1$ also control $q_y$. Since $p=8$,
		we have obtained the full Sobolev estimate
		\[
		\|q\|_{W^{2,1}_8((0,\tau)\times I_{\frac{5}{2}})}
		\le C\left(\|q\|_{L^8((0,\tau)\times I_{\frac{5}{2}})}
		+\mathcal H_8\left(\frac{5}{2}\right)\right)
		\le C_T.
		\]
		
		We now pass from this Sobolev bound to a H\"older bound for $q_y$.
		In one spatial dimension, the parabolic dimension is $1+2=3$, since
		one time derivative has the same scaling as two spatial derivatives.
		The parabolic Sobolev embedding therefore controls the first spatial
		derivative in $C^{\beta,\frac{\beta}{2}}$ whenever
		$0<\beta<1-\frac{3}{p}$. For $p=8$, its form on a fixed cylinder is
		\[
		\|w_y\|_{C^{\beta,\frac{\beta}{2}}([0,1]\times[-2,2])}
		\le C_\beta
		\|w\|_{W^{2,1}_8((0,1)\times I_{\frac{5}{2}})},
		\qquad 0<\beta<\frac{5}{8}.
		\]
		Thus the time H\"older exponent of the derivative is half its spatial
		exponent.
		
		To apply this fixed-cylinder estimate uniformly in $\tau$, put
		$\widehat q(s,y)=q(\tau s,y)$ for $0\le s\le1$.
		Changing variables $t=\tau s$ multiplies the $L^8$ norms of
		$q,q_y,q_{yy}$ by $\tau^{-\frac{1}{8}}$, while
		$\widehat q_s=\tau q_t(\tau s,y)$ gives the factor
		$\tau^{\frac{7}{8}}$ for the time derivative. Consequently,
		\[
		\|\widehat q\|_{W^{2,1}_8((0,1)\times I_{\frac{5}{2}})}
		\le\left(\tau^{-\frac{1}{8}}+\tau^{\frac{7}{8}}\right)
		\|q\|_{W^{2,1}_8((0,\tau)\times I_{\frac{5}{2}})}
		\le C_T
		\]
		for $\tau\in[\frac{T}{2},T)$. Returning to the original time variable
		multiplies the temporal H\"older seminorm by
		$\tau^{-\frac{\beta}{2}}$, which is also bounded in terms of $T$ and
		$\beta$ on this range. Taking $\beta=\frac{1}{4}<\frac{5}{8}$ gives
		\[
		\|q_y\|_{C^{\frac{1}{4},\frac{1}{8}}([0,\tau]\times[-2,2])}
		\le C_T.
		\]
		In particular, for $s,t\in[0,\tau]$ and $y,z\in[-2,2]$,
		\[
		|q_y(t,y)-q_y(s,z)|
		\le C_T\left(|y-z|^{\frac{1}{4}}+|t-s|^{\frac{1}{8}}\right).
		\]
		The embedding gives a continuous representative up to the initial
		time. Since $q(0,y)=0$, its spatial derivative there is $q_y(0,y)=0$;
		in particular, the last estimate with $s=0$ and $z=y$ yields
		$|q_y(t,y)|\le C_Tt^{\frac{1}{8}}$. Finally,
		$q_y=\mathcal J_y$, which yields the stated H\"older estimate for
		$\mathcal J_y$ up to and including $t=0$.
		
		All constants depend only on the established coefficient bounds and
		modulus, the amplitude bounds, and $T$; they are independent of the
		center and $\ell^{-1}$.
	\end{proof}
	
	\begin{proposition}[Uniform exterior bound for the radial velocity derivative]\label{prop:exteriorstrain}
		For every $t_0\in(0,T)$ there are finite $X_T,C_{T,t_0}$ such that
		\begin{equation}\label{eq:exteriorstrain}
			\sup_{t_0\le t<T,\,a\ge X_T}|u_r(t,r(t,a))|\le C_{T,t_0}.
		\end{equation}
	\end{proposition}
	\begin{proof}
		Fix $\max\{t_0,\frac{T}{2}\}<\tau<T$ and scaled coordinates
		$a=a_c+\ell y$, $|y|<4$, as in Lemma~\ref{lem:intrinsic}; write $I_s=(-s,s)$.
		We first use Lemma~\ref{lem:lp} to obtain H\"older bounds for the
		drift and source in the $\mathcal J$ equation. This will allow us to
		apply Theorem~\ref{thm:interior-schauder} and bound $\mathcal J_t$,
		from which the continuity equation gives the radial velocity derivative $u_r$.
		All H\"older norms below are on $[0,\tau]\times\overline I_2$,
		unless a smaller domain is specified; parabolic $C^{\frac{1}{4}}$
		means $C^{\frac{1}{4},\frac{1}{8}}$.
		
		The specific input from Lemma~\ref{lem:lp} is
		$\|\mathcal J_y\|_{C^{\frac{1}{4},\frac{1}{8}}}\le C_T$.
		Since $J=\mathcal J^{-\frac{1}{\alpha}}$, we have
		\[
		J_y=-\frac{1}{\alpha}\mathcal J^{-\frac{1}{\alpha}-1}\mathcal J_y.
		\]
		Lemma~\ref{lem:firstholder} bounds the H\"older norm of $\mathcal J$,
		and the two-sided Jacobian bound keeps $\mathcal J$ away from zero.
		The product formula therefore gives the same uniform H\"older bound
		for $J_y$. This first-derivative modulus is the additional information
		needed to improve the bounded source in Lemma~\ref{lem:intrinsic}
		to a H\"older source.
		
		In the material coordinates, the effective velocity satisfies
		\[
		v_y=\frac{\ell \mathcal{R}}{S^2},\qquad v_t=-F_1(v-u).
		\]
		The first identity uses $r_y=\frac{\ell J}{S^2}$ and
		$\mathcal R=Jv_r$; the second is the effective-velocity equation along
		a trajectory. The established amplitude and Jacobian bounds make both
		right-hand sides uniformly bounded, so $v$ is Lipschitz in space and time.
		To recover $u$, use
		\[
		\rho^{\alpha-1}=\ell^2kJ^{1-\alpha},\qquad
		\partial_r=\frac{S^2}{\ell J}\partial_y,\qquad
		k_y=(\alpha-1)\ell\psi_0k.
		\]
		Differentiating the first identity with the other two gives
		\[
		{\partial_r \rho^{\alpha-1}}={(\alpha-1)}\ell S^2kJ^{-\alpha}(\ell\psi_0-\frac{J_y}{J})
		\]
		and hence
		\[
		v-u=\frac{\alpha}{\alpha-1}\partial_r\rho^{\alpha-1}
		=\alpha\ell S^2kJ^{-\alpha}
		\left(\ell\psi_0-\frac{J_y}{J}\right).
		\]
		Every factor on the right has a uniform parabolic H\"older norm:
		use Lemma~\ref{lem:intrinsic} for $S,k$, Lemma~\ref{lem:initialstrain}
		for the initial logarithmic derivatives, and the bounds just obtained
		for $J,J_y$. Also $\ell$ is bounded above, so this formula introduces
		no inverse scale. Thus $v-u$, then $u=v-(v-u)$ and $U=\frac{u}{r}$,
		are uniformly parabolic $C^{\frac{1}{4}}$.
		
		The pressure factors have the same modulus. Indeed, for each fixed
		$\sigma\ge0$, the initial factor is bounded and satisfies
		\[
		\partial_y\bigl[\rho_0(a_c+\ell y)^\sigma\bigr]
		=\sigma\ell\psi_0(a_c+\ell y)\rho_0(a_c+\ell y)^\sigma.
		\]
		It is therefore uniformly Lipschitz on the scaled spatial interval and independent of
		time. Since $\gamma-\alpha>0$ and $\gamma-2\alpha+1\ge0$, the formulae
		\[
		\begin{aligned}
			F_1&=\frac{\gamma}{\alpha}\rho_0^{\gamma-\alpha}J^{\alpha-\gamma},\\
			Q&=\frac{\gamma(\gamma-\alpha)}{\alpha^2}
			\rho_0^{\gamma-2\alpha+1}J^{-(\gamma-2\alpha+1)}(v-u)^2
		\end{aligned}
		\]
		give uniform parabolic $C^{\frac{1}{4}}$ norms for $F_1,Q$ by taking
		products of the already controlled factors. At the endpoint
		$\gamma=2\alpha-1$, the initial factor in $Q$ is simply $\rho_0^0=1$.
		
		It remains to establish the same regularity for $\mathcal R$, which
		appears in the source $\mathfrak f$. Recall from
		Lemma~\ref{lem:reference} that
		\[
		\mathcal Z=\mathcal R-\Phi_a(J),\qquad
		\Phi_a(J)=F_{1,0}(a)\int_1^J s^{\alpha-\gamma}\,\dd s.
		\]
		On the fixed compact interval containing all values of $J$ and $1$,
		the integral and its derivative in $J$ are bounded. The H\"older
		bounds for $F_{1,0}$ and $J$ therefore give a uniform parabolic
		H\"older bound for $\Phi_a(J)$. Equation~\eqref{eq:jac:3.6} reads
		\[
		\mathcal Z_t=(2U-F_1)\mathcal Z
		+(2U-F_1)\Phi_a(J)-2UF_1J-QJ,
		\qquad \mathcal Z(0,y)=v_{0r}(a_c+\ell y).
		\]
		All coefficients and source terms in this ordinary differential equation
		now have uniform spatial $C^{\frac{1}{4}}$ norms, and $\mathcal Z$
		is already bounded by \eqref{eq:jac:3.8} and the upper bound for $J$.
		For $y,z\in I_2$, write $\Delta_{y,z}g(t)=g(t,y)-g(t,z)$.
		In subtracting the equations at these two points, the only term
		containing an unknown difference is
		\[
		\Delta_{y,z}\bigl[(2U-F_1)\mathcal Z\bigr]
		=(2U-F_1)(t,y)\Delta_{y,z}\mathcal Z(t)
		+\Delta_{y,z}(2U-F_1)(t)\mathcal Z(t,z).
		\]
		The initial difference is bounded by $C_T|y-z|^{\frac{1}{4}}$
		by Lemma~\ref{lem:initialstrain} and the upper bound for $\ell$.
		Integration in time thus gives
		\[
		|\Delta_{y,z}\mathcal Z(t)|
		\le C_T|y-z|^{\frac{1}{4}}
		+C_T\int_0^t|\Delta_{y,z}\mathcal Z(s)|\,\dd s.
		\]
		Gronwall's inequality yields the uniform spatial H\"older bound.
		The same evolution equation gives $|\mathcal Z_t|\le C_T$, hence
		\[
		|\mathcal Z(t,y)-\mathcal Z(s,z)|
		\le C_T\left(|y-z|^{\frac{1}{4}}+|t-s|\right).
		\]
		Since $|t-s|\le T$, this implies a parabolic
		$C^{\frac{1}{4},\frac{1}{8}}$ bound for $\mathcal Z$.
		Adding $\Phi_a(J)$ gives the required bound for $\mathcal R$.
		This argument uses no estimate of $J_t$ or $u_r$.
		
		We can now check all the coefficients in \eqref{eq:bound:5.5}.
		The leading coefficient $\mathfrak a$ was controlled in
		Lemma~\ref{lem:firstholder}. The drift $\mathfrak b$ involves only
		$\mathfrak a,J,S,\frac{1}{r},\frac{1}{a},\psi_0$, while the source
		$\mathfrak f$ additionally involves $\mathcal R,U,k,\psi_0'$ and
		bounded powers of $J$. The initial functions have the uniform scaled
		H\"older bounds recorded after Lemma~\ref{lem:intrinsic}; all the
		other factors were controlled above. Products and the positive lower
		bounds for $J,S$ therefore give
		\[
		\|\mathfrak a\|_{C^{\frac{1}{4},\frac{1}{8}}}
		+\|\mathfrak b\|_{C^{\frac{1}{4},\frac{1}{8}}}
		+\|\mathfrak f\|_{C^{\frac{1}{4},\frac{1}{8}}}\le C_T.
		\]
		Together with the ellipticity in \eqref{eq:bound:5.6}, this verifies
		the coefficient hypotheses of Theorem~\ref{thm:interior-schauder}.
		On the fixed interval $[0,\tau]$, the solution is smooth, so the
		solution-regularity hypothesis also holds;
		the corresponding smooth norm is not used in the estimate. Apply
		Theorem~\ref{thm:interior-schauder} with $w=\mathcal J$,
		$\beta=\frac{1}{4}$, and
		\[
		\Omega=(\frac{t_0}{2},\tau)\times I_2,
		\qquad
		\Omega'=(t_0,\tau)\times I_1.
		\]
		The spatial gap is one and the gap from the bottom is $\frac{t_0}{2}$.
		The theorem gives, in particular,
		\[
		\begin{aligned}
			\|\mathcal J_t\|_{L^\infty(\Omega')}
			+\|\mathcal J_{yy}\|_{L^\infty(\Omega')}
			&\le C_{T,t_0}\left(\|\mathcal J\|_{L^\infty(\Omega)}
			+\|\mathfrak f\|_{C^{\frac{1}{4},\frac{1}{8}}(\overline\Omega)}\right)\\
			&\le C_{T,t_0}.
		\end{aligned}
		\]
		Continuity includes the times $t_0$ and $\tau$. The constant is
		independent of $\tau$, the center, and inverse initial density.
		No prescribed data at $t=\tau$ or solution values beyond $\tau$ are
		used; letting
		$\tau\uparrow T$ gives the bound throughout $[t_0,T)\times I_1$.
		
		To obtain the stated bound for $u_r$, recall the continuity identity
		$J_t=J(u_r+2U)$ from \eqref{eq:jac:3.3}. Since
		$\mathcal J=J^{-\alpha}$, it follows that
		\[
		\mathcal J_t=-\alpha\mathcal J\frac{J_t}{J},\qquad
		u_r=\frac{J_t}{J}-2U
		=-\frac{\mathcal J_t}{\alpha\mathcal J}-2U.
		\]
		The positive lower bound for $\mathcal J$ and the uniform bound for
		$U$ therefore imply $|u_r|\le C_{T,t_0}$ on the smaller cylinder
		$[t_0,\tau]\times I_1$.
		Every sufficiently large label can be chosen as its center $a_c$,
		where $y=0$. Since all constants are independent of this center,
		one fixed exterior threshold $X_T$ gives \eqref{eq:exteriorstrain}.
		
		For an additional physical-coordinate reconstruction, we check directly
		how the spatial scaling cancels in the second density derivative.
		First, differentiating $J=\mathcal J^{-\frac{1}{\alpha}}$ gives
		\[
		J_{yy}=\frac{\alpha+1}{\alpha^2}
		\mathcal J^{-\frac{1}{\alpha}-2}\mathcal J_y^2
		-\frac{1}{\alpha}\mathcal J^{-\frac{1}{\alpha}-1}\mathcal J_{yy}.
		\]
		Thus $J_y,J_{yy}$ are bounded on $[t_0,T)\times I_1$ by
		Lemma~\ref{lem:lp} and the Schauder estimate. Put
		\[
		L=\frac{2S_a}{S}=2\left(\frac{J}{S^2r}-\frac{1}{a}\right).
		\]
		The two factors in the formula for $\partial_r\rho^{\alpha-1}$ satisfy
		\[
		\begin{aligned}
			\partial_y(S^2kJ^{-\alpha})
			&=S^2kJ^{-\alpha}\left(\ell L+(\alpha-1)\ell\psi_0
			-\alpha\frac{J_y}{J}\right),\\
			\partial_y\left(\ell\psi_0-\frac{J_y}{J}\right)
			&=\ell^2\psi_0'-\frac{J_{yy}}{J}+\frac{J_y^2}{J^2}.
		\end{aligned}
		\]
		Applying $\partial_r=\frac{S^2}{\ell J}\partial_y$ and the product
		rule therefore yields
		\begin{equation}\label{eq:para:A6}
			\begin{aligned}
				{\partial_{rr} \rho^{\alpha-1}}={(\alpha-1)}S^4kJ^{-\alpha-1}\bigg[&
				\ell^2\psi_0'-\frac{J_{yy}}{J}+\frac{J_y^2}{J^2}\\
				&+\left(\ell L+{(\alpha-1)}\ell\psi_0-\alpha\frac{J_y}{J}\right)
				\left(\ell\psi_0-\frac{J_y}{J}\right)\bigg].
			\end{aligned}
		\end{equation}
		The factor $\ell^{-1}$ in $\partial_r$ has canceled the prefactor
		$\ell$ in the first derivative. Every remaining power of $\ell$
		is $0,1$, or $2$, and $L$ is uniformly bounded by its displayed
		formula. Hence ${\partial_{rr} \rho^{\alpha-1}}$ is bounded uniformly
		in the tail.
		The independent velocity identity gives
		\[
		u_r=v_r-\frac{\alpha}{\alpha-1} {\partial_{rr} \rho^{\alpha-1}} =\frac{\mathcal{R}}{J}-\frac{\alpha}{\alpha-1} {\partial_{rr} \rho^{\alpha-1}}.
		\]
		This gives the same bound as the continuity formula above and makes
		the cancellation of the inverse spatial scale explicit.
	\end{proof}
	
	\begin{corollary}[Squared full deformation]\label{cor:fulldeformation}
		The physical velocity satisfies
		\begin{equation}\label{eq:fulldeformation}
			\int_0^T\|\nabla\bu(t)\|_\infty^2\dd t<\infty.
		\end{equation}
	\end{corollary}
	\begin{proof}
		The radial eigenvalues of $\nabla\bu$ are $u_r,\frac{u}{r},\frac{u}{r}$.
		Choose $t_0=\frac{T}{4}$ and a fixed ball containing all physical radii outside
		which \eqref{eq:exteriorstrain} applies. Bounded exterior $u$ also bounds
		$\frac{u}{r}$ there. Since $t_0=\frac{T}{4}$ is fixed, we absorb $C_{T,t_0}$ into $C_T$.
		Hence, with a sufficiently large fixed $L$,
		\[
		\begin{split}
			\int_0^T\|\nabla\bu\|_\infty^2\dd t
			\le{}& t_0\sup_{t\le t_0}\|\nabla\bu(t)\|_\infty^2
			+\int_{t_0}^T\|\nabla\bu\|_{L^\infty(B_L)}^2\dd t\\
			&+(T-t_0)C_T^2<\infty.
		\end{split}
		\]
		The first quantity is a norm on a fixed compact subinterval of the
		existing branch. The second is bounded by Proposition~\ref{prop:coredeformation}.
	\end{proof}
	
	\section{Continuation and proof of the main theorem}\label{sec:completion}
	Corollary~\ref{cor:fulldeformation} controls the time integral of
	$\|\nabla\bu\|_\infty^2$. To apply Proposition~\ref{prop:local}, it
	remains to control the gradients of the two density powers in
	\eqref{eq:continuationcriterion}. We obtain these bounds from a
	one-sided second-derivative estimate for $\rho^{\alpha-1}$.
	
	\begin{lemma}[One-sided Glaeser inequality]\label{lem:glaeser}
		Let $H\in C^2(\R)$ be nonnegative and suppose $H''\le K$ for
		some $K\ge0$. Then $|H'|^2\le2KH$ on $\R$.
	\end{lemma}
	\begin{proof}
		The function $z\mapsto H(z)-\frac{K}{2}z^2$ is concave. Its tangent
		bound at an arbitrary $x$ gives, for every $s\in\R$,
		\[
		0\le H(x+s)\le H(x)+sH'(x)+\frac{K}{2}s^2.
		\]
		If $K>0$, choose $s=-\frac{H'(x)}{K}$ to obtain
		\[
		0\le H(x)-\frac{|H'(x)|^2}{2K}.
		\]
		If $K=0$, the affine expression
		$H(x)+sH'(x)$ can be nonnegative for every $s$ only if $H'(x)=0$.
		This proves both cases.
	\end{proof}
	
	\begin{corollary}[Integrable slopes of the density powers]\label{cor:slopes}
		On a putative finite maximal interval $[0,T)$,
		\[
		\int_0^T\left(
		\|\nabla\rho^{\frac{\alpha-1}{2}}(t)\|_\infty^2
		+\|\nabla\rho^{\frac{\alpha-1}{2}}(t)\|_\infty^4
		+\|\nabla\rho^{\frac{\gamma-1}{2}}(t)\|_\infty
		\right)\dd t<\infty.
		\]
	\end{corollary}
	\begin{proof}
		Proposition~\ref{prop:lowerjacobian} and
		Lemma~\ref{lem:localvgradient} give a global constant $B_T$ with
		$|v_r|\le B_T$. Set $A(t)=\|(u_r)_-(t)\|_\infty$. The
		effective-velocity identity yields
		\[
		\partial_{rr}\rho^{\alpha-1}
		=\frac{\alpha-1}{\alpha}(v_r-u_r)
		\le\frac{\alpha-1}{\alpha}\bigl(B_T+A(t)\bigr).
		\]
		Apply Lemma~\ref{lem:glaeser} at each time to the smooth even extension
		of $\rho^{\alpha-1}(t,r)$ across $r=0$. Since $\rho^{\alpha-1}>0$,
		\[
		\left|\partial_r\rho^{\frac{\alpha-1}{2}}\right|^2
		=\frac{|\partial_r\rho^{\alpha-1}|^2}{4\rho^{\alpha-1}}
		\le\frac{\alpha-1}{2\alpha}\bigl(B_T+A(t)\bigr).
		\]
		The estimate includes the center by parity. For a radial scalar
		field, the magnitude of its Cartesian gradient equals the magnitude
		of its radial derivative away from the center, so taking the spatial
		supremum and using $A(t)\le\|\nabla\bu(t)\|_\infty$ gives
		\[
		\|\nabla\rho^{\frac{\alpha-1}{2}}(t)\|_\infty^2
		\le\frac{\alpha-1}{2\alpha}
		\bigl(B_T+\|\nabla\bu(t)\|_\infty\bigr).
		\]
		Squaring and integrating, we obtain
		\[
		\int_0^T\|\nabla\rho^{\frac{\alpha-1}{2}}(t)\|_\infty^4\dd t
		\le C_\alpha\left(
		TB_T^2+\int_0^T\|\nabla\bu(t)\|_\infty^2\dd t
		\right)<\infty
		\]
		by Corollary~\ref{cor:fulldeformation}. H\"older's inequality also gives
		\[
		\int_0^T\|\nabla\rho^{\frac{\alpha-1}{2}}(t)\|_\infty^2\dd t
		\le T^{\frac{1}{2}}\left(
		\int_0^T\|\nabla\rho^{\frac{\alpha-1}{2}}(t)\|_\infty^4\dd t
		\right)^{\frac{1}{2}}<\infty.
		\]
		Finally, \eqref{eq:range} gives $\frac{\gamma-1}{\alpha-1}\ge2$.
		The chain rule yields
		\[
		\begin{aligned}
			\nabla\rho^{\frac{\gamma-1}{2}}
			&=\frac{\gamma-1}{\alpha-1}\rho^{\frac{\gamma-\alpha}{2}}
			\nabla\rho^{\frac{\alpha-1}{2}},\\
			\|\nabla\rho^{\frac{\gamma-1}{2}}(t)\|_\infty
			&\le\frac{\gamma-1}{\alpha-1}\|\rho(t)\|_\infty^{\frac{\gamma-\alpha}{2}}
			\|\nabla\rho^{\frac{\alpha-1}{2}}(t)\|_\infty
			\le C_T\|\nabla\rho^{\frac{\alpha-1}{2}}(t)\|_\infty.
		\end{aligned}
		\]
		Here the last inequality uses Proposition~\ref{prop:upper}.
		Integrating and applying Cauchy--Schwarz gives
		\[
		\int_0^T\|\nabla\rho^{\frac{\gamma-1}{2}}(t)\|_\infty\dd t
		\le C_T T^{\frac{1}{2}}\left(
		\int_0^T\|\nabla\rho^{\frac{\alpha-1}{2}}(t)\|_\infty^2\dd t
		\right)^{\frac{1}{2}}<\infty.
		\]
		This proves all three integrability assertions.
	\end{proof}
	
	\begin{proof}[Proof of Theorem~\ref{thm:main}]
		Proposition~\ref{prop:local} gives a unique local radial solution in
		the stated homogeneous class, smooth on finite spacetime cylinders.
		Let $[0,T)$ be its maximal interval and suppose $T<\infty$.
		
		Sections~\ref{sec:upper}--\ref{sec:barriers} give the upper density
		bound and a positive lower bound on every fixed ball.
		Section~\ref{sec:core} gives compact squared deformation and global
		amplitude bounds. Section~\ref{sec:variation} gives vanishing exterior
		variation tails, which select a material boundary with small time
		variation in Section~\ref{sec:jacobian}. The resulting two-sided
		Jacobian bounds allow the parabolic estimates of
		Section~\ref{sec:parabolic} to control the exterior strain.
		Corollary~\ref{cor:fulldeformation} then supplies the global squared
		deformation bound, and Corollary~\ref{cor:slopes} controls the
		density-power gradients. Together with Proposition~\ref{prop:upper},
		these estimates yield
		\begin{equation}\label{eq:hom:9.1}
			\begin{gathered}
				\sup_{0\le t<T}\|\rho(t)\|_\infty<\infty,\\
				\int_0^T\left(
				1+\|\nabla\bu(t)\|_\infty^2
				+\|\nabla\rho^{\frac{\alpha-1}{2}}(t)\|_\infty^2
				+\|\nabla\rho^{\frac{\gamma-1}{2}}(t)\|_\infty
				\right)\dd t<\infty.
			\end{gathered}
		\end{equation}
		These are the hypotheses of
		\eqref{eq:continuationcriterion}. Proposition~\ref{prop:local}
		extends the solution beyond $T$, contradicting maximality.
		Thus the solution exists for all time.
		
		On every finite interval, the solution class, uniqueness, mass
		conservation, $L^1$ continuity of the density, uniform far-field
		decay, and smoothness on finite cylinders are supplied by
		Proposition~\ref{prop:local} and its continuation assertion.
		Proposition~\ref{prop:upper} and Corollary~\ref{cor:compactpositive}
		give the upper and compact lower density bounds in
		\eqref{eq:densityconclusion}. These establish all conclusions of
		Theorem~\ref{thm:main}.
	\end{proof}
	
	\section*{Acknowledgments}
	X Huang is partially supported by Chinese Academy of Sciences Project for Young Scientists in Basic Research (Grant No. YSBR-031), National Natural Science Foundation of China (Grant Nos. 12494542, 11688101). B Li is supported in part by National Key R\&D Program of China (No.~2024YF A 1013303).
	
	\par\medskip
	\noindent\textbf{Data availability statement.}
	Data sharing is not applicable to this article.
	\par
	
	\medskip
	\noindent\textbf{Conflict of interest.}
	The authors declare that they have no conflict of interest.
	\par

\end{document}